\documentclass[11pt]{article}

\usepackage{amsmath,amsthm, amssymb}
\usepackage{mathtools}
\usepackage{xcolor}
\usepackage{comment}
\usepackage{xspace}
\usepackage{tikz}
\usepackage{graphicx}
\usepackage{subcaption}
\usepackage{multirow}
\usepackage{authblk}
\usepackage{fullpage}

\newcommand{\pack}[1]{\underline{\mathbb{#1}}}
\newcommand{\rib}[1]{\mathbf{#1}}
\newcommand{\ar}[1]{\mathbb{#1}}
\newcommand{\vp}[1]{\widehat{\mathbb{#1}}}

\newcommand{\V}{\mathcal{V}}
\newcommand{\B}{\mathcal{B}}
\newcommand{\bw}{\mathbf{w}}
\newcommand{\mvW}{\mathbf{W}}
\newcommand{\al}{\alpha}
\newcommand{\be}{\beta}
\newcommand{\ga}{\gamma}
\newcommand{\Q}{\hat{Q}}
\newcommand{\Z}{\hat{Z}}
\newcommand{\mvZ}{\tilde{\mathbf{Z}}}
\newcommand{\mvQ}{\mathbf{Q}}
\newcommand{\tQ}{\tilde{\mathbf{Q}}}
\newcommand{\dZ}{\dot{Z}}

\newcommand{\bphi}{\boldsymbol{\varphi}}

\DeclareMathOperator{\ba}{\backslash}
\DeclareMathOperator{\con}{/}
\DeclareMathOperator{\mba}{\reflectbox{$\angle$}}
\DeclareMathOperator{\mcon}{\angle}
\DeclareMathOperator{\pcon}{\rightthreetimes}
\DeclareMathOperator{\ot}{\otimes}
\newcommand{\x}{\times}
\newcommand{\pl}{\parallel}
\newcommand{\eq}{=}

\newtheorem{theorem}{Theorem}[section]
\newtheorem{corollary}[theorem]{Corollary}
\newtheorem{lemma}[theorem]{Lemma}
\newtheorem{proposition}[theorem]{Proposition}
\theoremstyle{definition}
\newtheorem{convention}{Convention}
\newtheorem{technique}{Technique}
\newtheorem{definition}{Definition}
\theoremstyle{remark}
\newtheorem*{remark}{Remark}

\title{Tensor product formulas for the Bollob\'as--Riordan and Krushkal polynomials}

\date{}

\author{Iain Moffatt\footnote{Department of Mathematics, Royal Holloway, University of London, Egham, TW20 0EX, United Kingdom; iain.moffatt@rhul.ac.uk.} and Maya Thompson\footnote{mayathompson.math@gmail.com.}}

\begin{document}
\maketitle

\begin{abstract}
Brylawski's tensor product formula expresses the Tutte polynomial of the tensor product of two graphs in terms of Tutte polynomials arising from the tensor factors. Analogous tensor product formulas are known for the ribbon graph polynomial and transition polynomials of graphs embedded in surfaces, as well as for the Bollob\'as--Riordan polynomial in some special cases. We define  the tensor product of graphs embedded in pseudo-surfaces and use this to generalize and unify all of the above results, providing Brylawski-style formulas for both the Bollob\'as--Riordan and  Krushkal polynomials. 
\end{abstract}

\section{Introduction}
Let $G$ and $H$ be graphs, $G$ loopless, and $e$ be a non-loop edge of $H$. A \emph{tensor product} $G\ot H$ is a graph obtained by identifying each edge in $G$ with the edge $e$ in its own copy of $H$, then deleting each of the identified edges. In other words, a tensor product is obtained by, for every edge in $G$, taking the 2-sum of $G$ and a copy of $H$.
In~\cite{MR863010}, Brylawski expressed the Tutte polynomial of a tensor product $G\ot H$ in terms of polynomials arising from $G$ and $H$:
\begin{equation}\label{btf}
T(G\ot H; x,y)=\phi^{n(G)}\psi^{r(G)}T\Big(G;\frac{T(H\ba e;x,y)}{\psi},\frac{T(H/e;x,y)}{\phi}\Big),
\end{equation}
where $\phi$ and $\psi$ are the unique solutions to the system of equations
\begin{align*}
    (x-1)\phi+\psi&=T(H\ba e),\\
   \phi+ (y-1)\psi& = T(H\con e).
\end{align*}
This result is known as \emph{Brylawski's tensor product formula}. It is  a useful computational aid (see the survey~\cite{merinohandbook}) and is significant in the theory of the Tutte polynomial, for example playing  a crucial role in Jaeger, Vertigan, and Welsh's seminal work~\cite{MR1049758} on the computational complexity of the Tutte polynomial. We note that Brylawski proved his result for matroids, although we only consider its graph version here.

We are interested in extending  Brylawski's tensor product formula to analogues of the Tutte polynomial for embedded graphs, or equivalently ribbon graphs. This problem was first considered by Huggett and Moffatt in~\cite{MR2854787}. Their paper offered three results for computing the Bollob\'as--Riordan polynomial~\cite{bollobasriordanpoly,zbMATH01801590}, $R(\rib{G}\ot \rib{H};x,y,z)$, of a tensor product $\rib{G}\ot \rib{H}$ of two ribbon graphs $\rib{G}$ and $\rib{H}$. However, these results are either partial results, or require compromises that moves them away from the form of~\eqref{btf}.
 For example, Theorem~4.3 and Corollary~4.4 of~\cite{MR2854787} require that $\rib{H}$ be plane; while Theorems~5.6 and~5.7 of~\cite{MR2854787} require the use of a multivariate extension of the Bollob\'as--Riordan polynomial and, for the latter, $\rib{G}$ is replaced by a more complicated ribbon graph. Similarly, in~\cite{EllisMoff} Ellis--Monaghan and Moffatt gave a formula for $R(\rib{G}\otimes \rib{H};x,y,z)$ that only holds when $(x-1)yz^2=1$. Thus this earlier work does not satisfactorily extend Brylawski's tensor product formula to the Bollob\'as--Riordan polynomial. Furthermore, we are not aware of any prior work on extending Brylawski's tensor product formula to the Krushkal polynomial~\cite{krushkalpoly}, $K(G\subset \Sigma ; x,y,a,b )$ of a graph embedded in a surface. (Tensor products  are defined in Section~\ref{s.bg}, the Bollob\'as--Riordan polynomial in Subsection~\ref{ss:brp}, and the Krushkal polynomial in Subsection~\ref{ss:kp}.)

The reason the tensor product formulas in~\cite{EllisMoff} and~\cite{MR2854787} required these compromises is that the Bollob\'as--Riordan polynomial has no known deletion-contraction relations that applies to every possible type of edge in a ribbon graph. Here we resolve this difficulty by taking a different approach. 
 Following the work of~\cite{HM, KMT}, we consider  ribbon graphs whose vertices and boundary components are coloured (or, equivalently, with graphs that are not-necessarily cellularly embedded in pseudo-surfaces). We realize these as ``packaged arrow presentations'' below. 

This language allows us to  complete the definition of the 2-sum and tensor product for graphs embedded in surfaces. (Previously, the 2-sum along two loops had not been defined.) Furthermore, this more general setting allows us to make use of  deletion-contraction constructions of the Bollob\'as--Riordan polynomial and  Krushkal polynomial, to obtain both 2-sum and tensor product formula for these polynomials. We note such formulas were first considered in the second-named author's PhD Thesis~\cite{thesis}, and the results presented here develops and refines that work.

\medskip

This paper is structured as follows. In Section~\ref{s.bg} we review arrow presentations (which describe graphs cellularly embedded in surfaces) and packaged arrow presentations (which describe graphs embedded in pseudo-surfaces), and define 2-sums and tensor products for them.  
In Section~\ref{s.tpf} we introduce a transition polynomial for packaged arrow presentations, provide 2-sum and tensor product formulas for this polynomial, and discuss how this relates to the Krushkal polynomial.
In Section~\ref{s:cases} we show how our general result specialises to give a tensor product formula for the Bollob\'as--Riordan polynomial and show that known formula for the transition polynomial, ribbon graph polynomial and Tutte polynomial can be recovered from the general result.

\section{Arrow presentations and packaged arrow presentations}\label{s.bg}

\subsection{Arrow presentations}\label{ss:ap}
Our interest is in \emph{topological Tutte polynomials} which are analogues of the Tutte polynomial for  graphs embedded in  surfaces. 
In the context of  topological Tutte polynomials, graphs (cellularly) embedded in closed surfaces are usually described using ribbon graphs. Here, however, we find it most natural to describe them using Chmutov's arrow presentations~\cite{MR2507944}. In this section we review arrow presentations. Additional background on arrow presentations, ribbon graphs and embedded graphs can be found in the book~\cite{graphsonsurfaces}.

An {\em arrow presentation} $\ar{G}$ consists of a set of circles (i.e., closed 1-manifolds), a set of arrows lying disjointly on the circles, and a set of labels. Each label is assigned to exactly two arrows and each arrow has a label.  The set of labels is called the \emph{edge set} and its elements are \emph{edges}.
 Thus an edge is identified with a pair of arrows, and we say that this pair of arrows \emph{constitutes} the edge. 
 We shall use $V(\ar{G})$ to denote the vertex set of $\ar{G}$, and $E(\ar{G})$ its edge set. 
 A vertex is \emph{incident} with an edge $e$ if there is an $e$-labelled arrow lying on it, and two vertices are \emph{adjacent} if they are both incident to the same edge. (These terms correspond with the standard graph terms.) A vertex is \emph{isolated} if it is not incident with any edge.
  At times we will express individual arrows in an arrow presentation in the form $\overrightarrow{pq}$ meaning that the arrow lies on the directed arc of a  vertex (i.e., a circle) of the arrow presentation from a  point $p$ to a point $q$. 
 See Figure~\ref{Arrp} for an example of an arrow presentation. It has three edges and two vertices.

Two arrow presentation are \emph{equivalent} if one can be obtained from the other by a diffeomorphism of the vertices that sends arrows to arrows, reversing the directions of both $e$-labelled arrows for some (potentially empty) subset of edges $e$, and a bijection between the edge sets (where the labelling of the arrows respects the bijection). Here we consider arrow presentations up to equivalence. For example, the arrow presentations shown in Figures~\ref{Arrp} and~\ref{rgarrp.2} are equivalent.

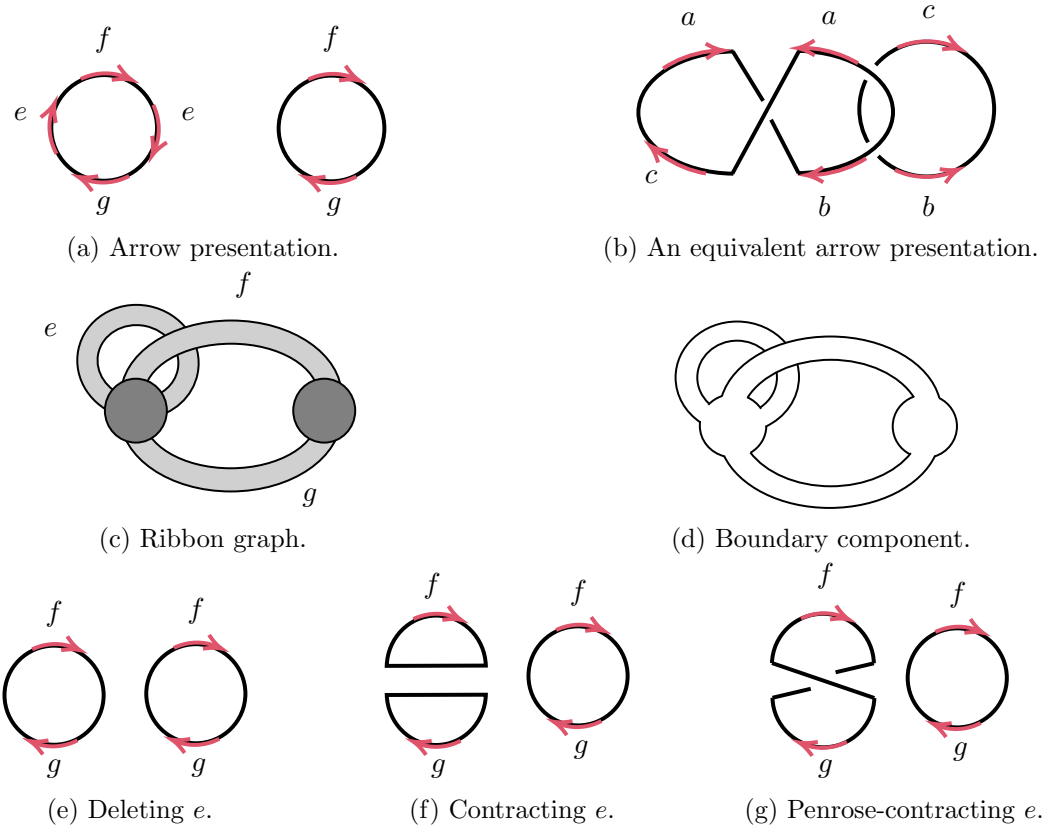
\begin{figure}[t!]
    \centering
      \begin{subfigure}[b]{.49\linewidth}
    \centering
    \tikzset{every picture/.style={line width=0.75pt}} %set default line width to 0.75pt        

\begin{tikzpicture}[x=0.8pt,y=0.8pt,yscale=-1,xscale=1]

%Shape: Circle [id:dp18731970584179414] 
\draw  [line width=1.5]  (252.39,77.84) .. controls (252.39,64.03) and (263.58,52.84) .. (277.39,52.84) .. controls (291.2,52.84) and (302.39,64.03) .. (302.39,77.84) .. controls (302.39,91.64) and (291.2,102.84) .. (277.39,102.84) .. controls (263.58,102.84) and (252.39,91.64) .. (252.39,77.84) -- cycle ;
%Shape: Arc [id:dp2631701798224957] 
\draw  [draw opacity=0][line width=1.5]  (266.36,54.7) .. controls (272,51.98) and (278.59,51.19) .. (285.01,52.97) .. controls (286.9,53.49) and (288.67,54.2) .. (290.32,55.08) -- (277.39,77.84) -- cycle ; \draw [color={rgb, 255:red, 223; green, 83; blue, 107 }  ,draw opacity=1 ][line width=1.5]    (266.36,54.7) .. controls (272,51.98) and (278.59,51.19) .. (285.01,52.97) .. controls (285.9,53.21) and (286.76,53.5) .. (287.6,53.83) ; \draw [shift={(290.32,55.08)}, rotate = 203.76] [color={rgb, 255:red, 223; green, 83; blue, 107 }  ,draw opacity=1 ][line width=1.5]    (8.53,-3.82) .. controls (5.42,-1.79) and (2.58,-0.52) .. (0,0) .. controls (2.58,0.52) and (5.42,1.79) .. (8.53,3.82)   ; 
%Shape: Arc [id:dp8740481682631627] 
\draw  [draw opacity=0][line width=1.5]  (289.18,100.6) .. controls (283.63,103.51) and (277.07,104.51) .. (270.59,102.94) .. controls (268.69,102.49) and (266.9,101.83) .. (265.23,101.01) -- (277.39,77.84) -- cycle ; \draw [color={rgb, 255:red, 223; green, 83; blue, 107 }  ,draw opacity=1 ][line width=1.5]    (289.18,100.6) .. controls (283.63,103.51) and (277.07,104.51) .. (270.59,102.94) .. controls (269.7,102.73) and (268.83,102.47) .. (267.99,102.17) ; \draw [shift={(265.23,101.01)}, rotate = 21.87] [color={rgb, 255:red, 223; green, 83; blue, 107 }  ,draw opacity=1 ][line width=1.5]    (8.53,-3.82) .. controls (5.42,-1.79) and (2.58,-0.52) .. (0,0) .. controls (2.58,0.52) and (5.42,1.79) .. (8.53,3.82)   ; 
%Shape: Arc [id:dp04238658329545686] 
\draw  [draw opacity=0][line width=1.5]  (300.61,66.97) .. controls (303.29,72.63) and (304.03,79.23) .. (302.21,85.64) .. controls (301.67,87.52) and (300.95,89.28) .. (300.05,90.92) -- (277.39,77.84) -- cycle ; \draw [color={rgb, 255:red, 223; green, 83; blue, 107 }  ,draw opacity=1 ][line width=1.5]    (300.61,66.97) .. controls (303.29,72.63) and (304.03,79.23) .. (302.21,85.64) .. controls (301.96,86.52) and (301.66,87.38) .. (301.33,88.21) ; \draw [shift={(300.05,90.92)}, rotate = 294.17] [color={rgb, 255:red, 223; green, 83; blue, 107 }  ,draw opacity=1 ][line width=1.5]    (8.53,-3.82) .. controls (5.42,-1.79) and (2.58,-0.52) .. (0,0) .. controls (2.58,0.52) and (5.42,1.79) .. (8.53,3.82)   ; 
%Shape: Circle [id:dp5067213288684962] 
\draw  [line width=1.5]  (359.89,78.09) .. controls (359.89,64.28) and (371.08,53.09) .. (384.89,53.09) .. controls (398.7,53.09) and (409.89,64.28) .. (409.89,78.09) .. controls (409.89,91.89) and (398.7,103.09) .. (384.89,103.09) .. controls (371.08,103.09) and (359.89,91.89) .. (359.89,78.09) -- cycle ;
%Shape: Arc [id:dp5192035217140911] 
\draw  [draw opacity=0][line width=1.5]  (373.86,54.95) .. controls (379.5,52.23) and (386.09,51.44) .. (392.51,53.22) .. controls (394.4,53.74) and (396.17,54.45) .. (397.82,55.33) -- (384.89,78.09) -- cycle ; \draw [color={rgb, 255:red, 223; green, 83; blue, 107 }  ,draw opacity=1 ][line width=1.5]    (373.86,54.95) .. controls (379.5,52.23) and (386.09,51.44) .. (392.51,53.22) .. controls (393.4,53.46) and (394.26,53.75) .. (395.1,54.08) ; \draw [shift={(397.82,55.33)}, rotate = 203.76] [color={rgb, 255:red, 223; green, 83; blue, 107 }  ,draw opacity=1 ][line width=1.5]    (8.53,-3.82) .. controls (5.42,-1.79) and (2.58,-0.52) .. (0,0) .. controls (2.58,0.52) and (5.42,1.79) .. (8.53,3.82)   ; 
%Shape: Arc [id:dp6547715229933235] 
\draw  [draw opacity=0][line width=1.5]  (396.68,100.85) .. controls (391.13,103.76) and (384.57,104.76) .. (378.09,103.19) .. controls (376.19,102.74) and (374.4,102.08) .. (372.73,101.26) -- (384.89,78.09) -- cycle ; \draw [color={rgb, 255:red, 223; green, 83; blue, 107 }  ,draw opacity=1 ][line width=1.5]    (396.68,100.85) .. controls (391.13,103.76) and (384.57,104.76) .. (378.09,103.19) .. controls (377.2,102.98) and (376.33,102.72) .. (375.49,102.42) ; \draw [shift={(372.73,101.26)}, rotate = 21.87] [color={rgb, 255:red, 223; green, 83; blue, 107 }  ,draw opacity=1 ][line width=1.5]    (8.53,-3.82) .. controls (5.42,-1.79) and (2.58,-0.52) .. (0,0) .. controls (2.58,0.52) and (5.42,1.79) .. (8.53,3.82)   ; 
%Shape: Arc [id:dp01554497382030795] 
\draw  [draw opacity=0][line width=1.5]  (254.88,90.1) .. controls (251.86,84.61) and (250.72,78.07) .. (252.15,71.57) .. controls (252.57,69.66) and (253.18,67.85) .. (253.97,66.16) -- (277.39,77.84) -- cycle ; \draw [color={rgb, 255:red, 223; green, 83; blue, 107 }  ,draw opacity=1 ][line width=1.5]    (254.88,90.1) .. controls (251.86,84.61) and (250.72,78.07) .. (252.15,71.57) .. controls (252.34,70.67) and (252.59,69.79) .. (252.87,68.94) ; \draw [shift={(253.97,66.16)}, rotate = 110.67] [color={rgb, 255:red, 223; green, 83; blue, 107 }  ,draw opacity=1 ][line width=1.5]    (8.53,-3.82) .. controls (5.42,-1.79) and (2.58,-0.52) .. (0,0) .. controls (2.58,0.52) and (5.42,1.79) .. (8.53,3.82)   ; 

% Text Node
\draw (233.25,66.65) node [anchor=north west][inner sep=0.75pt]  [font=\normalsize]  {$e$};
% Text Node
\draw (312.5,66.65) node [anchor=north west][inner sep=0.75pt]  [font=\normalsize]  {$e$};
% Text Node
\draw (271.5,28.65) node [anchor=north west][inner sep=0.75pt]  [font=\normalsize]  {$f$};
% Text Node
\draw (272.25,108.9) node [anchor=north west][inner sep=0.75pt]  [font=\normalsize]  {$g$};
% Text Node
\draw (379.25,28.15) node [anchor=north west][inner sep=0.75pt]  [font=\normalsize]  {$f$};
% Text Node
\draw (381.25,109.4) node [anchor=north west][inner sep=0.75pt]  [font=\normalsize]  {$g$};

\end{tikzpicture}
    \caption{Arrow presentation.}
    \label{Arrp}
    \end{subfigure}
    %%%
     \begin{subfigure}[b]{.49\linewidth}
    \centering
  \tikzset{every picture/.style={line width=0.75pt}} %set default line width to 0.75pt        

\begin{tikzpicture}[x=0.75pt,y=0.75pt,yscale=-1,xscale=1]

%Shape: Arc [id:dp00632191815534866] 
\draw  [draw opacity=0][line width=1.5]  (362.26,60.78) .. controls (368.5,53.57) and (377.71,49) .. (388,49) .. controls (406.78,49) and (422,64.22) .. (422,83) .. controls (422,101.78) and (406.78,117) .. (388,117) .. controls (378.42,117) and (369.76,113.04) .. (363.58,106.66) -- (388,83) -- cycle ; \draw  [line width=1.5]  (362.26,60.78) .. controls (368.5,53.57) and (377.71,49) .. (388,49) .. controls (406.78,49) and (422,64.22) .. (422,83) .. controls (422,101.78) and (406.78,117) .. (388,117) .. controls (378.42,117) and (369.76,113.04) .. (363.58,106.66) ;  
%Shape: Arc [id:dp2970440007383096] 
\draw  [draw opacity=0][line width=1.5]  (290.46,115.6) .. controls (263.94,115.31) and (242.58,101.61) .. (242.58,84.76) .. controls (242.58,67.87) and (264.02,54.15) .. (290.61,53.92) -- (291.3,84.76) -- cycle ; \draw  [line width=1.5]  (290.46,115.6) .. controls (263.94,115.31) and (242.58,101.61) .. (242.58,84.76) .. controls (242.58,67.87) and (264.02,54.15) .. (290.61,53.92) ;  
%Shape: Arc [id:dp22462144262771933] 
\draw  [draw opacity=0][line width=1.5]  (323.11,53.62) .. controls (349.64,53.9) and (371,67.6) .. (371,84.46) .. controls (371,101.35) and (349.56,115.06) .. (322.96,115.3) -- (322.27,84.46) -- cycle ; \draw  [line width=1.5]  (323.11,53.62) .. controls (349.64,53.9) and (371,67.6) .. (371,84.46) .. controls (371,101.35) and (349.56,115.06) .. (322.96,115.3) ;  
%Straight Lines [id:da14277759987261085] 
\draw [line width=1.5]    (323.61,53.63) -- (289.97,115.59) ;
%Straight Lines [id:da06913257238825743] 
\draw [line width=1.5]    (290.21,53.92) -- (305.5,78.85) ;
%Straight Lines [id:da4911338428598149] 
\draw [line width=1.5]    (309.5,88.35) -- (323.37,115.29) ;
%Shape: Arc [id:dp7595943371949999] 
\draw  [draw opacity=0][line width=1.5]  (373.2,51.96) .. controls (380.74,48.62) and (389.5,47.76) .. (398.08,50.12) .. controls (400.57,50.81) and (402.92,51.74) .. (405.11,52.87) -- (388,83) -- cycle ; \draw [color={rgb, 255:red, 223; green, 83; blue, 107 }  ,draw opacity=1 ][line width=1.5]    (373.2,51.96) .. controls (380.74,48.62) and (389.5,47.76) .. (398.08,50.12) .. controls (399.57,50.54) and (401.02,51.04) .. (402.41,51.62) ; \draw [shift={(405.11,52.87)}, rotate = 204.24] [color={rgb, 255:red, 223; green, 83; blue, 107 }  ,draw opacity=1 ][line width=1.5]    (8.53,-3.82) .. controls (5.42,-1.79) and (2.58,-0.52) .. (0,0) .. controls (2.58,0.52) and (5.42,1.79) .. (8.53,3.82)   ; 
%Shape: Arc [id:dp7803286389580906] 
\draw  [draw opacity=0][line width=1.5]  (403.88,113.5) .. controls (396.47,117.11) and (387.74,118.28) .. (379.09,116.21) .. controls (376.57,115.61) and (374.19,114.77) .. (371.96,113.71) -- (388,83) -- cycle ; \draw [color={rgb, 255:red, 223; green, 83; blue, 107 }  ,draw opacity=1 ][line width=1.5]    (401.04,114.74) .. controls (394.29,117.35) and (386.66,118.02) .. (379.09,116.21) .. controls (376.57,115.61) and (374.19,114.77) .. (371.96,113.71) ;  \draw [shift={(403.88,113.5)}, rotate = 157.33] [color={rgb, 255:red, 223; green, 83; blue, 107 }  ,draw opacity=1 ][line width=1.5]    (8.53,-3.82) .. controls (5.42,-1.79) and (2.58,-0.52) .. (0,0) .. controls (2.58,0.52) and (5.42,1.79) .. (8.53,3.82)   ;
%Shape: Arc [id:dp7404335675478927] 
\draw  [draw opacity=0][line width=1.5]  (255.11,62.5) .. controls (261.47,58.16) and (269.37,55.06) .. (278.11,53.88) .. controls (280.5,53.56) and (282.87,53.4) .. (285.18,53.38) -- (281.51,85.77) -- cycle ; \draw [color={rgb, 255:red, 223; green, 83; blue, 107 }  ,draw opacity=1 ][line width=1.5]    (255.11,62.5) .. controls (261.47,58.16) and (269.37,55.06) .. (278.11,53.88) .. controls (279.49,53.7) and (280.85,53.57) .. (282.2,53.48) ; \draw [shift={(285.18,53.38)}, rotate = 177.52] [color={rgb, 255:red, 223; green, 83; blue, 107 }  ,draw opacity=1 ][line width=1.5]    (8.53,-3.82) .. controls (5.42,-1.79) and (2.58,-0.52) .. (0,0) .. controls (2.58,0.52) and (5.42,1.79) .. (8.53,3.82)   ; 
%Shape: Arc [id:dp3140714482917507] 
\draw  [draw opacity=0][line width=1.5]  (276.68,115.5) .. controls (269.07,114.3) and (261.2,111.13) .. (254.03,106) .. controls (252.07,104.59) and (250.24,103.09) .. (248.54,101.52) -- (273.43,80.46) -- cycle ; \draw [color={rgb, 255:red, 223; green, 83; blue, 107 }  ,draw opacity=1 ][line width=1.5]    (276.68,115.5) .. controls (269.07,114.3) and (261.2,111.13) .. (254.03,106) .. controls (252.9,105.19) and (251.82,104.35) .. (250.78,103.48) ; \draw [shift={(248.54,101.52)}, rotate = 40.82] [color={rgb, 255:red, 223; green, 83; blue, 107 }  ,draw opacity=1 ][line width=1.5]    (8.53,-3.82) .. controls (5.42,-1.79) and (2.58,-0.52) .. (0,0) .. controls (2.58,0.52) and (5.42,1.79) .. (8.53,3.82)   ; 
%Shape: Arc [id:dp9429730143892288] 
\draw  [draw opacity=0][line width=1.5]  (322.9,52.22) .. controls (330.6,51.78) and (338.96,53.21) .. (347.05,56.71) .. controls (349.27,57.66) and (351.38,58.74) .. (353.37,59.92) -- (333.51,85.77) -- cycle ; \draw [color={rgb, 255:red, 223; green, 83; blue, 107 }  ,draw opacity=1 ][line width=1.5]    (325.94,52.14) .. controls (332.77,52.19) and (340.01,53.66) .. (347.05,56.71) .. controls (349.27,57.66) and (351.38,58.74) .. (353.37,59.92) ;  \draw [shift={(322.9,52.22)}, rotate = 359.28] [color={rgb, 255:red, 223; green, 83; blue, 107 }  ,draw opacity=1 ][line width=1.5]    (8.53,-3.82) .. controls (5.42,-1.79) and (2.58,-0.52) .. (0,0) .. controls (2.58,0.52) and (5.42,1.79) .. (8.53,3.82)   ;
%Shape: Arc [id:dp0807428698211925] 
\draw  [draw opacity=0][line width=1.5]  (357.44,107.8) .. controls (350.99,112.02) and (343.03,114.96) .. (334.27,115.97) .. controls (331.87,116.24) and (329.51,116.36) .. (327.19,116.33) -- (331.51,84.02) -- cycle ; \draw [color={rgb, 255:red, 223; green, 83; blue, 107 }  ,draw opacity=1 ][line width=1.5]    (357.44,107.8) .. controls (350.99,112.02) and (343.03,114.96) .. (334.27,115.97) .. controls (332.89,116.12) and (331.52,116.23) .. (330.17,116.29) ; \draw [shift={(327.19,116.33)}, rotate = 358.65] [color={rgb, 255:red, 223; green, 83; blue, 107 }  ,draw opacity=1 ][line width=1.5]    (8.53,-3.82) .. controls (5.42,-1.79) and (2.58,-0.52) .. (0,0) .. controls (2.58,0.52) and (5.42,1.79) .. (8.53,3.82)   ; 
%Shape: Arc [id:dp5291826608780413] 
\draw  [draw opacity=0][line width=1.5]  (358.42,99.78) .. controls (355.61,94.83) and (354,89.1) .. (354,83) .. controls (354,77.77) and (355.18,72.81) .. (357.29,68.38) -- (388,83) -- cycle ; \draw  [line width=1.5]  (358.42,99.78) .. controls (355.61,94.83) and (354,89.1) .. (354,83) .. controls (354,77.77) and (355.18,72.81) .. (357.29,68.38) ;  

% Text Node
\draw (262.5,32.9) node [anchor=north west][inner sep=0.75pt]  [font=\normalsize]  {$a$};
% Text Node
\draw (333.5,32.4) node [anchor=north west][inner sep=0.75pt]  [font=\normalsize]  {$a$};
% Text Node
\draw (331.75,125.15) node [anchor=north west][inner sep=0.75pt]  [font=\normalsize]  {$b$};
% Text Node
\draw (384.25,125.15) node [anchor=north west][inner sep=0.75pt]  [font=\normalsize]  {$b$};
% Text Node
\draw (244.25,112.4) node [anchor=north west][inner sep=0.75pt]  [font=\normalsize]  {$c$};
% Text Node
\draw (384.25,29.4) node [anchor=north west][inner sep=0.75pt]  [font=\normalsize]  {$c$};

\end{tikzpicture}
    \caption{An equivalent arrow presentation.}
        \label{rgarrp.2}
    \end{subfigure}
%%%%
    \begin{subfigure}[b]{.49\linewidth}
    \centering
    \tikzset{every picture/.style={line width=0.75pt}} %set default line width to 0.75pt        

\begin{tikzpicture}[x=0.95pt,y=0.95pt,yscale=-1,xscale=1]

%Shape: Block Arc [id:dp8671598870031405] 
\draw  [color={rgb, 255:red, 0; green, 0; blue, 0 }  ,draw opacity=1 ][fill={rgb, 255:red, 208; green, 208; blue, 208 }  ,fill opacity=1 ][line width=0.75]  (202.56,102) .. controls (190.99,98.42) and (183.64,87.26) .. (185.84,75.86) .. controls (188.21,63.58) and (200.78,55.69) .. (213.91,58.23) .. controls (227.05,60.77) and (235.77,72.78) .. (233.39,85.05) .. controls (231.57,94.52) and (223.68,101.38) .. (214.15,102.86) -- (212.54,94.93) .. controls (218.97,94.08) and (224.28,89.67) .. (225.47,83.52) .. controls (226.99,75.62) and (221.13,67.85) .. (212.38,66.16) .. controls (203.63,64.46) and (195.29,69.5) .. (193.77,77.39) .. controls (192.35,84.74) and (197.32,91.99) .. (205.07,94.32) -- cycle ;
%Shape: Block Arc [id:dp12617574321961245] 
\draw  [color={rgb, 255:red, 0; green, 0; blue, 0 }  ,draw opacity=1 ][fill={rgb, 255:red, 208; green, 208; blue, 208 }  ,fill opacity=1 ][line width=0.75]  (290.53,100.94) .. controls (290.55,101.29) and (290.56,101.64) .. (290.56,102) .. controls (290.56,118.63) and (270.86,132.12) .. (246.56,132.12) .. controls (222.26,132.12) and (202.56,118.63) .. (202.56,102) .. controls (202.56,101.85) and (202.56,101.7) .. (202.56,101.55) -- (211.81,101.64) .. controls (211.81,101.76) and (211.8,101.88) .. (211.8,102) .. controls (211.8,113.53) and (227.36,122.88) .. (246.56,122.88) .. controls (265.75,122.88) and (281.31,113.53) .. (281.31,102) .. controls (281.31,101.72) and (281.31,101.44) .. (281.29,101.17) -- cycle ;
%Shape: Block Arc [id:dp26884433318941103] 
\draw  [color={rgb, 255:red, 0; green, 0; blue, 0 }  ,draw opacity=1 ][fill={rgb, 255:red, 208; green, 208; blue, 208 }  ,fill opacity=1 ][line width=0.75]  (202.78,95.18) .. controls (202.76,94.83) and (202.75,94.48) .. (202.75,94.13) .. controls (202.75,77.49) and (222.45,64) .. (246.75,64) .. controls (271.05,64) and (290.75,77.49) .. (290.75,94.13) .. controls (290.75,94.27) and (290.75,94.42) .. (290.75,94.57) -- (281.5,94.48) .. controls (281.5,94.36) and (281.5,94.24) .. (281.5,94.13) .. controls (281.5,82.59) and (265.94,73.25) .. (246.75,73.25) .. controls (227.56,73.25) and (212,82.59) .. (212,94.13) .. controls (212,94.4) and (212,94.68) .. (212.02,94.96) -- cycle ;
%Shape: Ellipse [id:dp5864922938555954] 
\draw  [color={rgb, 255:red, 0; green, 0; blue, 0 }  ,draw opacity=1 ][fill={rgb, 255:red, 128; green, 128; blue, 128 }  ,fill opacity=1 ][line width=0.75]  (196.45,99.92) .. controls (196.45,92.89) and (201.98,87.19) .. (208.79,87.19) .. controls (215.6,87.19) and (221.12,92.89) .. (221.12,99.92) .. controls (221.12,106.95) and (215.6,112.65) .. (208.79,112.65) .. controls (201.98,112.65) and (196.45,106.95) .. (196.45,99.92) -- cycle ;
%Shape: Ellipse [id:dp29579300867020897] 
\draw  [color={rgb, 255:red, 0; green, 0; blue, 0 }  ,draw opacity=1 ][fill={rgb, 255:red, 128; green, 128; blue, 128 }  ,fill opacity=1 ][line width=0.75]  (271.45,99.92) .. controls (271.45,92.89) and (276.98,87.19) .. (283.79,87.19) .. controls (290.6,87.19) and (296.12,92.89) .. (296.12,99.92) .. controls (296.12,106.95) and (290.6,112.65) .. (283.79,112.65) .. controls (276.98,112.65) and (271.45,106.95) .. (271.45,99.92) -- cycle ;

% Text Node
\draw (171,63.9) node [anchor=north west][inner sep=0.75pt]    {$e$};
% Text Node
\draw (247,43.4) node [anchor=north west][inner sep=0.75pt]    {$f$};
% Text Node
\draw (274,129.9) node [anchor=north west][inner sep=0.75pt]    {$g$};

\end{tikzpicture}  
    \caption{Ribbon graph.}
        \label{rgarrp.3}
    \end{subfigure}
%%%%
        \begin{subfigure}[b]{.49\linewidth}
    \centering
    \tikzset{every picture/.style={line width=0.75pt}} %set default line width to 0.75pt        

\begin{tikzpicture}[x=0.95pt,y=0.95pt,yscale=-1,xscale=1]

%Shape: Arc [id:dp1769771517982679] 
\draw  [draw opacity=0] (204.42,111.87) .. controls (198.79,110.43) and (194.65,105.54) .. (194.65,99.73) .. controls (194.65,99.09) and (194.69,98.47) .. (194.79,97.86) -- (207.88,99.73) -- cycle ; \draw   (204.42,111.87) .. controls (198.79,110.43) and (194.65,105.54) .. (194.65,99.73) .. controls (194.65,99.09) and (194.69,98.47) .. (194.79,97.86) ;  
%Shape: Arc [id:dp33125219099400904] 
\draw  [draw opacity=0] (221.12,99.92) .. controls (221.05,104.37) and (218.55,108.26) .. (214.82,110.44) -- (207.88,99.73) -- cycle ; \draw   (221.12,99.92) .. controls (221.05,104.37) and (218.55,108.26) .. (214.82,110.44) ;  
%Shape: Arc [id:dp6516659777530766] 
\draw  [draw opacity=0] (213.43,88.3) .. controls (215.73,89.31) and (217.67,90.92) .. (219.03,92.93) -- (207.88,99.73) -- cycle ; \draw   (213.43,88.3) .. controls (215.73,89.31) and (217.67,90.92) .. (219.03,92.93) ;  
%Shape: Arc [id:dp6261303298462193] 
\draw  [draw opacity=0] (198,91.35) .. controls (199.5,89.76) and (201.41,88.53) .. (203.58,87.82) -- (207.88,99.73) -- cycle ; \draw   (198,91.35) .. controls (199.5,89.76) and (201.41,88.53) .. (203.58,87.82) ;  
%Shape: Arc [id:dp06051508011172191] 
\draw  [draw opacity=0] (194.84,98.33) .. controls (188.86,94.26) and (184.99,87.77) .. (184.99,80.46) .. controls (184.99,68.14) and (195.96,58.16) .. (209.5,58.16) .. controls (217.23,58.16) and (224.13,61.42) .. (228.62,66.52) -- (209.5,80.46) -- cycle ; \draw   (194.84,98.33) .. controls (188.86,94.26) and (184.99,87.77) .. (184.99,80.46) .. controls (184.99,68.14) and (195.96,58.16) .. (209.5,58.16) .. controls (217.23,58.16) and (224.13,61.42) .. (228.62,66.52) ;  
%Shape: Arc [id:dp1536208012171656] 
\draw  [draw opacity=0] (198.18,90.73) .. controls (195.37,88.09) and (193.64,84.46) .. (193.64,80.46) .. controls (193.64,72.36) and (200.74,65.79) .. (209.5,65.79) .. controls (213.8,65.79) and (217.7,67.38) .. (220.56,69.95) -- (209.5,80.46) -- cycle ; \draw   (198.18,90.73) .. controls (195.37,88.09) and (193.64,84.46) .. (193.64,80.46) .. controls (193.64,72.36) and (200.74,65.79) .. (209.5,65.79) .. controls (213.8,65.79) and (217.7,67.38) .. (220.56,69.95) ;  
%Shape: Arc [id:dp1521593861321674] 
\draw  [draw opacity=0] (278.39,111.1) .. controls (274.27,109.01) and (271.45,104.79) .. (271.45,99.92) .. controls (271.45,94.55) and (274.87,89.96) .. (279.68,88.16) -- (284.24,99.92) -- cycle ; \draw   (278.39,111.1) .. controls (274.27,109.01) and (271.45,104.79) .. (271.45,99.92) .. controls (271.45,94.55) and (274.87,89.96) .. (279.68,88.16) ;  
%Shape: Arc [id:dp5511499644354971] 
\draw  [draw opacity=0] (289.87,88.62) .. controls (294.11,90.66) and (297.03,94.95) .. (297.03,99.92) .. controls (297.03,105.55) and (293.27,110.31) .. (288.08,111.92) -- (284.24,99.92) -- cycle ; \draw   (289.87,88.62) .. controls (294.11,90.66) and (297.03,94.95) .. (297.03,99.92) .. controls (297.03,105.55) and (293.27,110.31) .. (288.08,111.92) ;  
%Shape: Arc [id:dp08787582829726603] 
\draw  [draw opacity=0] (203.45,87.91) .. controls (207.66,74.32) and (225.39,64.13) .. (246.62,64.13) .. controls (268.21,64.13) and (286.18,74.67) .. (290,88.59) -- (246.62,94.13) -- cycle ; \draw   (203.45,87.91) .. controls (207.66,74.32) and (225.39,64.13) .. (246.62,64.13) .. controls (268.21,64.13) and (286.18,74.67) .. (290,88.59) ;  
%Shape: Arc [id:dp9673687539944187] 
\draw  [draw opacity=0] (288.32,111.62) .. controls (282.62,123.7) and (265.97,132.44) .. (246.3,132.44) .. controls (226.67,132.44) and (210.03,123.72) .. (204.3,111.66) -- (246.3,102.44) -- cycle ; \draw   (288.32,111.62) .. controls (282.62,123.7) and (265.97,132.44) .. (246.3,132.44) .. controls (226.67,132.44) and (210.03,123.72) .. (204.3,111.66) ;  
%Shape: Arc [id:dp7455505570511923] 
\draw  [draw opacity=0] (213.77,88.31) .. controls (219.43,79.52) and (232.09,73.41) .. (246.81,73.41) .. controls (261.61,73.41) and (274.33,79.59) .. (279.94,88.45) -- (246.81,98.61) -- cycle ; \draw   (213.77,88.31) .. controls (219.43,79.52) and (232.09,73.41) .. (246.81,73.41) .. controls (261.61,73.41) and (274.33,79.59) .. (279.94,88.45) ;  
%Shape: Arc [id:dp047945760865582665] 
\draw  [draw opacity=0] (278.67,110.58) .. controls (272.55,118.46) and (260.58,123.81) .. (246.81,123.81) .. controls (232.9,123.81) and (220.83,118.35) .. (214.77,110.34) -- (246.81,98.61) -- cycle ; \draw   (278.67,110.58) .. controls (272.55,118.46) and (260.58,123.81) .. (246.81,123.81) .. controls (232.9,123.81) and (220.83,118.35) .. (214.77,110.34) ;  
%Shape: Arc [id:dp09257550827145022] 
\draw  [draw opacity=0] (225.19,78.35) .. controls (225.3,79.04) and (225.35,79.74) .. (225.35,80.46) .. controls (225.35,85.37) and (222.75,89.71) .. (218.74,92.37) -- (209.5,80.46) -- cycle ; \draw   (225.19,78.35) .. controls (225.3,79.04) and (225.35,79.74) .. (225.35,80.46) .. controls (225.35,85.37) and (222.75,89.71) .. (218.74,92.37) ;  
%Shape: Arc [id:dp5223786057685188] 
\draw  [draw opacity=0] (233.51,75.21) .. controls (233.91,76.82) and (234.13,78.5) .. (234.13,80.22) .. controls (234.13,88.91) and (228.67,96.44) .. (220.69,100.12) -- (209.62,80.22) -- cycle ; \draw   (233.51,75.21) .. controls (233.91,76.82) and (234.13,78.5) .. (234.13,80.22) .. controls (234.13,88.91) and (228.67,96.44) .. (220.69,100.12) ;

\end{tikzpicture}
    \caption{Boundary component.}
        \label{rgarrp.4}
    \end{subfigure}
        \begin{subfigure}[b]{.3\linewidth}
    \centering
     \tikzset{every picture/.style={line width=0.75pt}} %set default line width to 0.75pt        

\begin{tikzpicture}[x=0.75pt,y=0.75pt,yscale=-1,xscale=1]

%Shape: Circle [id:dp18731970584179414] 
\draw  [line width=1.5]  (252.39,77.84) .. controls (252.39,64.03) and (263.58,52.84) .. (277.39,52.84) .. controls (291.2,52.84) and (302.39,64.03) .. (302.39,77.84) .. controls (302.39,91.64) and (291.2,102.84) .. (277.39,102.84) .. controls (263.58,102.84) and (252.39,91.64) .. (252.39,77.84) -- cycle ;
%Shape: Arc [id:dp2631701798224957] 
\draw  [draw opacity=0][line width=1.5]  (266.36,54.7) .. controls (272,51.98) and (278.59,51.19) .. (285.01,52.97) .. controls (286.9,53.49) and (288.67,54.2) .. (290.32,55.08) -- (277.39,77.84) -- cycle ; \draw [color={rgb, 255:red, 223; green, 83; blue, 107 }  ,draw opacity=1 ][line width=1.5]    (266.36,54.7) .. controls (272,51.98) and (278.59,51.19) .. (285.01,52.97) .. controls (285.9,53.21) and (286.76,53.5) .. (287.6,53.83) ; \draw [shift={(290.32,55.08)}, rotate = 203.76] [color={rgb, 255:red, 223; green, 83; blue, 107 }  ,draw opacity=1 ][line width=1.5]    (8.53,-3.82) .. controls (5.42,-1.79) and (2.58,-0.52) .. (0,0) .. controls (2.58,0.52) and (5.42,1.79) .. (8.53,3.82)   ; 
%Shape: Arc [id:dp8740481682631627] 
\draw  [draw opacity=0][line width=1.5]  (289.18,100.6) .. controls (283.63,103.51) and (277.07,104.51) .. (270.59,102.94) .. controls (268.69,102.49) and (266.9,101.83) .. (265.23,101.01) -- (277.39,77.84) -- cycle ; \draw [color={rgb, 255:red, 223; green, 83; blue, 107 }  ,draw opacity=1 ][line width=1.5]    (289.18,100.6) .. controls (283.63,103.51) and (277.07,104.51) .. (270.59,102.94) .. controls (269.7,102.73) and (268.83,102.47) .. (267.99,102.17) ; \draw [shift={(265.23,101.01)}, rotate = 21.87] [color={rgb, 255:red, 223; green, 83; blue, 107 }  ,draw opacity=1 ][line width=1.5]    (8.53,-3.82) .. controls (5.42,-1.79) and (2.58,-0.52) .. (0,0) .. controls (2.58,0.52) and (5.42,1.79) .. (8.53,3.82)   ; 
%Shape: Circle [id:dp5067213288684962] 
\draw  [line width=1.5]  (323.89,77.09) .. controls (323.89,63.28) and (335.08,52.09) .. (348.89,52.09) .. controls (362.7,52.09) and (373.89,63.28) .. (373.89,77.09) .. controls (373.89,90.89) and (362.7,102.09) .. (348.89,102.09) .. controls (335.08,102.09) and (323.89,90.89) .. (323.89,77.09) -- cycle ;
%Shape: Arc [id:dp5192035217140911] 
\draw  [draw opacity=0][line width=1.5]  (337.86,53.95) .. controls (343.5,51.23) and (350.09,50.44) .. (356.51,52.22) .. controls (358.4,52.74) and (360.17,53.45) .. (361.82,54.33) -- (348.89,77.09) -- cycle ; \draw [color={rgb, 255:red, 223; green, 83; blue, 107 }  ,draw opacity=1 ][line width=1.5]    (337.86,53.95) .. controls (343.5,51.23) and (350.09,50.44) .. (356.51,52.22) .. controls (357.4,52.46) and (358.26,52.75) .. (359.1,53.08) ; \draw [shift={(361.82,54.33)}, rotate = 203.76] [color={rgb, 255:red, 223; green, 83; blue, 107 }  ,draw opacity=1 ][line width=1.5]    (8.53,-3.82) .. controls (5.42,-1.79) and (2.58,-0.52) .. (0,0) .. controls (2.58,0.52) and (5.42,1.79) .. (8.53,3.82)   ; 
%Shape: Arc [id:dp6547715229933235] 
\draw  [draw opacity=0][line width=1.5]  (360.68,99.85) .. controls (355.13,102.76) and (348.57,103.76) .. (342.09,102.19) .. controls (340.19,101.74) and (338.4,101.08) .. (336.73,100.26) -- (348.89,77.09) -- cycle ; \draw [color={rgb, 255:red, 223; green, 83; blue, 107 }  ,draw opacity=1 ][line width=1.5]    (360.68,99.85) .. controls (355.13,102.76) and (348.57,103.76) .. (342.09,102.19) .. controls (341.2,101.98) and (340.33,101.72) .. (339.49,101.42) ; \draw [shift={(336.73,100.26)}, rotate = 21.87] [color={rgb, 255:red, 223; green, 83; blue, 107 }  ,draw opacity=1 ][line width=1.5]    (8.53,-3.82) .. controls (5.42,-1.79) and (2.58,-0.52) .. (0,0) .. controls (2.58,0.52) and (5.42,1.79) .. (8.53,3.82)   ; 

% Text Node
\draw (271.5,28.65) node [anchor=north west][inner sep=0.75pt]  [font=\normalsize]  {$f$};
% Text Node
\draw (272.25,108.9) node [anchor=north west][inner sep=0.75pt]  [font=\normalsize]  {$g$};
% Text Node
\draw (343.25,27.15) node [anchor=north west][inner sep=0.75pt]  [font=\normalsize]  {$f$};
% Text Node
\draw (345.25,108.4) node [anchor=north west][inner sep=0.75pt]  [font=\normalsize]  {$g$};

\end{tikzpicture}
    \caption{Deleting $e$.}
        \label{rgarrp.5}
    \end{subfigure}
        \begin{subfigure}[b]{.3\linewidth}
         \tikzset{every picture/.style={line width=0.75pt}} %set default line width to 0.75pt        

\begin{tikzpicture}[x=0.75pt,y=0.75pt,yscale=-1,xscale=1]

%Shape: Chord [id:dp7762122383880941] 
\draw  [line width=1.5]  (301.38,86.59) .. controls (300.99,100.68) and (289.95,111.97) .. (276.39,111.97) .. controls (262.88,111.97) and (251.87,100.75) .. (251.41,86.73) -- cycle ;
%Shape: Arc [id:dp8740481682631627] 
\draw  [draw opacity=0][line width=1.5]  (288.68,109.83) .. controls (283.13,112.73) and (276.57,113.73) .. (270.09,112.17) .. controls (268.19,111.71) and (266.4,111.06) .. (264.72,110.23) -- (276.89,87.06) -- cycle ; \draw [color={rgb, 255:red, 223; green, 83; blue, 107 }  ,draw opacity=1 ][line width=1.5]    (288.68,109.83) .. controls (283.13,112.73) and (276.57,113.73) .. (270.09,112.17) .. controls (269.2,111.95) and (268.33,111.7) .. (267.49,111.4) ; \draw [shift={(264.72,110.23)}, rotate = 21.87] [color={rgb, 255:red, 223; green, 83; blue, 107 }  ,draw opacity=1 ][line width=1.5]    (8.53,-3.82) .. controls (5.42,-1.79) and (2.58,-0.52) .. (0,0) .. controls (2.58,0.52) and (5.42,1.79) .. (8.53,3.82)   ; 
%Shape: Circle [id:dp5067213288684962] 
\draw  [line width=1.5]  (322.89,77.09) .. controls (322.89,63.28) and (334.08,52.09) .. (347.89,52.09) .. controls (361.7,52.09) and (372.89,63.28) .. (372.89,77.09) .. controls (372.89,90.89) and (361.7,102.09) .. (347.89,102.09) .. controls (334.08,102.09) and (322.89,90.89) .. (322.89,77.09) -- cycle ;
%Shape: Arc [id:dp5192035217140911] 
\draw  [draw opacity=0][line width=1.5]  (336.86,53.95) .. controls (342.5,51.23) and (349.09,50.44) .. (355.51,52.22) .. controls (357.4,52.74) and (359.17,53.45) .. (360.82,54.33) -- (347.89,77.09) -- cycle ; \draw [color={rgb, 255:red, 223; green, 83; blue, 107 }  ,draw opacity=1 ][line width=1.5]    (336.86,53.95) .. controls (342.5,51.23) and (349.09,50.44) .. (355.51,52.22) .. controls (356.4,52.46) and (357.26,52.75) .. (358.1,53.08) ; \draw [shift={(360.82,54.33)}, rotate = 203.76] [color={rgb, 255:red, 223; green, 83; blue, 107 }  ,draw opacity=1 ][line width=1.5]    (8.53,-3.82) .. controls (5.42,-1.79) and (2.58,-0.52) .. (0,0) .. controls (2.58,0.52) and (5.42,1.79) .. (8.53,3.82)   ; 
%Shape: Arc [id:dp6547715229933235] 
\draw  [draw opacity=0][line width=1.5]  (359.68,99.85) .. controls (354.13,102.76) and (347.57,103.76) .. (341.09,102.19) .. controls (339.19,101.74) and (337.4,101.08) .. (335.73,100.26) -- (347.89,77.09) -- cycle ; \draw [color={rgb, 255:red, 223; green, 83; blue, 107 }  ,draw opacity=1 ][line width=1.5]    (359.68,99.85) .. controls (354.13,102.76) and (347.57,103.76) .. (341.09,102.19) .. controls (340.2,101.98) and (339.33,101.72) .. (338.49,101.42) ; \draw [shift={(335.73,100.26)}, rotate = 21.87] [color={rgb, 255:red, 223; green, 83; blue, 107 }  ,draw opacity=1 ][line width=1.5]    (8.53,-3.82) .. controls (5.42,-1.79) and (2.58,-0.52) .. (0,0) .. controls (2.58,0.52) and (5.42,1.79) .. (8.53,3.82)   ; 
%Shape: Chord [id:dp17897030315352025] 
\draw  [line width=1.5]  (251.4,72.04) .. controls (251.79,57.95) and (262.83,46.66) .. (276.39,46.66) .. controls (289.9,46.66) and (300.91,57.87) .. (301.38,71.89) -- cycle ;
%Shape: Arc [id:dp2631701798224957] 
\draw  [draw opacity=0][line width=1.5]  (264.86,48.42) .. controls (270.5,45.7) and (277.09,44.92) .. (283.51,46.69) .. controls (285.4,47.21) and (287.17,47.93) .. (288.81,48.81) -- (275.89,71.56) -- cycle ; \draw [color={rgb, 255:red, 223; green, 83; blue, 107 }  ,draw opacity=1 ][line width=1.5]    (264.86,48.42) .. controls (270.5,45.7) and (277.09,44.92) .. (283.51,46.69) .. controls (284.4,46.94) and (285.26,47.23) .. (286.09,47.55) ; \draw [shift={(288.81,48.81)}, rotate = 203.76] [color={rgb, 255:red, 223; green, 83; blue, 107 }  ,draw opacity=1 ][line width=1.5]    (8.53,-3.82) .. controls (5.42,-1.79) and (2.58,-0.52) .. (0,0) .. controls (2.58,0.52) and (5.42,1.79) .. (8.53,3.82)   ; 

% Text Node
\draw (268.5,23.65) node [anchor=north west][inner sep=0.75pt]  [font=\normalsize]  {$f$};
% Text Node
\draw (272.25,117.9) node [anchor=north west][inner sep=0.75pt]  [font=\normalsize]  {$g$};
% Text Node
\draw (342.25,27.15) node [anchor=north west][inner sep=0.75pt]  [font=\normalsize]  {$f$};
% Text Node
\draw (344.25,108.4) node [anchor=north west][inner sep=0.75pt]  [font=\normalsize]  {$g$};

\end{tikzpicture}
    \centering
    \caption{Contracting $e$.}
        \label{rgarrp.6}
    \end{subfigure}
        \begin{subfigure}[b]{.3\linewidth}
    \centering
     \tikzset{every picture/.style={line width=0.75pt}} %set default line width to 0.75pt        

\begin{tikzpicture}[x=0.75pt,y=0.75pt,yscale=-1,xscale=1]

%Shape: Arc [id:dp032160704809069185] 
\draw  [draw opacity=0][line width=1.5]  (302.08,88.26) .. controls (301.71,102.06) and (290.35,113.13) .. (276.39,113.13) .. controls (262.39,113.13) and (251.01,101.99) .. (250.7,88.14) -- (276.39,87.56) -- cycle ; \draw  [line width=1.5]  (302.08,88.26) .. controls (301.71,102.06) and (290.35,113.13) .. (276.39,113.13) .. controls (262.39,113.13) and (251.01,101.99) .. (250.7,88.14) ;  
%Shape: Arc [id:dp3578430700224684] 
\draw  [draw opacity=0][line width=1.5]  (250.45,70.61) .. controls (250.83,56.82) and (262.18,45.75) .. (276.14,45.75) .. controls (290.14,45.75) and (301.53,56.88) .. (301.83,70.74) -- (276.14,71.31) -- cycle ; \draw  [line width=1.5]  (250.45,70.61) .. controls (250.83,56.82) and (262.18,45.75) .. (276.14,45.75) .. controls (290.14,45.75) and (301.53,56.88) .. (301.83,70.74) ;  
%Shape: Arc [id:dp8740481682631627] 
\draw  [draw opacity=0][line width=1.5]  (288.18,110.33) .. controls (282.63,113.23) and (276.07,114.23) .. (269.59,112.67) .. controls (267.69,112.21) and (265.9,111.56) .. (264.22,110.73) -- (276.39,87.56) -- cycle ; \draw [color={rgb, 255:red, 223; green, 83; blue, 107 }  ,draw opacity=1 ][line width=1.5]    (288.18,110.33) .. controls (282.63,113.23) and (276.07,114.23) .. (269.59,112.67) .. controls (268.7,112.45) and (267.83,112.2) .. (266.99,111.9) ; \draw [shift={(264.22,110.73)}, rotate = 21.87] [color={rgb, 255:red, 223; green, 83; blue, 107 }  ,draw opacity=1 ][line width=1.5]    (8.53,-3.82) .. controls (5.42,-1.79) and (2.58,-0.52) .. (0,0) .. controls (2.58,0.52) and (5.42,1.79) .. (8.53,3.82)   ; 
%Shape: Circle [id:dp5067213288684962] 
\draw  [line width=1.5]  (318.89,78.09) .. controls (318.89,64.28) and (330.08,53.09) .. (343.89,53.09) .. controls (357.7,53.09) and (368.89,64.28) .. (368.89,78.09) .. controls (368.89,91.89) and (357.7,103.09) .. (343.89,103.09) .. controls (330.08,103.09) and (318.89,91.89) .. (318.89,78.09) -- cycle ;
%Shape: Arc [id:dp5192035217140911] 
\draw  [draw opacity=0][line width=1.5]  (332.86,54.95) .. controls (338.5,52.23) and (345.09,51.44) .. (351.51,53.22) .. controls (353.4,53.74) and (355.17,54.45) .. (356.82,55.33) -- (343.89,78.09) -- cycle ; \draw [color={rgb, 255:red, 223; green, 83; blue, 107 }  ,draw opacity=1 ][line width=1.5]    (332.86,54.95) .. controls (338.5,52.23) and (345.09,51.44) .. (351.51,53.22) .. controls (352.4,53.46) and (353.26,53.75) .. (354.1,54.08) ; \draw [shift={(356.82,55.33)}, rotate = 203.76] [color={rgb, 255:red, 223; green, 83; blue, 107 }  ,draw opacity=1 ][line width=1.5]    (8.53,-3.82) .. controls (5.42,-1.79) and (2.58,-0.52) .. (0,0) .. controls (2.58,0.52) and (5.42,1.79) .. (8.53,3.82)   ; 
%Shape: Arc [id:dp6547715229933235] 
\draw  [draw opacity=0][line width=1.5]  (355.68,100.85) .. controls (350.13,103.76) and (343.57,104.76) .. (337.09,103.19) .. controls (335.19,102.74) and (333.4,102.08) .. (331.73,101.26) -- (343.89,78.09) -- cycle ; \draw [color={rgb, 255:red, 223; green, 83; blue, 107 }  ,draw opacity=1 ][line width=1.5]    (355.68,100.85) .. controls (350.13,103.76) and (343.57,104.76) .. (337.09,103.19) .. controls (336.2,102.98) and (335.33,102.72) .. (334.49,102.42) ; \draw [shift={(331.73,101.26)}, rotate = 21.87] [color={rgb, 255:red, 223; green, 83; blue, 107 }  ,draw opacity=1 ][line width=1.5]    (8.53,-3.82) .. controls (5.42,-1.79) and (2.58,-0.52) .. (0,0) .. controls (2.58,0.52) and (5.42,1.79) .. (8.53,3.82)   ; 
%Shape: Arc [id:dp2631701798224957] 
\draw  [draw opacity=0][line width=1.5]  (265.11,48.17) .. controls (270.75,45.45) and (277.34,44.67) .. (283.76,46.44) .. controls (285.65,46.96) and (287.42,47.68) .. (289.06,48.56) -- (276.14,71.31) -- cycle ; \draw [color={rgb, 255:red, 223; green, 83; blue, 107 }  ,draw opacity=1 ][line width=1.5]    (265.11,48.17) .. controls (270.75,45.45) and (277.34,44.67) .. (283.76,46.44) .. controls (284.65,46.69) and (285.51,46.98) .. (286.34,47.3) ; \draw [shift={(289.06,48.56)}, rotate = 203.76] [color={rgb, 255:red, 223; green, 83; blue, 107 }  ,draw opacity=1 ][line width=1.5]    (8.53,-3.82) .. controls (5.42,-1.79) and (2.58,-0.52) .. (0,0) .. controls (2.58,0.52) and (5.42,1.79) .. (8.53,3.82)   ; 
%Straight Lines [id:da9245396953702181] 
\draw [line width=1.5]    (250.45,70.61) -- (302.08,88.26) ;
%Straight Lines [id:da9312008616975819] 
\draw [line width=1.5]    (250.7,88.14) -- (270,83.6) ;
%Straight Lines [id:da5231249772338727] 
\draw [line width=1.5]    (282.53,75.28) -- (301.83,70.74) ;

% Text Node
\draw (271.5,18.65) node [anchor=north west][inner sep=0.75pt]  [font=\normalsize]  {$f$};
% Text Node
\draw (272.25,117.9) node [anchor=north west][inner sep=0.75pt]  [font=\normalsize]  {$g$};
% Text Node
\draw (338.25,28.15) node [anchor=north west][inner sep=0.75pt]  [font=\normalsize]  {$f$};
% Text Node
\draw (340.25,109.4) node [anchor=north west][inner sep=0.75pt]  [font=\normalsize]  {$g$};

\end{tikzpicture}
    \caption{Penrose-contracting $e$.}
        \label{rgarrp.7}
    \end{subfigure}

    \caption{An arrow presentation, its corresponding ribbon graph, and the results of ribbon graph operations.}
    \label{rgarrp}
\end{figure}

\medskip

Arrow presentations describe ribbon graphs (which in turn describe graphs cellularly embedded in closed surfaces). We briefly review the connection here. (A reader may safely skip this discussion of ribbon graphs. It is included only for context.)
 A \emph{ribbon graph} $\rib{G}=(V,E)$ is a surface with boundary represented by the union of a set of discs $V$, called the \emph{vertices}, and another set of discs $E$, called the \emph{edges}, satisfying the following conditions: (i) the vertices and edges intersect in disjoint arcs; (ii) each such arc may only lie on the boundaries of precisely one   vertex and one edge; (iii) every edge contains precisely two such arcs. 
   
Ribbon graphs can be represented by arrow presentations. Given a ribbon graph, obtain its arrow presentation as follows: for each edge $e$, arbitrarily orient its boundary, place an arrow labelled $e$ in the direction of the orientation along each arc where the edge intersects a vertex. Taking the boundaries of the vertices together with the labelled arrows results in an arrow presentation. 
Conversely, every arrow presentation gives rise to a ribbon graph.
Given an arrow presentation, obtain a ribbon graph by identifying each circle with the boundary of a vertex disc, then for each label $e$ take a disc, arbitrarily orient its boundary, and identify disjoint arcs on its boundary to the two arrows labelled $e$ such that the direction of the arrows agree with that of the boundary. Figures~\ref{Arrp} and~\ref{rgarrp.3} provide an example of a ribbon graph and its representation as an arrow presentation.
These processes set up a 1-1 correspondence between (equivalence classes of) arrow presentations and (equivalence classes of) ribbon graphs. Further details can be found in~\cite{MR2507944, graphsonsurfaces}.

\medskip

While boundary components are immediately defined for ribbon graphs, being the boundary components of the underlying surface, they are a little awkward to define and work with in arrow presentations. (This is the one drawback here of working with arrow presentations, but the benefits outweigh the costs).
Let $\ar{G}$ be an arrow presentation. Construct a set of closed curves as follows. 
Let $\overrightarrow{p_{e,1}p_{e,2}}$ and $\overrightarrow{q_{e,1}q_{e,2}}$ be the two arrows constituting an edge $e$. 
For each edge $e$, delete the arrows together with the interiors of the vertex arcs they lie on. Next add a curve joining $q_{e,2}$
  and  $p_{e,1}$, and another one joining  $p_{e,2}$ and $q_{e,1}$. This results in a set of closed curves. Each curve is a \emph{boundary component} of $\ar{G}$.  We use $B(\ar{G})$ to denote the set of all boundary components of $\ar{G}$.
 We say a curve containing any of the points $p_{e,1}$, $p_{e,2}$, $q_{e,1}$, or $q_{e,2}$ is \emph{incident} to the edge $e$, and two boundary components are \emph{adjacent} if they are incident to the same edge. (These terms correspond exactly to their standard usage for ribbon graphs.) 
 See Figure~\ref{rgarrp.4} which shows the single boundary component of the arrow presentation of Figure~\ref{Arrp}.

\medskip

\begin{table}[t!]
\centering
\begin{tabular}{cccc}
 \tikzset{every picture/.style={line width=0.75pt}} %set default line width to 0.75pt        

\begin{tikzpicture}[x=0.5pt,y=0.5pt,yscale=-1,xscale=1]
%uncomment if require: \path (0,403); %set diagram left start at 0, and has height of 403

%Shape: Arc [id:dp1844624146015248] 
\draw  [draw opacity=0][line width=1.5]  (87.13,46.59) .. controls (87.2,46.59) and (87.27,46.59) .. (87.34,46.59) .. controls (108.57,46.59) and (125.77,62.26) .. (125.77,81.59) .. controls (125.77,100.88) and (108.64,116.53) .. (87.47,116.59) -- (87.34,81.59) -- cycle ; \draw  [color={rgb, 255:red, 0; green, 0; blue, 0 }  ,draw opacity=1 ][line width=1.5]  (87.13,46.59) .. controls (87.2,46.59) and (87.27,46.59) .. (87.34,46.59) .. controls (108.57,46.59) and (125.77,62.26) .. (125.77,81.59) .. controls (125.77,100.88) and (108.64,116.53) .. (87.47,116.59) ;  
%Shape: Arc [id:dp8390505060217666] 
\draw  [draw opacity=0][line width=1.5]  (207.14,116.54) .. controls (185.93,116.52) and (168.75,100.86) .. (168.75,81.54) .. controls (168.75,62.23) and (185.93,46.57) .. (207.14,46.54) -- (207.19,81.54) -- cycle ; \draw  [color={rgb, 255:red, 0; green, 0; blue, 0 }  ,draw opacity=1 ][line width=1.5]  (207.14,116.54) .. controls (185.93,116.52) and (168.75,100.86) .. (168.75,81.54) .. controls (168.75,62.23) and (185.93,46.57) .. (207.14,46.54) ;  
%Shape: Arc [id:dp4534169274267863] 
\draw  [draw opacity=0][line width=1.5]  (172.73,66.52) .. controls (169.64,72.21) and (168.15,78.68) .. (168.76,85.4) .. controls (169.18,89.95) and (170.52,94.22) .. (172.61,98.06) -- (208.44,81.79) -- cycle ; \draw [color={rgb, 255:red, 223; green, 83; blue, 107 }  ,draw opacity=1 ][line width=1.5]    (172.73,66.52) .. controls (169.64,72.21) and (168.15,78.68) .. (168.76,85.4) .. controls (169.05,88.58) and (169.8,91.63) .. (170.93,94.49) ; \draw [shift={(172.61,98.06)}, rotate = 251.7] [fill={rgb, 255:red, 223; green, 83; blue, 107 }  ,fill opacity=1 ][line width=0.08]  [draw opacity=0] (13.4,-6.43) -- (0,0) -- (13.4,6.44) -- (8.9,0) -- cycle    ; 
%Shape: Arc [id:dp8137208248443346] 
\draw  [draw opacity=0][line width=1.5]  (121.22,97.89) .. controls (124.43,92.54) and (126.16,86.42) .. (125.94,79.98) .. controls (125.77,75.29) and (124.59,70.85) .. (122.58,66.82) -- (85.02,81.41) -- cycle ; \draw [color={rgb, 255:red, 223; green, 83; blue, 107 }  ,draw opacity=1 ][line width=1.5]    (121.22,97.89) .. controls (124.43,92.54) and (126.16,86.42) .. (125.94,79.98) .. controls (125.82,76.68) and (125.2,73.5) .. (124.14,70.49) ; \draw [shift={(122.58,66.82)}, rotate = 73.95] [fill={rgb, 255:red, 223; green, 83; blue, 107 }  ,fill opacity=1 ][line width=0.08]  [draw opacity=0] (13.4,-6.43) -- (0,0) -- (13.4,6.44) -- (8.9,0) -- cycle    ; 

% Text Node
\draw (179.72,73.31) node [anchor=north west][inner sep=0.75pt]    {$e$};
% Text Node
\draw (98.15,72) node [anchor=north west][inner sep=0.75pt]    {$e$};

\end{tikzpicture} &
 \tikzset{every picture/.style={line width=0.75pt}} %set default line width to 0.75pt        

\begin{tikzpicture}[x=0.5pt,y=0.5pt,yscale=-1,xscale=1]
%uncomment if require: \path (0,403); %set diagram left start at 0, and has height of 403

%Shape: Arc [id:dp1844624146015248] 
\draw  [draw opacity=0][line width=1.5]  (87.13,46.59) .. controls (87.2,46.59) and (87.27,46.59) .. (87.34,46.59) .. controls (108.57,46.59) and (125.77,62.26) .. (125.77,81.59) .. controls (125.77,100.88) and (108.64,116.53) .. (87.47,116.59) -- (87.34,81.59) -- cycle ; \draw  [color={rgb, 255:red, 0; green, 0; blue, 0 }  ,draw opacity=1 ][line width=1.5]  (87.13,46.59) .. controls (87.2,46.59) and (87.27,46.59) .. (87.34,46.59) .. controls (108.57,46.59) and (125.77,62.26) .. (125.77,81.59) .. controls (125.77,100.88) and (108.64,116.53) .. (87.47,116.59) ;  
%Shape: Arc [id:dp8390505060217666] 
\draw  [draw opacity=0][line width=1.5]  (187.14,116.54) .. controls (165.93,116.52) and (148.75,100.86) .. (148.75,81.54) .. controls (148.75,62.23) and (165.93,46.57) .. (187.14,46.54) -- (187.19,81.54) -- cycle ; \draw  [color={rgb, 255:red, 0; green, 0; blue, 0 }  ,draw opacity=1 ][line width=1.5]  (187.14,116.54) .. controls (165.93,116.52) and (148.75,100.86) .. (148.75,81.54) .. controls (148.75,62.23) and (165.93,46.57) .. (187.14,46.54) ;

\end{tikzpicture} &
 \tikzset{every picture/.style={line width=0.75pt}} %set default line width to 0.75pt        

\begin{tikzpicture}[x=0.5pt,y=0.5pt,yscale=-1,xscale=1]
%uncomment if require: \path (0,403); %set diagram left start at 0, and has height of 403

%Shape: Arc [id:dp47892203953159895] 
\draw  [draw opacity=0][line width=1.5]  (151.6,79.5) .. controls (151.55,79.5) and (151.5,79.5) .. (151.45,79.5) .. controls (130.22,79.5) and (113.01,95.17) .. (113.01,114.5) .. controls (113.01,133.83) and (130.22,149.5) .. (151.45,149.5) .. controls (151.5,149.5) and (151.55,149.5) .. (151.6,149.5) -- (151.45,114.5) -- cycle ; \draw  [draw opacity=0][line width=1.5]  (151.6,79.5) .. controls (151.55,79.5) and (151.5,79.5) .. (151.45,79.5) .. controls (130.22,79.5) and (113.01,95.17) .. (113.01,114.5) .. controls (113.01,133.83) and (130.22,149.5) .. (151.45,149.5) .. controls (151.5,149.5) and (151.55,149.5) .. (151.6,149.5) ;  
%Shape: Arc [id:dp21578575308256465] 
\draw  [draw opacity=0][line width=1.5]  (81.3,79.5) .. controls (81.35,79.5) and (81.4,79.5) .. (81.45,79.5) .. controls (102.68,79.5) and (119.89,95.17) .. (119.89,114.5) .. controls (119.89,133.83) and (102.68,149.5) .. (81.45,149.5) .. controls (81.4,149.5) and (81.35,149.5) .. (81.3,149.5) -- (81.45,114.5) -- cycle ; \draw  [draw opacity=0][line width=1.5]  (81.3,79.5) .. controls (81.35,79.5) and (81.4,79.5) .. (81.45,79.5) .. controls (102.68,79.5) and (119.89,95.17) .. (119.89,114.5) .. controls (119.89,133.83) and (102.68,149.5) .. (81.45,149.5) .. controls (81.4,149.5) and (81.35,149.5) .. (81.3,149.5) ;  
%Shape: Arc [id:dp8779694649377007] 
\draw  [draw opacity=0][line width=1.5]  (80.3,79.52) .. controls (80.68,79.51) and (81.06,79.5) .. (81.45,79.5) .. controls (97.3,79.5) and (110.91,88.24) .. (116.79,100.71) -- (81.45,114.5) -- cycle ; \draw  [line width=1.5]  (80.3,79.52) .. controls (80.68,79.51) and (81.06,79.5) .. (81.45,79.5) .. controls (97.3,79.5) and (110.91,88.24) .. (116.79,100.71) ;  
%Shape: Arc [id:dp1206055040744094] 
\draw  [draw opacity=0][line width=1.5]  (116.81,128.25) .. controls (110.94,140.75) and (97.32,149.5) .. (81.45,149.5) .. controls (80.99,149.5) and (80.54,149.5) .. (80.09,149.48) -- (81.45,114.5) -- cycle ; \draw  [line width=1.5]  (116.81,128.25) .. controls (110.94,140.75) and (97.32,149.5) .. (81.45,149.5) .. controls (80.99,149.5) and (80.54,149.5) .. (80.09,149.48) ;  
%Shape: Arc [id:dp19836771891269533] 
\draw  [draw opacity=0][line width=1.5]  (152.6,79.52) .. controls (152.21,79.51) and (151.83,79.5) .. (151.45,79.5) .. controls (135.6,79.5) and (121.99,88.24) .. (116.11,100.71) -- (151.45,114.5) -- cycle ; \draw  [line width=1.5]  (152.6,79.52) .. controls (152.21,79.51) and (151.83,79.5) .. (151.45,79.5) .. controls (135.6,79.5) and (121.99,88.24) .. (116.11,100.71) ;  
%Shape: Arc [id:dp15769465954872297] 
\draw  [draw opacity=0][line width=1.5]  (116.49,129.08) .. controls (122.57,141.13) and (135.93,149.5) .. (151.45,149.5) .. controls (151.9,149.5) and (152.35,149.5) .. (152.8,149.48) -- (151.45,114.5) -- cycle ; \draw  [line width=1.5]  (116.49,129.08) .. controls (122.57,141.13) and (135.93,149.5) .. (151.45,149.5) .. controls (151.9,149.5) and (152.35,149.5) .. (152.8,149.48) ;

\end{tikzpicture} &
 \tikzset{every picture/.style={line width=0.75pt}} %set default line width to 0.75pt        

\begin{tikzpicture}[x=0.5pt,y=0.5pt,yscale=-1,xscale=1]
%uncomment if require: \path (0,403); %set diagram left start at 0, and has height of 403

%Shape: Arc [id:dp6920209365616659] 
\draw  [draw opacity=0][line width=1.5]  (67.3,66.5) .. controls (67.35,66.5) and (67.4,66.5) .. (67.45,66.5) .. controls (88.68,66.5) and (105.89,82.17) .. (105.89,101.5) .. controls (105.89,120.83) and (88.68,136.5) .. (67.45,136.5) .. controls (67.4,136.5) and (67.35,136.5) .. (67.3,136.5) -- (67.45,101.5) -- cycle ; \draw  [draw opacity=0][line width=1.5]  (67.3,66.5) .. controls (67.35,66.5) and (67.4,66.5) .. (67.45,66.5) .. controls (88.68,66.5) and (105.89,82.17) .. (105.89,101.5) .. controls (105.89,120.83) and (88.68,136.5) .. (67.45,136.5) .. controls (67.4,136.5) and (67.35,136.5) .. (67.3,136.5) ;  
%Shape: Arc [id:dp8320259027437764] 
\draw  [draw opacity=0][line width=1.5]  (164.6,65.5) .. controls (164.55,65.5) and (164.5,65.5) .. (164.45,65.5) .. controls (143.22,65.5) and (126.01,81.17) .. (126.01,100.5) .. controls (126.01,119.83) and (143.22,135.5) .. (164.45,135.5) .. controls (164.5,135.5) and (164.55,135.5) .. (164.6,135.5) -- (164.45,100.5) -- cycle ; \draw  [draw opacity=0][line width=1.5]  (164.6,65.5) .. controls (164.55,65.5) and (164.5,65.5) .. (164.45,65.5) .. controls (143.22,65.5) and (126.01,81.17) .. (126.01,100.5) .. controls (126.01,119.83) and (143.22,135.5) .. (164.45,135.5) .. controls (164.5,135.5) and (164.55,135.5) .. (164.6,135.5) ;  
%Straight Lines [id:da7044018816844937] 
\draw [line width=1.5]    (118.6,104.05) -- (133.99,115.83) ;
%Straight Lines [id:da578952161897081] 
\draw [line width=1.5]    (97.04,87.52) -- (104.85,93.8) ;
%Shape: Arc [id:dp9492213236371121] 
\draw  [draw opacity=0][line width=1.5]  (60.3,65.77) .. controls (60.68,65.76) and (61.06,65.75) .. (61.45,65.75) .. controls (77.92,65.75) and (91.97,75.18) .. (97.44,88.44) -- (61.45,100.75) -- cycle ; \draw  [line width=1.5]  (60.3,65.77) .. controls (60.68,65.76) and (61.06,65.75) .. (61.45,65.75) .. controls (77.92,65.75) and (91.97,75.18) .. (97.44,88.44) ;  
%Shape: Arc [id:dp9878456372527857] 
\draw  [draw opacity=0][line width=1.5]  (97.72,112.35) .. controls (92.47,125.98) and (78.21,135.75) .. (61.45,135.75) .. controls (60.99,135.75) and (60.54,135.75) .. (60.09,135.73) -- (61.45,100.75) -- cycle ; \draw  [line width=1.5]  (97.72,112.35) .. controls (92.47,125.98) and (78.21,135.75) .. (61.45,135.75) .. controls (60.99,135.75) and (60.54,135.75) .. (60.09,135.73) ;  
%Shape: Arc [id:dp17061596681554958] 
\draw  [draw opacity=0][line width=1.5]  (169.85,66.27) .. controls (169.46,66.26) and (169.08,66.25) .. (168.7,66.25) .. controls (152.24,66.25) and (138.19,75.67) .. (132.72,88.92) -- (168.7,101.25) -- cycle ; \draw  [line width=1.5]  (169.85,66.27) .. controls (169.46,66.26) and (169.08,66.25) .. (168.7,66.25) .. controls (152.24,66.25) and (138.19,75.67) .. (132.72,88.92) ;  
%Shape: Arc [id:dp1489727381904833] 
\draw  [draw opacity=0][line width=1.5]  (133.24,114.79) .. controls (139.05,127.4) and (152.74,136.25) .. (168.7,136.25) .. controls (169.15,136.25) and (169.6,136.25) .. (170.05,136.23) -- (168.7,101.25) -- cycle ; \draw  [line width=1.5]  (133.24,114.79) .. controls (139.05,127.4) and (152.74,136.25) .. (168.7,136.25) .. controls (169.15,136.25) and (169.6,136.25) .. (170.05,136.23) ;  
%Straight Lines [id:da9946390441178703] 
\draw [line width=1.5]    (97.21,113.61) -- (133.11,88) ;

\end{tikzpicture} \\
 $\ar{G}$& $\ar{G}\ba e$ & $\ar{G} \con e$ & $\ar{G}\pcon e$ 
\end{tabular}
\caption{Operations on an edge $e$ of an arrow presentation $\ar{G}$.}
\label{t.apops}
\end{table}

We recall the operations of  deletion, contraction and Penrose-contraction. These  are illustrated in Table~\ref{t.apops}. 
Let $\ar{G}$ be an arrow presentation with edge $e$. Let $\overrightarrow{p_1p_2}$ and $\overrightarrow{q_1q_2}$ be the two arrows constituting the edge $e$. 
Then $\ar{G}$ \emph{delete} $e$, denoted $\ar{G}\ba e$, is the arrow presentation obtained by  removing the pair of arrows labelled $e$ (but not the arcs of the vertices they lie on) from $\ar{G}$, and removing $e$ from the edge set.
Next, $\ar{G}$ \emph{contract} $e$, denoted $\ar{G}\con e$,  is the 
arrow presentation obtained by deleting the arrows $\overrightarrow{p_1p_2}$ and $\overrightarrow{q_1q_2}$ together with the interiors of the vertex arcs they lie on. Then identifying  $p_1$ and $q_2$, and identifying  $p_2$ and $q_1$. Then, finally, removing $e$ from the edge set. 
Similarly, $\ar{G}$ \emph{Penrose-contract} $e$, denoted $\ar{G}\pcon e$, is the 
arrow presentation obtained by deleting the arrows $\overrightarrow{p_1p_2}$ and $\overrightarrow{q_1q_2}$ together with the interiors of the vertex arcs they lie on. Then identifying  $p_1$ and $q_1$, and identifying  $p_2$ and $q_2$. Then, finally, removing $e$ from the edge set. 
As an example, Figures~\ref{rgarrp.5}--\ref{rgarrp.7} show the results of applying these operations to the edge $e$ of the arrow presentation shown in Figure~\ref{Arrp}. 

Note that the operations of deletion, contraction and Penrose-contraction all commute with each other and themselves when applied to distinct edges. (This follows immediately from the observation  that the operations act locally at edges.) They also correspond to the usual ribbon graph operations of the same name.

\subsection{2-sums and tensor products for arrow presentations} \label{ss:tenp}
Next we consider 2-sums and tensor products.
Although somewhat awkward in terms of ribbon graphs, these  operations are straightforward to work with when expressed in the language of arrow presentations. This is the main reason why we work with arrow presentations here rather than ribbon graphs.

Let $f$ be an edge of an arrow presentation $\ar{G}$, and $e$  be an edge of an arrow presentation $\ar{H}$.
A \emph{coupling} of $f$ and $e$ is a bijection $\varphi$ from the pair of arrows in $\ar{G}$ that constitute the edge $f$ to the pair of arrows in $\ar{H}$ that constitute the edge $e$.

The reader may find it helpful to consult Figure~\ref{2sumDfn} when reading the following definition.

\begin{definition}\label{def:ar2sum}
Let $\ar{G}$ and $\ar{H}$ be arrow presentations,  $f$ be an edge of $\ar{G}$, $e$ be an edge of  $\ar{H}$ and $\varphi$ be a coupling of $f$ and $e$. 
Further suppose $\varphi$ sends the arrow $\overrightarrow{p_1p_2}$ to the arrow $\overrightarrow{q_1q_2}$, and  sends the arrow $\overrightarrow{p_3p_4}$ to  the arrow $\overrightarrow{q_3q_4}$.
Then the \emph{2-sum} of $\ar{G}$ and $\ar{H}$ with respect to $\varphi$, denoted $\ar{G}\oplus_{\varphi} \ar{H}$ is the arrow presentation obtained by deleting the four arrows together  with the  the interiors of the arcs of the vertices that they lie on, then, for $i=1,\dots ,4$, identifying $p_i$ and $q_i$. 
\end{definition}

\begin{figure}
    \centering
    \input{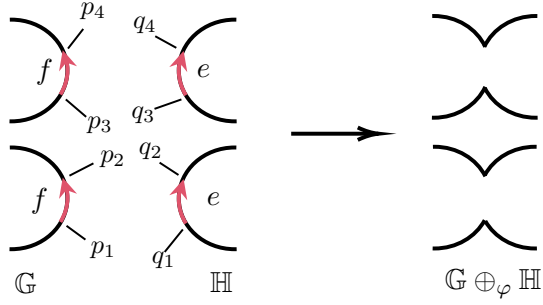}
    \caption{The 2-sum of $\ar{G}$ and $\ar{H}$  with respect to the coupling $\varphi$. }
    \label{2sumDfn}

\end{figure}

As an example, Figure~\ref{exarrp2sum} shows a 2-sum of a one-vertex arrow presentation with a three-vertex arrow presentation resulting in a two-vertex arrow presentation. Because of symmetry forming the 2-sum with respect to the other possible coupling of $f$ and $e$ would result in an equivalent arrow presentation, but in general the 2-sum does depend upon $\varphi$. 

Note that $|V(\ar{G}\oplus_{\varphi} \ar{H})|$ equals either $|V(\ar{G})|+|V(\ar{H})|-2$ or $|V(\ar{G})|+|V(\ar{H})|$. The latter is possible when each of $f$ and $e$ is incident to only one vertex (i.e., when both edges are loops).

\begin{figure}
    \centering
    \begin{tabular}{ccc}
      \input{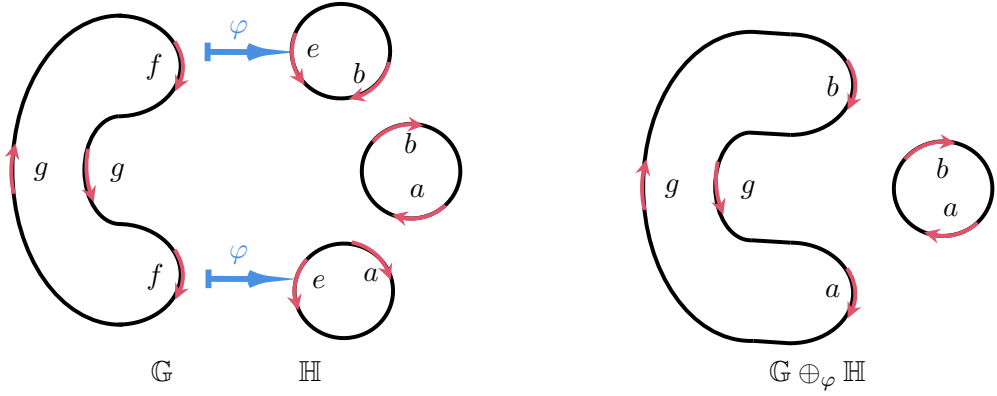} && \tikzset{every picture/.style={line width=0.75pt}} %set default line width to 0.75pt        

\begin{tikzpicture}[x=0.45pt,y=0.45pt,yscale=-1,xscale=1]
%uncomment if require: \path (0,403); %set diagram left start at 0, and has height of 403

%Shape: Arc [id:dp33237687366428903] 
\draw  [draw opacity=0][line width=1.5]  (147.38,155.76) .. controls (176.38,155.73) and (199.88,136.81) .. (199.88,113.47) .. controls (199.88,90.12) and (176.38,71.2) .. (147.38,71.17) -- (147.31,113.47) -- cycle ; \draw  [color={rgb, 255:red, 0; green, 0; blue, 0 }  ,draw opacity=1 ][line width=1.5]  (147.38,155.76) .. controls (176.38,155.73) and (199.88,136.81) .. (199.88,113.47) .. controls (199.88,90.12) and (176.38,71.2) .. (147.38,71.17) ;  
%Shape: Arc [id:dp59491695459116] 
\draw  [draw opacity=0][line width=1.5]  (194.74,92.87) .. controls (199.91,100.25) and (202.49,108.85) .. (201.57,117.82) .. controls (200.94,123.98) and (198.72,129.7) .. (195.25,134.77) -- (147.31,113.47) -- cycle ; \draw [color={rgb, 255:red, 223; green, 83; blue, 107 }  ,draw opacity=1 ][line width=1.5]    (194.74,92.87) .. controls (199.91,100.25) and (202.49,108.85) .. (201.57,117.82) .. controls (201.07,122.65) and (199.6,127.22) .. (197.31,131.41) ; \draw [shift={(195.25,134.77)}, rotate = 295.45] [fill={rgb, 255:red, 223; green, 83; blue, 107 }  ,fill opacity=1 ][line width=0.08]  [draw opacity=0] (13.4,-6.43) -- (0,0) -- (13.4,6.44) -- (8.9,0) -- cycle    ; 
%Shape: Arc [id:dp3336227878954384] 
\draw  [draw opacity=0][line width=1.5]  (147.38,330.98) .. controls (176.38,330.95) and (199.88,312.03) .. (199.88,288.69) .. controls (199.88,265.35) and (176.38,246.42) .. (147.38,246.39) -- (147.31,288.69) -- cycle ; \draw  [color={rgb, 255:red, 0; green, 0; blue, 0 }  ,draw opacity=1 ][line width=1.5]  (147.38,330.98) .. controls (176.38,330.95) and (199.88,312.03) .. (199.88,288.69) .. controls (199.88,265.35) and (176.38,246.42) .. (147.38,246.39) ;  
%Shape: Arc [id:dp9515306852369716] 
\draw  [draw opacity=0][line width=1.5]  (194.74,268.09) .. controls (199.91,275.47) and (202.49,284.07) .. (201.57,293.04) .. controls (200.94,299.2) and (198.72,304.93) .. (195.25,309.99) -- (147.31,288.69) -- cycle ; \draw [color={rgb, 255:red, 223; green, 83; blue, 107 }  ,draw opacity=1 ][line width=1.5]    (194.74,268.09) .. controls (199.91,275.47) and (202.49,284.07) .. (201.57,293.04) .. controls (201.07,297.88) and (199.6,302.44) .. (197.31,306.64) ; \draw [shift={(195.25,309.99)}, rotate = 295.45] [fill={rgb, 255:red, 223; green, 83; blue, 107 }  ,fill opacity=1 ][line width=0.08]  [draw opacity=0] (13.4,-6.43) -- (0,0) -- (13.4,6.44) -- (8.9,0) -- cycle    ; 
%Shape: Arc [id:dp9323522787371721] 
\draw  [draw opacity=0][line width=1.5]  (115.75,69.07) .. controls (115.15,69.06) and (114.55,69.05) .. (113.95,69.05) .. controls (63.96,69.05) and (23.43,127.26) .. (23.43,199.06) .. controls (23.43,270.86) and (63.96,329.07) .. (113.95,329.07) .. controls (114.71,329.07) and (115.47,329.05) .. (116.22,329.03) -- (113.95,199.06) -- cycle ; \draw  [line width=1.5]  (115.75,69.07) .. controls (115.15,69.06) and (114.55,69.05) .. (113.95,69.05) .. controls (63.96,69.05) and (23.43,127.26) .. (23.43,199.06) .. controls (23.43,270.86) and (63.96,329.07) .. (113.95,329.07) .. controls (114.71,329.07) and (115.47,329.05) .. (116.22,329.03) ;  
%Shape: Arc [id:dp6273273427613371] 
\draw  [draw opacity=0][line width=1.5]  (112.38,153.76) .. controls (96.54,155) and (83.95,174.81) .. (83.95,199.06) .. controls (83.95,224.11) and (97.38,244.42) .. (113.95,244.42) .. controls (114.22,244.42) and (114.48,244.41) .. (114.74,244.4) -- (113.95,199.06) -- cycle ; \draw  [line width=1.5]  (112.38,153.76) .. controls (96.54,155) and (83.95,174.81) .. (83.95,199.06) .. controls (83.95,224.11) and (97.38,244.42) .. (113.95,244.42) .. controls (114.22,244.42) and (114.48,244.41) .. (114.74,244.4) ;  
%Shape: Arc [id:dp7244781975231495] 
\draw  [draw opacity=0][line width=1.5]  (24.64,218.97) .. controls (22.69,211.03) and (21.98,202.39) .. (22.76,193.42) .. controls (23.27,187.49) and (24.4,181.78) .. (26.06,176.4) -- (73.76,196.68) -- cycle ; \draw [color={rgb, 255:red, 223; green, 83; blue, 107 }  ,draw opacity=1 ][line width=1.5]    (24.64,218.97) .. controls (22.69,211.03) and (21.98,202.39) .. (22.76,193.42) .. controls (23.15,188.88) and (23.9,184.48) .. (24.98,180.25) ; \draw [shift={(26.06,176.4)}, rotate = 102.98] [fill={rgb, 255:red, 223; green, 83; blue, 107 }  ,fill opacity=1 ][line width=0.08]  [draw opacity=0] (13.4,-6.43) -- (0,0) -- (13.4,6.44) -- (8.9,0) -- cycle    ; 
%Shape: Arc [id:dp813773708371825] 
\draw  [draw opacity=0][line width=1.5]  (86.1,178.2) .. controls (84.33,186.68) and (84.08,195.86) .. (85.62,205.28) .. controls (86.65,211.53) and (88.39,217.46) .. (90.74,222.96) -- (139.52,197.68) -- cycle ; \draw [color={rgb, 255:red, 223; green, 83; blue, 107 }  ,draw opacity=1 ][line width=1.5]    (86.1,178.2) .. controls (84.33,186.68) and (84.08,195.86) .. (85.62,205.28) .. controls (86.42,210.15) and (87.66,214.83) .. (89.28,219.27) ; \draw [shift={(90.74,222.96)}, rotate = 251.33] [fill={rgb, 255:red, 223; green, 83; blue, 107 }  ,fill opacity=1 ][line width=0.08]  [draw opacity=0] (13.4,-6.43) -- (0,0) -- (13.4,6.44) -- (8.9,0) -- cycle    ; 
%Shape: Ellipse [id:dp8401220115927039] 
\draw  [line width=1.5]  (265.14,162.5) .. controls (287.27,156.99) and (309.87,169.75) .. (315.61,191) .. controls (321.35,212.24) and (308.06,233.93) .. (285.93,239.44) .. controls (263.79,244.95) and (241.2,232.2) .. (235.46,210.95) .. controls (229.71,189.7) and (243,168.01) .. (265.14,162.5) -- cycle ;
%Shape: Arc [id:dp5366521591806768] 
\draw  [draw opacity=0][line width=1.5]  (259.39,237.66) .. controls (267.5,240.97) and (276.77,241.83) .. (285.97,239.55) .. controls (293.16,237.76) and (299.42,234.26) .. (304.4,229.62) -- (275.57,201.08) -- cycle ; \draw [color={rgb, 255:red, 223; green, 83; blue, 107 }  ,draw opacity=1 ][line width=1.5]    (263.24,239.02) .. controls (270.37,241.16) and (278.2,241.48) .. (285.97,239.55) .. controls (293.16,237.76) and (299.42,234.26) .. (304.4,229.62) ;  \draw [shift={(259.39,237.66)}, rotate = 14.06] [fill={rgb, 255:red, 223; green, 83; blue, 107 }  ,fill opacity=1 ][line width=0.08]  [draw opacity=0] (13.4,-6.43) -- (0,0) -- (13.4,6.44) -- (8.9,0) -- cycle    ;
%Shape: Arc [id:dp5919418260673783] 
\draw  [draw opacity=0][line width=1.5]  (286.07,162.6) .. controls (277.89,160.62) and (268.97,161.02) .. (260.49,164.29) .. controls (253.59,166.95) and (247.82,171.18) .. (243.46,176.37) -- (275.57,201.08) -- cycle ; \draw [color={rgb, 255:red, 223; green, 83; blue, 107 }  ,draw opacity=1 ][line width=1.5]    (282.09,161.84) .. controls (275.05,160.82) and (267.62,161.55) .. (260.49,164.29) .. controls (253.59,166.95) and (247.82,171.18) .. (243.46,176.37) ;  \draw [shift={(286.07,162.6)}, rotate = 185.46] [fill={rgb, 255:red, 223; green, 83; blue, 107 }  ,fill opacity=1 ][line width=0.08]  [draw opacity=0] (13.4,-6.43) -- (0,0) -- (13.4,6.44) -- (8.9,0) -- cycle    ;
%Straight Lines [id:da793122553704506] 
\draw [line width=1.5]    (115.75,69.07) -- (147.38,71.17) ;
%Straight Lines [id:da011139481608608115] 
\draw [line width=1.5]    (112.38,153.76) -- (147.38,155.76) ;
%Straight Lines [id:da8718270222344696] 
\draw [line width=1.5]    (114.74,244.4) -- (147.38,246.39) ;
%Straight Lines [id:da9735852878448448] 
\draw [line width=1.5]    (116.22,329.03) -- (147.86,331.12) ;

% Text Node
\draw (103,191.4) node [anchor=north west][inner sep=0.75pt]    {$g$};
% Text Node
\draw (39,190.4) node [anchor=north west][inner sep=0.75pt]    {$g$};
% Text Node
\draw (273,211.4) node [anchor=north west][inner sep=0.75pt]    {$a$};
% Text Node
\draw (174,279.4) node [anchor=north west][inner sep=0.75pt]    {$a$};
% Text Node
\draw (175,103.4) node [anchor=north west][inner sep=0.75pt]    {$b$};
% Text Node
\draw (267.14,172.9) node [anchor=north west][inner sep=0.75pt]    {$b$};

\end{tikzpicture} \\
     $\ar{G}$ \qquad \qquad $\ar{H}$ &  ~\hspace{1cm}~&$\ar{G}\oplus_{\varphi} \ar{H}$
    \end{tabular}
    \caption{Two arrow presentations, $\ar{G}$ and $\ar{H}$, with a coupling $\varphi$, and the 2-sum $\ar{G}\oplus_{\varphi} \ar{H}$.}
    \label{exarrp2sum}
\end{figure}

The next two results are  obvious, following immediately upon drawing how the operations change an arrow presentation. 
\begin{lemma}\label{lem:sumord}
Let $\ar{G}$ be an arrow presentations  with an edge $f$,  let $\ar{H}$ be an arrow presentation with distinct edges $e$ and $g$, and let $\ar{K}$ an arrow presentation with an edge $h$. 
In addition let $\varphi_{f,e}$ be a coupling of $f$ and $e$, and $\varphi_{g,h}$ be a coupling of $g$ and $h$.

Then  \[
\ar{G} \oplus_{\varphi_{f,e}} \ar{H} = \ar{H} \oplus_{\varphi^{-1}_{f,e}} \ar{G}
\qquad \text{and} \qquad 
\ar{G} \oplus_{\varphi_{f,e}} (\ar{H} \oplus_{\varphi_{g,h}} \ar{K})  = (\ar{G} \oplus_{\varphi_{f,e}} \ar{H} )\oplus_{\varphi_{g,h}} \ar{K} .\] 
\end{lemma}

\begin{lemma}\label{lem:sumop}
   Let $\ar{G}$ and $\ar{H}$ be arrow presentations, $f$ be an edge of $\ar{G}$, and $e$ and $g$ be distinct edges of $\ar{H}$. Additionally let $\varphi$ be a coupling of $f$ and $e$. Then
   \begin{enumerate}
       \item $(\ar{G}\oplus_{\varphi}\ar{H})\ba g = \ar{G}\oplus_{\varphi}(\ar{H}\ba g) $,
       \item $(\ar{G}\oplus_{\varphi}\ar{H})\con g = \ar{G}\oplus_{\varphi}(\ar{H}\con g) $, and
       \item  $(\ar{G}\oplus_{\varphi}\ar{H})\pcon g = \ar{G}\oplus_{\varphi}(\ar{H}\pcon g) $.
   \end{enumerate}

\end{lemma}

We now make use of Lemma~\ref{lem:sumord} to define tensor products. The idea is to form a 2-sum at each edge $f$ of $\ar{G}$ with some arrow presentation $\ar{H}^{(f)}$.  
\begin{definition}\label{d.argenten}
Let $\ar{G}$ be an arrow presentation and $\{\ar{H}^{(f)}\}_{f\in E(\ar{G})}$ be a family of arrow presentations indexed by the edges of $\ar{G}$. All the arrow presentations here are  distinct.
Further suppose that for each edge $f$ of $\ar{G}$ there is a coupling $\varphi_f$ of $f$ with an edge $e^{(f)}$ of $\ar{H}^{(f)}$. Let $\bphi=\{ \varphi_{f} \}_{f\in E(\ar{G})}$ denote the indexed set of such couplings. 
Then the \emph{tensor product} $\ar{G}\ot_{\bphi} \{\ar{H}^{(f)}\}_{f\in E(\ar{G})}$ is the arrow presentation obtained by starting with $\ar{G}$ and then for each edge $f$ in $\ar{G}$ taking the 2-sum with $\ar{H}^{(f)}$ with respect to $\varphi_f$:
\[  \ar{G}\ot_{\bphi} \{\ar{H}^{(f)}\}_{f\in E(\ar{G})} =  \ar{G} \bigoplus_{\substack{\varphi_{f} \\f\in E(\ar{G})}} \ar{H}^{(f)} .  \]
\end{definition}

The following special case of a tensor product which repeatedly 2-sums a copy of the same arrow presentation is of interest. This is the analogue of the tensor product appearing in Brylawski's tensor product formula~\eqref{btf}.
\begin{definition}
Let $\ar{G}$ be and $\ar{H}$ be arrow presentations. Let $e$ be a fixed edge of $\ar{H}$ and for each edge $f$ of $\ar{G}$ let $\varphi_f$ be a coupling of $f$ and $e$, and $\bphi = \{\varphi_f\}_{f\in E(\ar{G})}$. Then 
\[  \ar{G}\ot_{\bphi} \ar{H} =  \ar{G} \ot_{\boldsymbol{\psi}} \{\ar{H}^{(f)}\}_{f\in E(\ar{G})},  \]
where each $\ar{H}^{(f)}$ is an arrow presentation equivalent to $\ar{H}$ (with all copies distinct); $e^{(f)}$ is the edge in $\ar{H}^{(f)}$ corresponding to $e$;  $\psi_f$ is the coupling of $f$ and $e^{(f)}$ induced from $\varphi_f$ under the equivalence; and  $\boldsymbol{\psi} =\{\psi_f\}_{f\in E(\ar{G})}$.
\end{definition}

\subsection{Packaged arrow presentations}\label{ss.dap}

Let $\ar{G}$ be an arrow presentation. A \emph{vertex partition} of $\ar{G}$ is a partition $\V$ of its vertex set $V(\ar{G})$. We shall call the blocks of the partition \emph{vertex classes} and use $[v]_{\V}$ to denote the vertex class containing $v$. 
Similarly, a \emph{boundary partition} $\B$ of $\ar{G}$ is a partition of its  set of boundary components $B(\ar{G})$. We shall call the blocks of this partition \emph{boundary classes} and use $[b]_{\B}$ to denote the boundary class containing $b$. 
At times, when there is no potential for confusion, we omit  subscripts and write $[v]$ for $[v]_{\V}$, and $[b]$ for $[b]_{\B}$.

\begin{definition}
A \emph{packaged arrow presentation} $\pack{G}$ is a triple $(\ar{G}, \V,\B)$ where $\ar{G}$ is an arrow presentation, $\V$ is a vertex partition of $\ar{G}$, and  $\B$ is a boundary partition of $\ar{G}$.
\end{definition}

Terms (and the corresponding notation) such as edges, vertices, boundary components, couplings, etc. of a packaged arrow presentation refer to the edges, vertices, boundary components, couplings, etc., of its arrow presentation.

\begin{convention}\label{figconvention}
When depicting packaged arrow presentations in figures we shall draw the vertices (i.e., circles) of the arrow presentation using solid lines, and the boundary components as dashed lines sitting a little away from the arrow presentation. (Which side of the arrow presentation it is drawn on has no meaning.) We shall indicate which vertex or boundary class a curve belongs with annotations $[u]_{\V}$, $[b]_{\B}$, etc. Unless otherwise indicated, it is possible that the classes with different labels,  $[u]_{\V}$, $[v]_{\V}$ etc., are equal.
\end{convention}
As an illustration of these conventions, Figure~\ref{expackdelcon} shows some packaged arrow presentations. Additionally, Figure~\ref{fig:kgrph} shows all the packaged arrow presentations on one edge that have no isolated vertices. In that figure $[u]_{\V}\neq [v]_{\V}$ and $[a]_{\B}\neq [b]_{\B}$.

\begin{remark}\label{r.papcrg}
As noted above, arrow presentations correspond to ribbon graphs.  Packaged arrow presentations correspond to Huggett and Moffatt's coloured ribbon graphs~\cite{HM}. (The correspondence is the obvious one: vertex and boundary components in the same block of the partition are given the same colour.) They arose as a model for  graphs non-cellularly embedded in pseudo-surfaces (see \cite[Section~2]{HM}), and  the results in this paper may be expressed in that language too. We shall shortly give definitions of deletion and contraction for packaged arrow presentations. These correspond to the operations on coloured ribbon graphs from~\cite{HM}. 
\end{remark}

\medskip 

We consider five  operations acting on the edges of packaged arrow presentations: deletion, contraction, Penrose-contraction, merge-deletion and merge-contraction. A reader may find it helpful to refer to Table~\ref{tab:edgeops} when reading the  definitions. Additionally, examples of contraction and merge-deletion are given in Figure~\ref{expackdelcon}.

\begin{table}[t!]
\centering
\begin{tabular}{cccc}\vspace{5mm}
 \tikzset{every picture/.style={line width=0.75pt}} %set default line width to 0.75pt        

\begin{tikzpicture}[x=0.5pt,y=0.5pt,yscale=-1,xscale=1]
%uncomment if require: \path (0,403); %set diagram left start at 0, and has height of 403

%Shape: Arc [id:dp1844624146015248] 
\draw  [draw opacity=0][line width=2.25]  (87.13,46.59) .. controls (87.2,46.59) and (87.27,46.59) .. (87.34,46.59) .. controls (108.57,46.59) and (125.77,62.26) .. (125.77,81.59) .. controls (125.77,100.88) and (108.64,116.53) .. (87.47,116.59) -- (87.34,81.59) -- cycle ; \draw  [color={rgb, 255:red, 0; green, 0; blue, 0 }  ,draw opacity=1 ][line width=2.25]  (87.13,46.59) .. controls (87.2,46.59) and (87.27,46.59) .. (87.34,46.59) .. controls (108.57,46.59) and (125.77,62.26) .. (125.77,81.59) .. controls (125.77,100.88) and (108.64,116.53) .. (87.47,116.59) ;  
%Shape: Arc [id:dp3333496576896178] 
\draw  [draw opacity=0][dash pattern={on 5.63pt off 4.5pt}][line width=1.5]  (128.34,63.97) .. controls (122.54,49.95) and (107.87,40) .. (90.69,40) .. controls (89.66,40) and (88.63,40.04) .. (87.62,40.11) -- (90.69,77.3) -- cycle ; \draw [color={rgb, 255:red, 0; green, 0; blue, 0 }  ,draw opacity=1 ][dash pattern={on 5.63pt off 4.5pt}][line width=1.5]  [dash pattern={on 5.63pt off 4.5pt}]  (128.34,63.97) .. controls (122.54,49.95) and (107.87,40) .. (90.69,40) .. controls (89.66,40) and (88.63,40.04) .. (87.62,40.11) ;  
%Shape: Arc [id:dp1309371760639726] 
\draw  [draw opacity=0][dash pattern={on 5.63pt off 4.5pt}][line width=1.5]  (88.55,124.43) .. controls (89.39,124.48) and (90.24,124.5) .. (91.09,124.5) .. controls (107.38,124.5) and (121.42,115.5) .. (127.77,102.55) -- (91.09,87.02) -- cycle ; \draw [color={rgb, 255:red, 0; green, 0; blue, 0 }  ,draw opacity=1 ][dash pattern={on 5.63pt off 4.5pt}][line width=1.5]  [dash pattern={on 5.63pt off 4.5pt}]  (88.55,124.43) .. controls (89.39,124.48) and (90.24,124.5) .. (91.09,124.5) .. controls (107.38,124.5) and (121.42,115.5) .. (127.77,102.55) ;  
%Shape: Arc [id:dp8390505060217666] 
\draw  [draw opacity=0][line width=2.25]  (207.14,116.54) .. controls (185.93,116.52) and (168.75,100.86) .. (168.75,81.54) .. controls (168.75,62.23) and (185.93,46.57) .. (207.14,46.54) -- (207.19,81.54) -- cycle ; \draw  [color={rgb, 255:red, 0; green, 0; blue, 0 }  ,draw opacity=1 ][line width=2.25]  (207.14,116.54) .. controls (185.93,116.52) and (168.75,100.86) .. (168.75,81.54) .. controls (168.75,62.23) and (185.93,46.57) .. (207.14,46.54) ;  
%Shape: Arc [id:dp23590314788047229] 
\draw  [draw opacity=0][dash pattern={on 5.63pt off 4.5pt}][line width=1.5]  (167.52,102.68) .. controls (174.08,115.31) and (187.81,124) .. (203.69,124) .. controls (204.75,124) and (205.8,123.96) .. (206.84,123.89) -- (203.69,85.87) -- cycle ; \draw [color={rgb, 255:red, 0; green, 0; blue, 0 }  ,draw opacity=1 ][dash pattern={on 5.63pt off 4.5pt}][line width=1.5]  [dash pattern={on 5.63pt off 4.5pt}]  (167.52,102.68) .. controls (174.08,115.31) and (187.81,124) .. (203.69,124) .. controls (204.75,124) and (205.8,123.96) .. (206.84,123.89) ;  
%Shape: Arc [id:dp4213503765254092] 
\draw  [draw opacity=0][dash pattern={on 5.63pt off 4.5pt}][line width=1.5]  (205.81,38.33) .. controls (204.96,38.28) and (204.09,38.25) .. (203.22,38.25) .. controls (185.46,38.25) and (170.39,49.18) .. (165.02,64.35) -- (203.22,76.58) -- cycle ; \draw [color={rgb, 255:red, 0; green, 0; blue, 0 }  ,draw opacity=1 ][dash pattern={on 5.63pt off 4.5pt}][line width=1.5]  [dash pattern={on 5.63pt off 4.5pt}]  (205.81,38.33) .. controls (204.96,38.28) and (204.09,38.25) .. (203.22,38.25) .. controls (185.46,38.25) and (170.39,49.18) .. (165.02,64.35) ;  
%Straight Lines [id:da8748297744976138] 
\draw [color={rgb, 255:red, 0; green, 0; blue, 0 }  ,draw opacity=1 ][line width=1.5]  [dash pattern={on 5.63pt off 4.5pt}]  (127.12,63.6) -- (166.01,63.61) ;
%Shape: Arc [id:dp4534169274267863] 
\draw  [draw opacity=0][line width=2.25]  (172.73,66.52) .. controls (169.64,72.21) and (168.15,78.68) .. (168.76,85.4) .. controls (169.18,89.95) and (170.52,94.22) .. (172.61,98.06) -- (208.44,81.79) -- cycle ; \draw [color={rgb, 255:red, 223; green, 83; blue, 107 }  ,draw opacity=1 ][line width=2.25]    (172.73,66.52) .. controls (169.64,72.21) and (168.15,78.68) .. (168.76,85.4) .. controls (169.02,88.24) and (169.64,90.97) .. (170.58,93.56) ; \draw [shift={(172.61,98.06)}, rotate = 251.7] [fill={rgb, 255:red, 223; green, 83; blue, 107 }  ,fill opacity=1 ][line width=0.08]  [draw opacity=0] (16.07,-7.72) -- (0,0) -- (16.07,7.72) -- (10.67,0) -- cycle    ; 
%Shape: Arc [id:dp8137208248443346] 
\draw  [draw opacity=0][line width=2.25]  (121.22,97.89) .. controls (124.43,92.54) and (126.16,86.42) .. (125.94,79.98) .. controls (125.77,75.29) and (124.59,70.85) .. (122.58,66.82) -- (85.02,81.41) -- cycle ; \draw [color={rgb, 255:red, 223; green, 83; blue, 107 }  ,draw opacity=1 ][line width=2.25]    (121.22,97.89) .. controls (124.43,92.54) and (126.16,86.42) .. (125.94,79.98) .. controls (125.83,77) and (125.32,74.13) .. (124.44,71.39) ; \draw [shift={(122.58,66.82)}, rotate = 73.95] [fill={rgb, 255:red, 223; green, 83; blue, 107 }  ,fill opacity=1 ][line width=0.08]  [draw opacity=0] (16.07,-7.72) -- (0,0) -- (16.07,7.72) -- (10.67,0) -- cycle    ; 
%Straight Lines [id:da09054739808662116] 
\draw    (69.25,108.35) -- (92.5,110.75) ;
%Straight Lines [id:da2887461789296373] 
\draw    (204,111) -- (227.5,106) ;
%Straight Lines [id:da1434931559633963] 
\draw    (146,108.5) -- (146,122.75) ;
%Straight Lines [id:da27729487042147305] 
\draw    (145.75,40.75) -- (145.75,55) ;
%Straight Lines [id:da5610922207679844] 
\draw [color={rgb, 255:red, 0; green, 0; blue, 0 }  ,draw opacity=1 ][line width=1.5]  [dash pattern={on 5.63pt off 4.5pt}]  (127.88,103.36) -- (168.95,103.4) ;

% Text Node
\draw (179.72,73.31) node [anchor=north west][inner sep=0.75pt]    {$e$};
% Text Node
\draw (98.15,72) node [anchor=north west][inner sep=0.75pt]    {$e$};
% Text Node
\draw (134.9,13.5) node [anchor=north west][inner sep=0.75pt]    {$[ a]\mathcal{_{B}}$};
% Text Node
\draw (134.9,126.85) node [anchor=north west][inner sep=0.75pt]    {$[ b]_{\mathcal{B}}$};
% Text Node
\draw (31.9,100.35) node [anchor=north west][inner sep=0.75pt]    {$[ u]_{\mathcal{V}}$};
% Text Node
\draw (231.4,100.6) node [anchor=north west][inner sep=0.75pt]    {$[ v]_{\mathcal{V}}$};

\end{tikzpicture} &
 \tikzset{every picture/.style={line width=0.75pt}} %set default line width to 0.75pt        

\begin{tikzpicture}[x=0.5pt,y=0.5pt,yscale=-1,xscale=1]
%uncomment if require: \path (0,403); %set diagram left start at 0, and has height of 403

%Shape: Arc [id:dp1844624146015248] 
\draw  [draw opacity=0][line width=2.25]  (87.13,46.59) .. controls (87.2,46.59) and (87.27,46.59) .. (87.34,46.59) .. controls (108.57,46.59) and (125.77,62.26) .. (125.77,81.59) .. controls (125.77,100.88) and (108.64,116.53) .. (87.47,116.59) -- (87.34,81.59) -- cycle ; \draw  [color={rgb, 255:red, 0; green, 0; blue, 0 }  ,draw opacity=1 ][line width=2.25]  (87.13,46.59) .. controls (87.2,46.59) and (87.27,46.59) .. (87.34,46.59) .. controls (108.57,46.59) and (125.77,62.26) .. (125.77,81.59) .. controls (125.77,100.88) and (108.64,116.53) .. (87.47,116.59) ;  
%Shape: Arc [id:dp3333496576896178] 
\draw  [draw opacity=0][dash pattern={on 5.63pt off 4.5pt}][line width=1.5]  (88.97,123.87) .. controls (113.88,123.08) and (133.82,104.46) .. (133.82,81.59) .. controls (133.82,58.23) and (113.01,39.29) .. (87.34,39.29) .. controls (86.91,39.29) and (86.48,39.3) .. (86.05,39.31) -- (87.34,81.59) -- cycle ; \draw [color={rgb, 255:red, 0; green, 0; blue, 0 }  ,draw opacity=1 ][dash pattern={on 5.63pt off 4.5pt}][line width=1.5]  [dash pattern={on 5.63pt off 4.5pt}]  (88.97,123.87) .. controls (113.88,123.08) and (133.82,104.46) .. (133.82,81.59) .. controls (133.82,58.23) and (113.01,39.29) .. (87.34,39.29) .. controls (86.91,39.29) and (86.48,39.3) .. (86.05,39.31) ;  
%Shape: Arc [id:dp8390505060217666] 
\draw  [draw opacity=0][line width=2.25]  (207.14,116.54) .. controls (185.93,116.52) and (168.75,100.86) .. (168.75,81.54) .. controls (168.75,62.23) and (185.93,46.57) .. (207.14,46.54) -- (207.19,81.54) -- cycle ; \draw  [color={rgb, 255:red, 0; green, 0; blue, 0 }  ,draw opacity=1 ][line width=2.25]  (207.14,116.54) .. controls (185.93,116.52) and (168.75,100.86) .. (168.75,81.54) .. controls (168.75,62.23) and (185.93,46.57) .. (207.14,46.54) ;  
%Straight Lines [id:da09054739808662116] 
\draw    (69.25,108.35) -- (92.5,110.75) ;
%Straight Lines [id:da2887461789296373] 
\draw    (204,111) -- (227.5,106) ;
%Straight Lines [id:da27729487042147305] 
\draw    (173.75,117.75) -- (163,128.6) ;
%Shape: Arc [id:dp004813584382169278] 
\draw  [draw opacity=0][dash pattern={on 5.63pt off 4.5pt}][line width=1.5]  (205.56,39.27) .. controls (180.64,40.05) and (160.71,58.68) .. (160.71,81.54) .. controls (160.71,104.9) and (181.52,123.84) .. (207.19,123.84) .. controls (207.62,123.84) and (208.05,123.84) .. (208.47,123.83) -- (207.19,81.54) -- cycle ; \draw [color={rgb, 255:red, 0; green, 0; blue, 0 }  ,draw opacity=1 ][dash pattern={on 5.63pt off 4.5pt}][line width=1.5]  [dash pattern={on 5.63pt off 4.5pt}]  (205.56,39.27) .. controls (180.64,40.05) and (160.71,58.68) .. (160.71,81.54) .. controls (160.71,104.9) and (181.52,123.84) .. (207.19,123.84) .. controls (207.62,123.84) and (208.05,123.84) .. (208.47,123.83) ;  
%Straight Lines [id:da4674118280722672] 
\draw    (125.5,116.65) -- (134.75,128.15) ;

% Text Node
\draw (136.75,131.55) node [anchor=north west][inner sep=0.75pt]    {$[ c]\mathcal{_{B'}}$};
% Text Node
\draw (31.9,100.35) node [anchor=north west][inner sep=0.75pt]    {$[ u]_{\mathcal{V}}$};
% Text Node
\draw (231.4,100.6) node [anchor=north west][inner sep=0.75pt]    {$[ v]_{\mathcal{V}}$};

\end{tikzpicture} &
 \tikzset{every picture/.style={line width=0.75pt}} %set default line width to 0.75pt        

\begin{tikzpicture}[x=0.5pt,y=0.5pt,yscale=-1,xscale=1]
%uncomment if require: \path (0,403); %set diagram left start at 0, and has height of 403

%Shape: Arc [id:dp1844624146015248] 
\draw  [draw opacity=0][line width=2.25]  (87.13,46.59) .. controls (87.2,46.59) and (87.27,46.59) .. (87.34,46.59) .. controls (108.57,46.59) and (125.77,62.26) .. (125.77,81.59) .. controls (125.77,100.88) and (108.64,116.53) .. (87.47,116.59) -- (87.34,81.59) -- cycle ; \draw  [color={rgb, 255:red, 0; green, 0; blue, 0 }  ,draw opacity=1 ][line width=2.25]  (87.13,46.59) .. controls (87.2,46.59) and (87.27,46.59) .. (87.34,46.59) .. controls (108.57,46.59) and (125.77,62.26) .. (125.77,81.59) .. controls (125.77,100.88) and (108.64,116.53) .. (87.47,116.59) ;  
%Shape: Arc [id:dp3333496576896178] 
\draw  [draw opacity=0][dash pattern={on 5.63pt off 4.5pt}][line width=1.5]  (88.97,123.87) .. controls (113.88,123.08) and (133.82,104.46) .. (133.82,81.59) .. controls (133.82,58.23) and (113.01,39.29) .. (87.34,39.29) .. controls (86.91,39.29) and (86.48,39.3) .. (86.05,39.31) -- (87.34,81.59) -- cycle ; \draw [color={rgb, 255:red, 0; green, 0; blue, 0 }  ,draw opacity=1 ][dash pattern={on 5.63pt off 4.5pt}][line width=1.5]  [dash pattern={on 5.63pt off 4.5pt}]  (88.97,123.87) .. controls (113.88,123.08) and (133.82,104.46) .. (133.82,81.59) .. controls (133.82,58.23) and (113.01,39.29) .. (87.34,39.29) .. controls (86.91,39.29) and (86.48,39.3) .. (86.05,39.31) ;  
%Shape: Arc [id:dp8390505060217666] 
\draw  [draw opacity=0][line width=2.25]  (207.14,116.54) .. controls (185.93,116.52) and (168.75,100.86) .. (168.75,81.54) .. controls (168.75,62.23) and (185.93,46.57) .. (207.14,46.54) -- (207.19,81.54) -- cycle ; \draw  [color={rgb, 255:red, 0; green, 0; blue, 0 }  ,draw opacity=1 ][line width=2.25]  (207.14,116.54) .. controls (185.93,116.52) and (168.75,100.86) .. (168.75,81.54) .. controls (168.75,62.23) and (185.93,46.57) .. (207.14,46.54) ;  
%Straight Lines [id:da09054739808662116] 
\draw    (69.25,108.35) -- (92.5,110.75) ;
%Straight Lines [id:da2887461789296373] 
\draw    (204,111) -- (227.5,106) ;
%Straight Lines [id:da27729487042147305] 
\draw    (173.75,117.75) -- (163,128.6) ;
%Shape: Arc [id:dp004813584382169278] 
\draw  [draw opacity=0][dash pattern={on 5.63pt off 4.5pt}][line width=1.5]  (205.56,39.27) .. controls (180.64,40.05) and (160.71,58.68) .. (160.71,81.54) .. controls (160.71,104.9) and (181.52,123.84) .. (207.19,123.84) .. controls (207.62,123.84) and (208.05,123.84) .. (208.47,123.83) -- (207.19,81.54) -- cycle ; \draw [color={rgb, 255:red, 0; green, 0; blue, 0 }  ,draw opacity=1 ][dash pattern={on 5.63pt off 4.5pt}][line width=1.5]  [dash pattern={on 5.63pt off 4.5pt}]  (205.56,39.27) .. controls (180.64,40.05) and (160.71,58.68) .. (160.71,81.54) .. controls (160.71,104.9) and (181.52,123.84) .. (207.19,123.84) .. controls (207.62,123.84) and (208.05,123.84) .. (208.47,123.83) ;  
%Straight Lines [id:da4674118280722672] 
\draw    (125.5,116.65) -- (134.75,128.15) ;

% Text Node
\draw (136.75,131.55) node [anchor=north west][inner sep=0.75pt]    {$[ c]\mathcal{_{B'}}$};
% Text Node
\draw (31.9,100.35) node [anchor=north west][inner sep=0.75pt]    {$[ w]_{\mathcal{V} '}$};
% Text Node
\draw (231.4,100.6) node [anchor=north west][inner sep=0.75pt]    {$[ w]_{\mathcal{V} '}$};

\end{tikzpicture} \\ \vspace{5mm}
 $\pack{G}$& $\pack{G}\ba e$ & $\pack{G} \mba e$ \\ \vspace{5mm}
 \tikzset{every picture/.style={line width=0.75pt}} %set default line width to 0.75pt        

\begin{tikzpicture}[x=0.5pt,y=0.5pt,yscale=-1,xscale=1]
%uncomment if require: \path (0,403); %set diagram left start at 0, and has height of 403

%Shape: Arc [id:dp47892203953159895] 
\draw  [draw opacity=0][line width=1.5]  (151.6,79.5) .. controls (151.55,79.5) and (151.5,79.5) .. (151.45,79.5) .. controls (130.22,79.5) and (113.01,95.17) .. (113.01,114.5) .. controls (113.01,133.83) and (130.22,149.5) .. (151.45,149.5) .. controls (151.5,149.5) and (151.55,149.5) .. (151.6,149.5) -- (151.45,114.5) -- cycle ; \draw  [draw opacity=0][line width=1.5]  (151.6,79.5) .. controls (151.55,79.5) and (151.5,79.5) .. (151.45,79.5) .. controls (130.22,79.5) and (113.01,95.17) .. (113.01,114.5) .. controls (113.01,133.83) and (130.22,149.5) .. (151.45,149.5) .. controls (151.5,149.5) and (151.55,149.5) .. (151.6,149.5) ;  
%Shape: Arc [id:dp21578575308256465] 
\draw  [draw opacity=0][line width=1.5]  (81.3,79.5) .. controls (81.35,79.5) and (81.4,79.5) .. (81.45,79.5) .. controls (102.68,79.5) and (119.89,95.17) .. (119.89,114.5) .. controls (119.89,133.83) and (102.68,149.5) .. (81.45,149.5) .. controls (81.4,149.5) and (81.35,149.5) .. (81.3,149.5) -- (81.45,114.5) -- cycle ; \draw  [draw opacity=0][line width=1.5]  (81.3,79.5) .. controls (81.35,79.5) and (81.4,79.5) .. (81.45,79.5) .. controls (102.68,79.5) and (119.89,95.17) .. (119.89,114.5) .. controls (119.89,133.83) and (102.68,149.5) .. (81.45,149.5) .. controls (81.4,149.5) and (81.35,149.5) .. (81.3,149.5) ;  
%Shape: Arc [id:dp8779694649377007] 
\draw  [draw opacity=0][line width=2.25]  (80.3,79.52) .. controls (80.68,79.51) and (81.06,79.5) .. (81.45,79.5) .. controls (97.3,79.5) and (110.91,88.24) .. (116.79,100.71) -- (81.45,114.5) -- cycle ; \draw  [line width=2.25]  (80.3,79.52) .. controls (80.68,79.51) and (81.06,79.5) .. (81.45,79.5) .. controls (97.3,79.5) and (110.91,88.24) .. (116.79,100.71) ;  
%Shape: Arc [id:dp1206055040744094] 
\draw  [draw opacity=0][line width=2.25]  (116.81,128.25) .. controls (110.94,140.75) and (97.32,149.5) .. (81.45,149.5) .. controls (80.99,149.5) and (80.54,149.5) .. (80.09,149.48) -- (81.45,114.5) -- cycle ; \draw  [line width=2.25]  (116.81,128.25) .. controls (110.94,140.75) and (97.32,149.5) .. (81.45,149.5) .. controls (80.99,149.5) and (80.54,149.5) .. (80.09,149.48) ;  
%Shape: Arc [id:dp19836771891269533] 
\draw  [draw opacity=0][line width=2.25]  (152.6,79.52) .. controls (152.21,79.51) and (151.83,79.5) .. (151.45,79.5) .. controls (135.6,79.5) and (121.99,88.24) .. (116.11,100.71) -- (151.45,114.5) -- cycle ; \draw  [line width=2.25]  (152.6,79.52) .. controls (152.21,79.51) and (151.83,79.5) .. (151.45,79.5) .. controls (135.6,79.5) and (121.99,88.24) .. (116.11,100.71) ;  
%Shape: Arc [id:dp15769465954872297] 
\draw  [draw opacity=0][line width=2.25]  (116.49,129.08) .. controls (122.57,141.13) and (135.93,149.5) .. (151.45,149.5) .. controls (151.9,149.5) and (152.35,149.5) .. (152.8,149.48) -- (151.45,114.5) -- cycle ; \draw  [line width=2.25]  (116.49,129.08) .. controls (122.57,141.13) and (135.93,149.5) .. (151.45,149.5) .. controls (151.9,149.5) and (152.35,149.5) .. (152.8,149.48) ;  
%Shape: Arc [id:dp3581336218708817] 
\draw  [draw opacity=0][dash pattern={on 5.63pt off 4.5pt}][line width=1.5]  (116.29,89.43) .. controls (109.51,77.84) and (96.13,69.95) .. (80.74,69.95) .. controls (80.62,69.95) and (80.51,69.95) .. (80.4,69.95) -- (80.74,106.75) -- cycle ; \draw [color={rgb, 255:red, 0; green, 0; blue, 0 }  ,draw opacity=1 ][dash pattern={on 5.63pt off 4.5pt}][line width=1.5]  [dash pattern={on 5.63pt off 4.5pt}]  (116.29,89.43) .. controls (109.51,77.84) and (96.13,69.95) .. (80.74,69.95) .. controls (80.62,69.95) and (80.51,69.95) .. (80.4,69.95) ;  
%Shape: Arc [id:dp628238313674974] 
\draw  [draw opacity=0][dash pattern={on 5.63pt off 4.5pt}][line width=1.5]  (116.29,89.43) .. controls (123.12,77.98) and (136.42,70.21) .. (151.7,70.21) .. controls (152.02,70.21) and (152.34,70.21) .. (152.65,70.22) -- (151.7,107) -- cycle ; \draw [color={rgb, 255:red, 0; green, 0; blue, 0 }  ,draw opacity=1 ][dash pattern={on 5.63pt off 4.5pt}][line width=1.5]  [dash pattern={on 5.63pt off 4.5pt}]  (116.29,89.43) .. controls (123.12,77.98) and (136.42,70.21) .. (151.7,70.21) .. controls (152.02,70.21) and (152.34,70.21) .. (152.65,70.22) ;  
%Shape: Arc [id:dp19403686072773618] 
\draw  [draw opacity=0][dash pattern={on 5.63pt off 4.5pt}][line width=1.5]  (115.93,140.57) .. controls (122.71,152.16) and (136.09,160.04) .. (151.49,160.04) .. controls (151.6,160.04) and (151.71,160.04) .. (151.82,160.04) -- (151.49,123.25) -- cycle ; \draw [color={rgb, 255:red, 0; green, 0; blue, 0 }  ,draw opacity=1 ][dash pattern={on 5.63pt off 4.5pt}][line width=1.5]  [dash pattern={on 5.63pt off 4.5pt}]  (115.93,140.57) .. controls (122.71,152.16) and (136.09,160.04) .. (151.49,160.04) .. controls (151.6,160.04) and (151.71,160.04) .. (151.82,160.04) ;  
%Shape: Arc [id:dp89003048824706] 
\draw  [draw opacity=0][dash pattern={on 5.63pt off 4.5pt}][line width=1.5]  (115.93,140.57) .. controls (109.1,152.02) and (95.8,159.79) .. (80.52,159.79) .. controls (80.2,159.79) and (79.88,159.78) .. (79.57,159.78) -- (80.52,122.99) -- cycle ; \draw [color={rgb, 255:red, 0; green, 0; blue, 0 }  ,draw opacity=1 ][dash pattern={on 5.63pt off 4.5pt}][line width=1.5]  [dash pattern={on 5.63pt off 4.5pt}]  (115.93,140.57) .. controls (109.1,152.02) and (95.8,159.79) .. (80.52,159.79) .. controls (80.2,159.79) and (79.88,159.78) .. (79.57,159.78) ;  
%Straight Lines [id:da4722141352187338] 
\draw    (148,138) -- (171.5,133) ;
%Straight Lines [id:da03353734029155453] 
\draw    (61.25,135.35) -- (84.5,137.75) ;
%Straight Lines [id:da46003987714827166] 
\draw    (114.75,60.65) -- (114.75,74.9) ;
%Straight Lines [id:da027095212728795004] 
\draw    (114,158.5) -- (114,172.75) ;

% Text Node
\draw (178.4,127.6) node [anchor=north west][inner sep=0.75pt]    {$[ w]_{\mathcal{V} ''}$};
% Text Node
\draw (10.9,127.35) node [anchor=north west][inner sep=0.75pt]    {$[ w]_{\mathcal{V} ''}$};
% Text Node
\draw (103.9,33.4) node [anchor=north west][inner sep=0.75pt]    {$[ a]\mathcal{_{B}}$};
% Text Node
\draw (102.9,176.85) node [anchor=north west][inner sep=0.75pt]    {$[ b]_{\mathcal{B}}$};

\end{tikzpicture} & \tikzset{every picture/.style={line width=0.75pt}} %set default line width to 0.75pt        

\begin{tikzpicture}[x=0.5pt,y=0.5pt,yscale=-1,xscale=1]
%uncomment if require: \path (0,403); %set diagram left start at 0, and has height of 403

%Shape: Arc [id:dp47892203953159895] 
\draw  [draw opacity=0][line width=1.5]  (151.6,79.5) .. controls (151.55,79.5) and (151.5,79.5) .. (151.45,79.5) .. controls (130.22,79.5) and (113.01,95.17) .. (113.01,114.5) .. controls (113.01,133.83) and (130.22,149.5) .. (151.45,149.5) .. controls (151.5,149.5) and (151.55,149.5) .. (151.6,149.5) -- (151.45,114.5) -- cycle ; \draw  [draw opacity=0][line width=1.5]  (151.6,79.5) .. controls (151.55,79.5) and (151.5,79.5) .. (151.45,79.5) .. controls (130.22,79.5) and (113.01,95.17) .. (113.01,114.5) .. controls (113.01,133.83) and (130.22,149.5) .. (151.45,149.5) .. controls (151.5,149.5) and (151.55,149.5) .. (151.6,149.5) ;  
%Shape: Arc [id:dp21578575308256465] 
\draw  [draw opacity=0][line width=1.5]  (81.3,79.5) .. controls (81.35,79.5) and (81.4,79.5) .. (81.45,79.5) .. controls (102.68,79.5) and (119.89,95.17) .. (119.89,114.5) .. controls (119.89,133.83) and (102.68,149.5) .. (81.45,149.5) .. controls (81.4,149.5) and (81.35,149.5) .. (81.3,149.5) -- (81.45,114.5) -- cycle ; \draw  [draw opacity=0][line width=1.5]  (81.3,79.5) .. controls (81.35,79.5) and (81.4,79.5) .. (81.45,79.5) .. controls (102.68,79.5) and (119.89,95.17) .. (119.89,114.5) .. controls (119.89,133.83) and (102.68,149.5) .. (81.45,149.5) .. controls (81.4,149.5) and (81.35,149.5) .. (81.3,149.5) ;  
%Shape: Arc [id:dp8779694649377007] 
\draw  [draw opacity=0][line width=2.25]  (80.3,79.52) .. controls (80.68,79.51) and (81.06,79.5) .. (81.45,79.5) .. controls (97.3,79.5) and (110.91,88.24) .. (116.79,100.71) -- (81.45,114.5) -- cycle ; \draw  [line width=2.25]  (80.3,79.52) .. controls (80.68,79.51) and (81.06,79.5) .. (81.45,79.5) .. controls (97.3,79.5) and (110.91,88.24) .. (116.79,100.71) ;  
%Shape: Arc [id:dp1206055040744094] 
\draw  [draw opacity=0][line width=2.25]  (116.81,128.25) .. controls (110.94,140.75) and (97.32,149.5) .. (81.45,149.5) .. controls (80.99,149.5) and (80.54,149.5) .. (80.09,149.48) -- (81.45,114.5) -- cycle ; \draw  [line width=2.25]  (116.81,128.25) .. controls (110.94,140.75) and (97.32,149.5) .. (81.45,149.5) .. controls (80.99,149.5) and (80.54,149.5) .. (80.09,149.48) ;  
%Shape: Arc [id:dp19836771891269533] 
\draw  [draw opacity=0][line width=2.25]  (152.6,79.52) .. controls (152.21,79.51) and (151.83,79.5) .. (151.45,79.5) .. controls (135.6,79.5) and (121.99,88.24) .. (116.11,100.71) -- (151.45,114.5) -- cycle ; \draw  [line width=2.25]  (152.6,79.52) .. controls (152.21,79.51) and (151.83,79.5) .. (151.45,79.5) .. controls (135.6,79.5) and (121.99,88.24) .. (116.11,100.71) ;  
%Shape: Arc [id:dp15769465954872297] 
\draw  [draw opacity=0][line width=2.25]  (116.49,129.08) .. controls (122.57,141.13) and (135.93,149.5) .. (151.45,149.5) .. controls (151.9,149.5) and (152.35,149.5) .. (152.8,149.48) -- (151.45,114.5) -- cycle ; \draw  [line width=2.25]  (116.49,129.08) .. controls (122.57,141.13) and (135.93,149.5) .. (151.45,149.5) .. controls (151.9,149.5) and (152.35,149.5) .. (152.8,149.48) ;  
%Shape: Arc [id:dp3581336218708817] 
\draw  [draw opacity=0][dash pattern={on 5.63pt off 4.5pt}][line width=1.5]  (116.29,89.43) .. controls (109.51,77.84) and (96.13,69.95) .. (80.74,69.95) .. controls (80.62,69.95) and (80.51,69.95) .. (80.4,69.95) -- (80.74,106.75) -- cycle ; \draw [color={rgb, 255:red, 0; green, 0; blue, 0 }  ,draw opacity=1 ][dash pattern={on 5.63pt off 4.5pt}][line width=1.5]  [dash pattern={on 5.63pt off 4.5pt}]  (116.29,89.43) .. controls (109.51,77.84) and (96.13,69.95) .. (80.74,69.95) .. controls (80.62,69.95) and (80.51,69.95) .. (80.4,69.95) ;  
%Shape: Arc [id:dp628238313674974] 
\draw  [draw opacity=0][dash pattern={on 5.63pt off 4.5pt}][line width=1.5]  (116.29,89.43) .. controls (123.12,77.98) and (136.42,70.21) .. (151.7,70.21) .. controls (152.02,70.21) and (152.34,70.21) .. (152.65,70.22) -- (151.7,107) -- cycle ; \draw [color={rgb, 255:red, 0; green, 0; blue, 0 }  ,draw opacity=1 ][dash pattern={on 5.63pt off 4.5pt}][line width=1.5]  [dash pattern={on 5.63pt off 4.5pt}]  (116.29,89.43) .. controls (123.12,77.98) and (136.42,70.21) .. (151.7,70.21) .. controls (152.02,70.21) and (152.34,70.21) .. (152.65,70.22) ;  
%Shape: Arc [id:dp19403686072773618] 
\draw  [draw opacity=0][dash pattern={on 5.63pt off 4.5pt}][line width=1.5]  (115.93,140.57) .. controls (122.71,152.16) and (136.09,160.04) .. (151.49,160.04) .. controls (151.6,160.04) and (151.71,160.04) .. (151.82,160.04) -- (151.49,123.25) -- cycle ; \draw [color={rgb, 255:red, 0; green, 0; blue, 0 }  ,draw opacity=1 ][dash pattern={on 5.63pt off 4.5pt}][line width=1.5]  [dash pattern={on 5.63pt off 4.5pt}]  (115.93,140.57) .. controls (122.71,152.16) and (136.09,160.04) .. (151.49,160.04) .. controls (151.6,160.04) and (151.71,160.04) .. (151.82,160.04) ;  
%Shape: Arc [id:dp89003048824706] 
\draw  [draw opacity=0][dash pattern={on 5.63pt off 4.5pt}][line width=1.5]  (115.93,140.57) .. controls (109.1,152.02) and (95.8,159.79) .. (80.52,159.79) .. controls (80.2,159.79) and (79.88,159.78) .. (79.57,159.78) -- (80.52,122.99) -- cycle ; \draw [color={rgb, 255:red, 0; green, 0; blue, 0 }  ,draw opacity=1 ][dash pattern={on 5.63pt off 4.5pt}][line width=1.5]  [dash pattern={on 5.63pt off 4.5pt}]  (115.93,140.57) .. controls (109.1,152.02) and (95.8,159.79) .. (80.52,159.79) .. controls (80.2,159.79) and (79.88,159.78) .. (79.57,159.78) ;  
%Straight Lines [id:da4722141352187338] 
\draw    (148,138) -- (171.5,133) ;
%Straight Lines [id:da03353734029155453] 
\draw    (61.25,135.35) -- (84.5,137.75) ;
%Straight Lines [id:da46003987714827166] 
\draw    (114.75,60.65) -- (114.75,74.9) ;
%Straight Lines [id:da027095212728795004] 
\draw    (114,158.5) -- (114,172.75) ;

% Text Node
\draw (178.4,127.6) node [anchor=north west][inner sep=0.75pt]    {$[ w]_{\mathcal{V} ''}$};
% Text Node
\draw (10.9,127.35) node [anchor=north west][inner sep=0.75pt]    {$[ w]_{\mathcal{V} ''}$};
% Text Node
\draw (103.9,33.4) node [anchor=north west][inner sep=0.75pt]    {$[ c]\mathcal{_{B''}}$};
% Text Node
\draw (102.9,176.85) node [anchor=north west][inner sep=0.75pt]    {$[ c]_{\mathcal{B} ''}$};

\end{tikzpicture} &
 \tikzset{every picture/.style={line width=0.75pt}} %set default line width to 0.75pt        

\begin{tikzpicture}[x=0.5pt,y=0.5pt,yscale=-1,xscale=1]
%uncomment if require: \path (0,403); %set diagram left start at 0, and has height of 403

%Shape: Arc [id:dp6920209365616659] 
\draw  [draw opacity=0][line width=1.5]  (67.3,66.5) .. controls (67.35,66.5) and (67.4,66.5) .. (67.45,66.5) .. controls (88.68,66.5) and (105.89,82.17) .. (105.89,101.5) .. controls (105.89,120.83) and (88.68,136.5) .. (67.45,136.5) .. controls (67.4,136.5) and (67.35,136.5) .. (67.3,136.5) -- (67.45,101.5) -- cycle ; \draw  [draw opacity=0][line width=1.5]  (67.3,66.5) .. controls (67.35,66.5) and (67.4,66.5) .. (67.45,66.5) .. controls (88.68,66.5) and (105.89,82.17) .. (105.89,101.5) .. controls (105.89,120.83) and (88.68,136.5) .. (67.45,136.5) .. controls (67.4,136.5) and (67.35,136.5) .. (67.3,136.5) ;  
%Shape: Arc [id:dp8320259027437764] 
\draw  [draw opacity=0][line width=1.5]  (164.6,65.5) .. controls (164.55,65.5) and (164.5,65.5) .. (164.45,65.5) .. controls (143.22,65.5) and (126.01,81.17) .. (126.01,100.5) .. controls (126.01,119.83) and (143.22,135.5) .. (164.45,135.5) .. controls (164.5,135.5) and (164.55,135.5) .. (164.6,135.5) -- (164.45,100.5) -- cycle ; \draw  [draw opacity=0][line width=1.5]  (164.6,65.5) .. controls (164.55,65.5) and (164.5,65.5) .. (164.45,65.5) .. controls (143.22,65.5) and (126.01,81.17) .. (126.01,100.5) .. controls (126.01,119.83) and (143.22,135.5) .. (164.45,135.5) .. controls (164.5,135.5) and (164.55,135.5) .. (164.6,135.5) ;  
%Straight Lines [id:da7044018816844937] 
\draw [line width=2.25]    (118.6,104.05) -- (133.99,115.83) ;
%Straight Lines [id:da578952161897081] 
\draw [line width=2.25]    (97.04,87.52) -- (104.85,93.8) ;
%Straight Lines [id:da8492182227245817] 
\draw [color={rgb, 255:red, 0; green, 0; blue, 0 }  ,draw opacity=1 ][line width=1.5]  [dash pattern={on 5.63pt off 4.5pt}]  (102.17,81.45) -- (111.6,89.3) ;
%Shape: Arc [id:dp9492213236371121] 
\draw  [draw opacity=0][line width=2.25]  (60.3,65.77) .. controls (60.68,65.76) and (61.06,65.75) .. (61.45,65.75) .. controls (77.92,65.75) and (91.97,75.18) .. (97.44,88.44) -- (61.45,100.75) -- cycle ; \draw  [line width=2.25]  (60.3,65.77) .. controls (60.68,65.76) and (61.06,65.75) .. (61.45,65.75) .. controls (77.92,65.75) and (91.97,75.18) .. (97.44,88.44) ;  
%Shape: Arc [id:dp9878456372527857] 
\draw  [draw opacity=0][line width=2.25]  (97.72,112.35) .. controls (92.47,125.98) and (78.21,135.75) .. (61.45,135.75) .. controls (60.99,135.75) and (60.54,135.75) .. (60.09,135.73) -- (61.45,100.75) -- cycle ; \draw  [line width=2.25]  (97.72,112.35) .. controls (92.47,125.98) and (78.21,135.75) .. (61.45,135.75) .. controls (60.99,135.75) and (60.54,135.75) .. (60.09,135.73) ;  
%Shape: Arc [id:dp17061596681554958] 
\draw  [draw opacity=0][line width=2.25]  (169.85,66.27) .. controls (169.46,66.26) and (169.08,66.25) .. (168.7,66.25) .. controls (152.24,66.25) and (138.19,75.67) .. (132.72,88.92) -- (168.7,101.25) -- cycle ; \draw  [line width=2.25]  (169.85,66.27) .. controls (169.46,66.26) and (169.08,66.25) .. (168.7,66.25) .. controls (152.24,66.25) and (138.19,75.67) .. (132.72,88.92) ;  
%Shape: Arc [id:dp1489727381904833] 
\draw  [draw opacity=0][line width=2.25]  (133.24,114.79) .. controls (139.05,127.4) and (152.74,136.25) .. (168.7,136.25) .. controls (169.15,136.25) and (169.6,136.25) .. (170.05,136.23) -- (168.7,101.25) -- cycle ; \draw  [line width=2.25]  (133.24,114.79) .. controls (139.05,127.4) and (152.74,136.25) .. (168.7,136.25) .. controls (169.15,136.25) and (169.6,136.25) .. (170.05,136.23) ;  
%Shape: Arc [id:dp5162733571190294] 
\draw  [draw opacity=0][dash pattern={on 5.63pt off 4.5pt}][line width=1.5]  (102.17,81.45) .. controls (96.65,67.15) and (81.71,56.9) .. (64.16,56.9) .. controls (63.63,56.9) and (63.11,56.91) .. (62.58,56.92) -- (64.16,93.69) -- cycle ; \draw [color={rgb, 255:red, 0; green, 0; blue, 0 }  ,draw opacity=1 ][dash pattern={on 5.63pt off 4.5pt}][line width=1.5]  [dash pattern={on 5.63pt off 4.5pt}]  (102.17,81.45) .. controls (96.65,67.15) and (81.71,56.9) .. (64.16,56.9) .. controls (63.63,56.9) and (63.11,56.91) .. (62.58,56.92) ;  
%Shape: Arc [id:dp9256187953552624] 
\draw  [draw opacity=0][dash pattern={on 5.63pt off 4.5pt}][line width=1.5]  (129.98,81.6) .. controls (136.18,68.57) and (150.4,59.46) .. (166.95,59.46) .. controls (167.9,59.46) and (168.85,59.49) .. (169.79,59.55) -- (166.95,96.25) -- cycle ; \draw [color={rgb, 255:red, 0; green, 0; blue, 0 }  ,draw opacity=1 ][dash pattern={on 5.63pt off 4.5pt}][line width=1.5]  [dash pattern={on 5.63pt off 4.5pt}]  (129.98,81.6) .. controls (136.18,68.57) and (150.4,59.46) .. (166.95,59.46) .. controls (167.9,59.46) and (168.85,59.49) .. (169.79,59.55) ;  
%Shape: Arc [id:dp13433606784223517] 
\draw  [draw opacity=0][dash pattern={on 5.63pt off 4.5pt}][line width=1.5]  (170.2,130.74) .. controls (156.4,129.49) and (144.79,121.8) .. (139.37,111.17) -- (174.16,97.64) -- cycle ; \draw [color={rgb, 255:red, 0; green, 0; blue, 0 }  ,draw opacity=1 ][dash pattern={on 5.63pt off 4.5pt}][line width=1.5]  [dash pattern={on 5.63pt off 4.5pt}]  (170.2,130.74) .. controls (156.4,129.49) and (144.79,121.8) .. (139.37,111.17) ;  
%Straight Lines [id:da9946390441178703] 
\draw [line width=2.25]    (97.21,113.61) -- (133.11,88) ;
%Straight Lines [id:da9414716692969469] 
\draw [color={rgb, 255:red, 0; green, 0; blue, 0 }  ,draw opacity=1 ][line width=1.5]  [dash pattern={on 5.63pt off 4.5pt}]  (95.03,106.25) -- (129.23,80.35) ;
%Straight Lines [id:da016221986706674185] 
\draw [color={rgb, 255:red, 0; green, 0; blue, 0 }  ,draw opacity=1 ][line width=1.5]  [dash pattern={on 5.63pt off 4.5pt}]  (139.37,111.17) -- (125.1,99.55) ;
%Shape: Arc [id:dp29748430272727233] 
\draw  [draw opacity=0][dash pattern={on 5.63pt off 4.5pt}][line width=1.5]  (60.29,128.97) .. controls (76.51,128.36) and (90.04,118.97) .. (95.03,106.25) -- (59.02,95.72) -- cycle ; \draw [color={rgb, 255:red, 0; green, 0; blue, 0 }  ,draw opacity=1 ][dash pattern={on 5.63pt off 4.5pt}][line width=1.5]  [dash pattern={on 5.63pt off 4.5pt}]  (60.29,128.97) .. controls (76.51,128.36) and (90.04,118.97) .. (95.03,106.25) ;  
%Straight Lines [id:da10293779646996348] 
\draw    (122.25,143.25) -- (137,130.6) ;
%Straight Lines [id:da41745771515060237] 
\draw    (72,140.6) -- (85,152.6) ;
%Straight Lines [id:da9365956332769692] 
\draw    (129,46.4) -- (140,55.6) ;
%Straight Lines [id:da015927255636972726] 
\draw    (91,43.6) -- (79,50.6) ;

% Text Node
\draw (91.9,17.4) node [anchor=north west][inner sep=0.75pt]    {$[ c]\mathcal{_{B'''}}$};
% Text Node
\draw (89.9,146.25) node [anchor=north west][inner sep=0.75pt]    {$[ w]_{\mathcal{V} '''}$};

\end{tikzpicture}\\ \vspace{5mm}
   $\pack{G}\con e$& $\pack{G}\mcon e$ & $\pack{G} \pcon e$
\end{tabular}
\caption{Operations on an edge $e$ of a packaged arrow presentation $\pack{G}=(\ar{G}, \V,\B)$. Here $[u]_{\V}$ and $[v]_{\V}$ may be equal, as may $[a]_{\B}$ and $[b]_{\B}$.}
\label{tab:edgeops}
\end{table}

For defining the five operations, let $\pack{G}=(\ar{G}, \V,\B)$ be a packaged arrow presentation containing an edge  $e$. Let $u$ and $v$ be the vertices incident to $e$ ($u$ may equal $v$), $[u]$ and $[v]$ be their vertex classes ($[u]$ may equal $[v]$); and let $a$ and $b$ be the boundary components incident to $e$ ($a$ may equal $b$), and $[a]$ and $[b]$ be their boundary classes ($[a]$ may equal $[b]$).
Observe the vertices of $\ar{G}$ and $\ar{G}\ba e$ are naturally identified. 
For $\ar{G}/ e$ (respectively, $\ar{G}\pcon e$), each vertex of $\ar{G}$ other than $u$ and $v$ is naturally identified with a vertex in $\ar{G}/ e$ (respectively, $\ar{G}\pcon e$), as these vertices are not altered. We say the remaining vertices  in $\ar{G}/ e$ (respectively, $\ar{G}\pcon e$) are \emph{created by the contraction} (respectively, \emph{created by the Penrose-contraction}).
Similarly, the boundary components of $\ar{G}$ and $\ar{G}/ e$ are naturally identified.
For $\ar{G}\ba e$ (respectively, $\ar{G}\pcon e$) each boundary component of $\ar{G}$ other than $a$ and $b$ is naturally identified with  a boundary component in $\ar{G}\ba e$  (respectively, $\ar{G}\pcon e$), as these boundary components are not altered. We say the remaining boundary components  in $\ar{G}\ba  e$  (respectively, $\ar{G}\pcon e$) are \emph{created by the deletion} (respectively, \emph{created by the Penrose-contraction}).

\begin{itemize}

\item Let $S$ be the set of boundary components of  $\ar{G}\ba e$ created by the deletion, and let 
\[ \V' = (\V-\{ [u],[v] \}) \cup \{[u]\cup [v]  \} ,\]
and
\[  \B' = (\B-\{ [a],[b] \}) \cup \{[a]\cup[b]\cup S -\{a,b\} \}.\]
Then
\begin{itemize}
\item $\pack{G}$ \emph{delete} $e$, denoted $\pack{G}\ba e$, is the packaged arrow presentation $(\ar{G}\ba e, \V,\B')$; and 
\item $\pack{G}$ \emph{merge-delete} $e$, denoted $\pack{G}\mba e$,  is the packaged arrow presentation $(\ar{G}\ba e, \V',\B')$.
\end{itemize}

\item Let $T$ be the set of vertices of  $\ar{G}/ e$ created by the contraction, and let 
\[ \V'' = (\V-\{ [u],[v] \}) \cup \{[u]\cup [v] \cup T -\{u,v\} \} ,\]
and
\[ \B'' = (\B-\{ [a],[b] \}) \cup \{[a]\cup [b]  \} .\]
Then
\begin{itemize}
\item $\pack{G}$ \emph{contract} $e$, denoted $\pack{G}/ e$,  is the packaged arrow presentation $(\ar{G}/ e, \V'',\B)$; and

\item $\pack{G}$ \emph{merge-contract} $e$, denoted $\pack{G}\mcon e$,  is the packaged arrow presentation $(\ar{G}/ e, \V'',\B'')$.

\end{itemize}

 \item $\pack{G}$ \emph{Penrose-contract} $e$, denoted $\pack{G}\pcon e$,  is the packaged arrow presentation $(\ar{G}\pcon e, \V''',\B'')$ where
 \[ \V''' = (\V-\{ [u],[v] \}) \cup \{[u]\cup [v] \cup T -\{u,v\}  \}, \]
where  $T$ is the set of vertices of  $\ar{G}\pcon e$ created by the Penrose-contraction; and 
 \[\B''' = (\B-\{ [a],[b] \}) \cup \{[a]\cup[b]\cup S -\{a,b\} \},\]
where  $S$ is the set of boundary components of  $\ar{G}\pcon e$ created by the Penrose-contraction. 
\end{itemize}

\begin{figure}
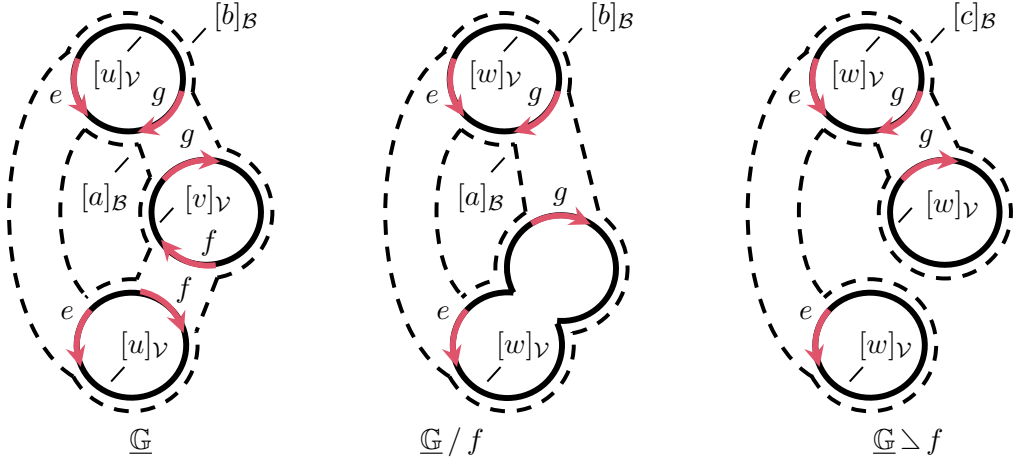

    \centering
    \begin{tabular}{ccc}
      \input{Diagrams/ExPack}\qquad&\qquad
      \input{Diagrams/ExPackCon}\qquad & \qquad\tikzset{every picture/.style={line width=0.75pt}} %set default line width to 0.75pt        

\begin{tikzpicture}[x=0.5pt,y=0.5pt,yscale=-1,xscale=1]
%uncomment if require: \path (0,403); %set diagram left start at 0, and has height of 403

%Shape: Ellipse [id:dp9518262786349504] 
\draw  [line width=2.25]  (130.14,53.5) .. controls (152.27,47.99) and (174.87,60.75) .. (180.61,82) .. controls (186.35,103.24) and (173.06,124.93) .. (150.93,130.44) .. controls (128.79,135.95) and (106.2,123.2) .. (100.46,101.95) .. controls (94.71,80.7) and (108,59.01) .. (130.14,53.5) -- cycle ;
%Shape: Arc [id:dp19607057061842448] 
\draw  [draw opacity=0][line width=2.25]  (102.55,76.64) .. controls (99.44,84.25) and (98.5,92.84) .. (100.34,101.43) .. controls (101.93,108.84) and (105.39,115.33) .. (110.11,120.49) -- (140.18,92.05) -- cycle ; \draw [color={rgb, 255:red, 223; green, 83; blue, 107 }  ,draw opacity=1 ][line width=2.25]    (102.55,76.64) .. controls (99.44,84.25) and (98.5,92.84) .. (100.34,101.43) .. controls (101.56,107.1) and (103.87,112.23) .. (107,116.63) ; \draw [shift={(110.11,120.49)}, rotate = 235.65] [fill={rgb, 255:red, 223; green, 83; blue, 107 }  ,fill opacity=1 ][line width=0.08]  [draw opacity=0] (16.07,-7.72) -- (0,0) -- (16.07,7.72) -- (10.67,0) -- cycle    ; 
%Shape: Ellipse [id:dp8428397282831787] 
\draw  [line width=2.25]  (154.18,331.85) .. controls (132.17,337.79) and (109.31,325.47) .. (103.12,304.34) .. controls (96.93,283.21) and (109.76,261.26) .. (131.77,255.32) .. controls (153.79,249.39) and (176.65,261.7) .. (182.83,282.83) .. controls (189.02,303.96) and (176.19,325.91) .. (154.18,331.85) -- cycle ;
%Shape: Arc [id:dp048031319744519685] 
\draw  [draw opacity=0][line width=2.25]  (111.93,266.92) .. controls (106.71,272.87) and (103.11,280.31) .. (101.92,288.65) .. controls (100.85,296.22) and (101.91,303.56) .. (104.65,310.07) -- (142.98,293.59) -- cycle ; \draw [color={rgb, 255:red, 223; green, 83; blue, 107 }  ,draw opacity=1 ][line width=2.25]    (111.93,266.92) .. controls (106.71,272.87) and (103.11,280.31) .. (101.92,288.65) .. controls (101.1,294.48) and (101.53,300.17) .. (103.03,305.46) ; \draw [shift={(104.65,310.07)}, rotate = 255.4] [fill={rgb, 255:red, 223; green, 83; blue, 107 }  ,fill opacity=1 ][line width=0.08]  [draw opacity=0] (16.07,-7.72) -- (0,0) -- (16.07,7.72) -- (10.67,0) -- cycle    ; 
%Shape: Ellipse [id:dp7252914170444525] 
\draw  [line width=2.25]  (189.14,154.5) .. controls (211.27,148.99) and (233.87,161.75) .. (239.61,183) .. controls (245.35,204.24) and (232.06,225.93) .. (209.93,231.44) .. controls (187.79,236.95) and (165.2,224.2) .. (159.46,202.95) .. controls (153.71,181.7) and (167,160.01) .. (189.14,154.5) -- cycle ;
%Shape: Arc [id:dp2646523797808519] 
\draw  [draw opacity=0][line width=2.25]  (210.07,154.6) .. controls (201.89,152.62) and (192.97,153.02) .. (184.49,156.29) .. controls (177.59,158.95) and (171.82,163.18) .. (167.46,168.37) -- (199.57,193.08) -- cycle ; \draw [color={rgb, 255:red, 223; green, 83; blue, 107 }  ,draw opacity=1 ][line width=2.25]    (205.08,153.7) .. controls (198.34,152.9) and (191.28,153.68) .. (184.49,156.29) .. controls (177.59,158.95) and (171.82,163.18) .. (167.46,168.37) ;  \draw [shift={(210.07,154.6)}, rotate = 185.46] [fill={rgb, 255:red, 223; green, 83; blue, 107 }  ,fill opacity=1 ][line width=0.08]  [draw opacity=0] (16.07,-7.72) -- (0,0) -- (16.07,7.72) -- (10.67,0) -- cycle    ;
%Shape: Arc [id:dp31402647132034167] 
\draw  [draw opacity=0][line width=2.25]  (147.32,132.22) .. controls (155.96,130.81) and (164.23,126.73) .. (170.74,120.02) .. controls (176.05,114.54) and (179.5,108.01) .. (181.13,101.15) -- (140.53,91.97) -- cycle ; \draw [color={rgb, 255:red, 223; green, 83; blue, 107 }  ,draw opacity=1 ][line width=2.25]    (152.2,131.13) .. controls (159.04,129.17) and (165.47,125.46) .. (170.74,120.02) .. controls (176.05,114.54) and (179.5,108.01) .. (181.13,101.15) ;  \draw [shift={(147.32,132.22)}, rotate = 342.57] [fill={rgb, 255:red, 223; green, 83; blue, 107 }  ,fill opacity=1 ][line width=0.08]  [draw opacity=0] (16.07,-7.72) -- (0,0) -- (16.07,7.72) -- (10.67,0) -- cycle    ;
%Shape: Arc [id:dp31853582270584235] 
\draw  [draw opacity=0][dash pattern={on 5.63pt off 4.5pt}][line width=1.5]  (98.5,316.44) .. controls (106.8,332.56) and (123.6,343.59) .. (142.98,343.59) .. controls (170.59,343.59) and (192.98,321.2) .. (192.98,293.59) .. controls (192.98,265.97) and (170.59,243.59) .. (142.98,243.59) .. controls (129.51,243.59) and (117.29,248.91) .. (108.3,257.56) -- (142.98,293.59) -- cycle ; \draw [color={rgb, 255:red, 0; green, 0; blue, 0 }  ,draw opacity=1 ][dash pattern={on 5.63pt off 4.5pt}][line width=1.5]  [dash pattern={on 5.63pt off 4.5pt}]  (98.5,316.44) .. controls (106.8,332.56) and (123.6,343.59) .. (142.98,343.59) .. controls (170.59,343.59) and (192.98,321.2) .. (192.98,293.59) .. controls (192.98,265.97) and (170.59,243.59) .. (142.98,243.59) .. controls (129.51,243.59) and (117.29,248.91) .. (108.3,257.56) ;  
%Straight Lines [id:da6140582467319334] 
\draw    (189.1,64.6) -- (199.1,53.6) ;
%Straight Lines [id:da7902046474405887] 
\draw    (127.6,321.1) -- (137.6,310.1) ;
%Straight Lines [id:da9828540891836116] 
\draw    (164.6,200.6) -- (174.6,189.6) ;
%Straight Lines [id:da516753519709517] 
\draw    (140.1,71.1) -- (150.1,60.1) ;
%Shape: Arc [id:dp012343703298803632] 
\draw  [draw opacity=0][dash pattern={on 5.63pt off 4.5pt}][line width=1.5]  (158.01,165.28) .. controls (152.68,173.23) and (149.57,182.79) .. (149.57,193.08) .. controls (149.57,220.69) and (171.96,243.08) .. (199.57,243.08) .. controls (227.19,243.08) and (249.57,220.69) .. (249.57,193.08) .. controls (249.57,169) and (232.56,148.9) .. (209.89,144.14) -- (199.57,193.08) -- cycle ; \draw [color={rgb, 255:red, 0; green, 0; blue, 0 }  ,draw opacity=1 ][dash pattern={on 5.63pt off 4.5pt}][line width=1.5]  [dash pattern={on 5.63pt off 4.5pt}]  (158.01,165.28) .. controls (152.68,173.23) and (149.57,182.79) .. (149.57,193.08) .. controls (149.57,220.69) and (171.96,243.08) .. (199.57,243.08) .. controls (227.19,243.08) and (249.57,220.69) .. (249.57,193.08) .. controls (249.57,169) and (232.56,148.9) .. (209.89,144.14) ;  
%Shape: Arc [id:dp4744967213371173] 
\draw  [draw opacity=0][dash pattern={on 5.63pt off 4.5pt}][line width=1.5]  (189.16,102.18) .. controls (189.83,98.91) and (190.18,95.52) .. (190.18,92.05) .. controls (190.18,64.44) and (167.8,42.05) .. (140.18,42.05) .. controls (119.46,42.05) and (101.67,54.67) .. (94.09,72.63) -- (140.18,92.05) -- cycle ; \draw [color={rgb, 255:red, 0; green, 0; blue, 0 }  ,draw opacity=1 ][dash pattern={on 5.63pt off 4.5pt}][line width=1.5]  [dash pattern={on 5.63pt off 4.5pt}]  (189.16,102.18) .. controls (189.83,98.91) and (190.18,95.52) .. (190.18,92.05) .. controls (190.18,64.44) and (167.8,42.05) .. (140.18,42.05) .. controls (119.46,42.05) and (101.67,54.67) .. (94.09,72.63) ;  
%Shape: Arc [id:dp09076846553542883] 
\draw  [draw opacity=0][dash pattern={on 5.63pt off 4.5pt}][line width=1.5]  (107.83,130.18) .. controls (116.55,137.59) and (127.84,142.05) .. (140.18,142.05) .. controls (142.99,142.05) and (145.74,141.82) .. (148.43,141.38) -- (140.18,92.05) -- cycle ; \draw [color={rgb, 255:red, 0; green, 0; blue, 0 }  ,draw opacity=1 ][dash pattern={on 5.63pt off 4.5pt}][line width=1.5]  [dash pattern={on 5.63pt off 4.5pt}]  (107.83,130.18) .. controls (116.55,137.59) and (127.84,142.05) .. (140.18,142.05) .. controls (142.99,142.05) and (145.74,141.82) .. (148.43,141.38) ;  
%Straight Lines [id:da9213611317722562] 
\draw [line width=1.5]  [dash pattern={on 5.63pt off 4.5pt}]  (148.43,141.38) -- (157.28,166.4) ;
%Straight Lines [id:da39280290548786845] 
\draw [line width=1.5]  [dash pattern={on 5.63pt off 4.5pt}]  (189.16,102.18) -- (209.89,144.14) ;
%Shape: Arc [id:dp20084677322203182] 
\draw  [draw opacity=0][dash pattern={on 5.63pt off 4.5pt}][line width=1.5]  (106.46,253.87) .. controls (96.15,242.26) and (89.13,219.25) .. (89.13,192.75) .. controls (89.13,165.1) and (96.78,141.24) .. (107.83,130.18) -- (121.71,192.75) -- cycle ; \draw [color={rgb, 255:red, 0; green, 0; blue, 0 }  ,draw opacity=1 ][dash pattern={on 5.63pt off 4.5pt}][line width=1.5]  [dash pattern={on 5.63pt off 4.5pt}]  (106.46,253.87) .. controls (96.15,242.26) and (89.13,219.25) .. (89.13,192.75) .. controls (89.13,165.1) and (96.78,141.24) .. (107.83,130.18) ;  
%Shape: Arc [id:dp2955615157735928] 
\draw  [draw opacity=0][dash pattern={on 5.63pt off 4.5pt}][line width=1.5]  (98.94,317.29) .. controls (70.77,300.59) and (50.36,251.39) .. (50.36,193.31) .. controls (50.36,137.12) and (69.45,89.26) .. (96.2,71.07) -- (120.09,193.31) -- cycle ; \draw [color={rgb, 255:red, 0; green, 0; blue, 0 }  ,draw opacity=1 ][dash pattern={on 5.63pt off 4.5pt}][line width=1.5]  [dash pattern={on 5.63pt off 4.5pt}]  (98.94,317.29) .. controls (70.77,300.59) and (50.36,251.39) .. (50.36,193.31) .. controls (50.36,137.12) and (69.45,89.26) .. (96.2,71.07) ;  

% Text Node
\draw (87,262.4) node [anchor=north west][inner sep=0.75pt]    {$e$};
% Text Node
\draw (78.5,98.8) node [anchor=north west][inner sep=0.75pt]    {$e$};
% Text Node
\draw (156,95.9) node [anchor=north west][inner sep=0.75pt]    {$g$};
% Text Node
\draw (177.64,128.4) node [anchor=north west][inner sep=0.75pt]    {$g$};
% Text Node
\draw (111,75.6) node [anchor=north west][inner sep=0.75pt]    {$[ w]_{\mathcal{V}}$};
% Text Node
\draw (202.75,31.4) node [anchor=north west][inner sep=0.75pt]    {$[ c]_{\mathcal{B}}$};
% Text Node
\draw (132,280.6) node [anchor=north west][inner sep=0.75pt]    {$[ w]_{\mathcal{V}}$};
% Text Node
\draw (182.5,174.1) node [anchor=north west][inner sep=0.75pt]    {$[ w]_{\mathcal{V}}$};

\end{tikzpicture} \\
    $ \pack{G}$& $\pack{G}\con f$\qquad\qquad & \qquad\qquad  $\pack{G}\mba f$
    \end{tabular}
    \caption{Contracting and merge-deleting the edge $f$.}
    \label{expackdelcon}
\end{figure}

\begin{lemma}\label{l.cmmt}
    The  operations of deletion, contraction, Penrose-contraction, merge-deletion and merge-contraction commute with each other and with one another when acting on different edges.
\end{lemma}

\begin{proof}
    As noted above,  deletion, contraction and Penrose-contraction commute for arrow presentations. All that remains is to check that the resulting partitions are the same. Checking this is a straightforward but lengthy calculation and we omit the details. 
\end{proof}

\subsection{2-sums and tensor products for packaged arrow presentations}

We shall now extend the definition of 2-sums to packaged arrow presentations. Although this 2-sum is in practice straightforward, the formal definition is rather cumbersome. This is because we need to  keep track of the  vertices and boundary components affected by the 2-sum. To aid the reader the naming of the vertex and boundary components used in the definition is indicated in Figure~\ref{f.dog1}, while Figure~\ref{f.dog2} indicates how the vertex and boundary partitions are determined. Note that in the definition  named vertices and boundary components need not be distinct (for example, $[u]_{\V}$ may equal $[v]_{\V}$, etc.). 
An example of a 2-sum of two packaged arrow presentations is given in Figure~\ref{expack2sum}.

\begin{figure}
\centering
\input{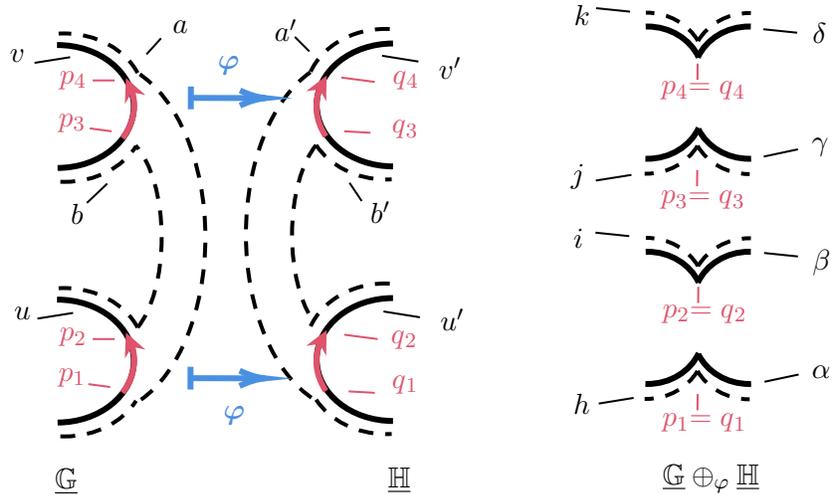}
\caption{Naming conventions used in Definition~\ref{d.big2sum}.}
\label{f.dog1}
\end{figure}

\begin{figure}
\centering
\input{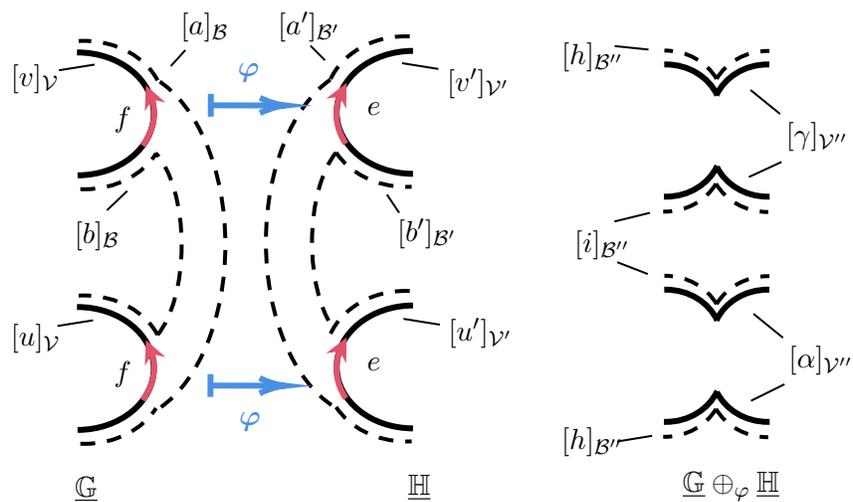}

\caption{Forming the 2-sum of packaged arrow presentations.}
\label{f.dog2}
\end{figure}

\begin{figure}
    \centering
    \begin{tabular}{cc}
      \input{Diagrams/ExPack2Sum1}  &  \input{Diagrams/ExPack2Sum2}
    \end{tabular}
    \caption{Two packaged arrow presentations, $\pack{G}$ and $\pack{H}$, with a coupling $\varphi$, and the 2-sum $\pack{G}\oplus_{\varphi} \pack{H}$.}
    \label{expack2sum}
\end{figure}

\begin{definition}\label{d.big2sum}
Let $\pack{G} = (\ar{G},\V,\B)$ and $\pack{H}=(\ar{H},\V',\B')$ be packaged arrow presentations,  $f$ be an edge of $\ar{G}$, $e$ be an edge of  $\ar{H}$ and $\varphi$ be a coupling of $f$ and $e$. 
Further suppose $\varphi$ sends the arrow $\overrightarrow{p_1p_2}$ to the arrow $\overrightarrow{q_1q_2}$, and sends the arrow $\overrightarrow{p_3p_4}$ to  the arrow $\overrightarrow{q_3q_4}$.
\begin{itemize}
\item In $\ar{G}$ let 
$u$ denote the vertex containing $\overrightarrow{p_1p_2}$; 
$v$ denote the vertex in containing $\overrightarrow{p_3p_4}$;  
$a$ denote the boundary component containing $p_1$ and $p_4$;
and 
$b$ denote the boundary component containing $p_2$ and $p_3$.

\item In $\ar{H}$ let 
$u'$ denote the vertex containing $\overrightarrow{q_1q_2}$; 
$v'$ denote the vertex in containing $\overrightarrow{q_3q_4}$;  
$a'$ denote the boundary component containing $q_1$ and $q_4$;
and 
$b'$ denote the boundary component containing $q_2$ and $q_3$.

\item In $\ar{G}\oplus_{\varphi} \ar{H}$ let 
$\alpha$ denote the vertex, and $h$ the boundary component, containing $p_1=q_1$; 
$\beta$ denote the vertex, and $i$ the boundary component, containing $p_2=q_2$; 
$\gamma$ denote the vertex, and $j$ the boundary component, containing $p_3=q_3$; 
and $\delta$ denote the vertex, and $k$ the boundary component, containing $p_4=q_4$. 
\end{itemize}

Then the \emph{2-sum} of $\pack{G}$ and $\pack{H}$ with respect to $\varphi$, denoted $\pack{G}\oplus_{\varphi} \pack{H}$ is the packaged arrow presentation $(\ar{G}\oplus_{\varphi} \ar{H}, \V'',\B'')$ where 
\begin{multline*}
\V'' = (\V \cup \V'-\{ [u]_{\V},[v]_{\V},[u']_{\V'} , [v']_{\V'}\}) 
\\
\cup\{  [u]_{\V} \cup [u']_{\V'} \cup \{\alpha,\beta\} -\{u,u'\}  \}
\\\cup\{  [v]_{\V} \cup [v']_{\V'} \cup \{\gamma,\delta\} -\{v,v'\}  \}
 ,\end{multline*}
and 
\begin{multline*}
\B'' = (\B \cup \B'-\{ [a]_{\B},[b]_{\B},[a']_{\B'} , [b']_{\B'}\}) 
\\
\cup\{  [a]_{\B} \cup [a']_{\B'} \cup \{h,k\} -\{a,a'\}  \}
\\\cup\{  [b]_{\B} \cup [b']_{\B'} \cup \{i,j\} -\{b,b'\}  \}
 .
 \end{multline*}
\end{definition}

The following two results extend Lemmas~\ref{lem:sumord} and~\ref{lem:sumop}. As  verifying the results is straightforward but  lengthy, we omit their proofs. 
\begin{lemma}
   Let $\pack{G} = (\ar{G},\V,\B)$, $\pack{H}=(\ar{H},\V',\B')$ and $\pack{K}=(\ar{K},\V'',\B'')$  be packaged arrow presentations, $f$ be an edge of $\ar{G}$, $e$ and $g$ be distinct edges of $\ar{H}$, and $h$ be an edge of $\ar{K}$.  
    In addition let $\varphi_{f,e}$ be a coupling of $f$ and $e$, and $\varphi_{g,h}$ be a coupling of $g$ and $h$. Then  \[
\pack{G} \oplus_{\varphi_{f,e}} \pack{H} = \pack{H} \oplus_{\varphi^{-1}_{f,e}} \pack{G}
\qquad \text{and} \qquad 
\pack{G} \oplus_{\varphi_{f,e}} (\pack{H} \oplus_{\varphi_{g,h}} \pack{K})  = (\pack{G} \oplus_{\varphi_{f,e}} \pack{H} )\oplus_{\varphi_{g,h}} \pack{K} .\] 
\end{lemma}

\begin{lemma}
 Let $\pack{G}$ and $\pack{H}$ be packaged arrow presentations, $f$ be an edge of $\pack{G}$, and $e$ and $g$ be distinct edges of $\pack{H}$. Additionally let $\varphi$ be a coupling of $f$ and $e$. Then the following hold. 
    \begin{enumerate}
        \item $(\pack{G}\oplus_{\varphi}\pack{H})\ba g = \pack{G}\oplus_{\varphi}(\pack{H}\ba g) $.
        
        \item $(\pack{G}\oplus_{\varphi}\pack{H})\con g = \pack{G}\oplus_{\varphi}(\pack{H}\con g) $.
        
        \item $(\pack{G}\oplus_{\varphi}\pack{H})\pcon g = \pack{G}\oplus_{\varphi}(\pack{H}\pcon g) $.
        
        \item $(\pack{G}\oplus_{\varphi}\pack{H})\mba g = \pack{G}\oplus_{\varphi}(\pack{H}\mba g) $.
        
        \item $(\pack{G}\oplus_{\varphi}\pack{H})\mcon g = \pack{G}\oplus_{\varphi}(\pack{H}\mcon g) $.
    \end{enumerate}
\end{lemma}

\begin{figure}
    \centering
    \input{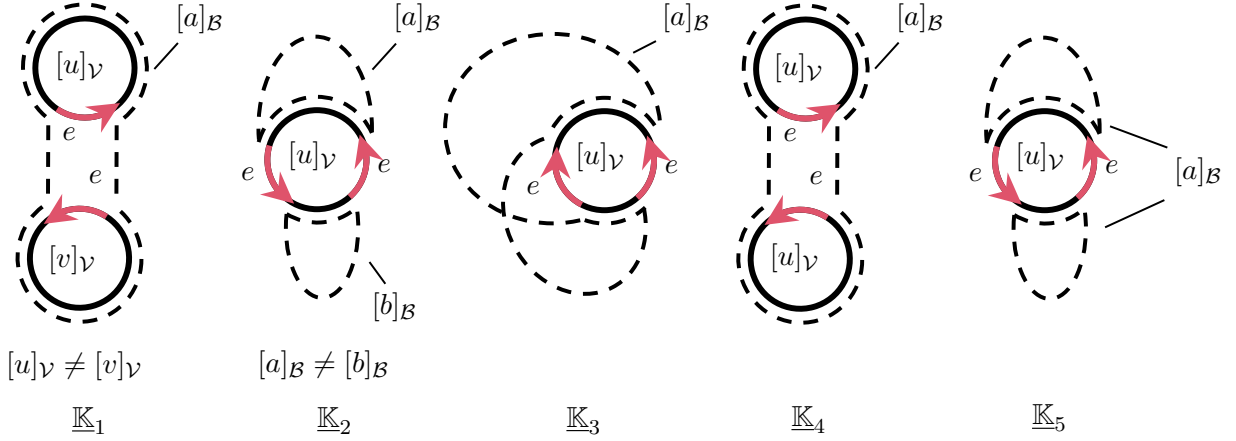}
    \caption{The five packaged arrow presentations with one edge and no isolated vertices. Here $[u]_{\V}\neq [v]_{\V}$ and $[a]_{\B}\neq [b]_{\B}$.}
    \label{fig:kgrph}
\end{figure}

Each of the five operations on edges in a packaged arrow presentation (deletion, contraction, Penrose-contraction, merge-deletion, and merge-contraction) can be realized by forming the 2-sum with one of the five distinct packaged arrow presentations on one edge.
\begin{lemma}\label{lem:2sumcon}
 Let $\pack{G}$ be a packaged arrow presentation with an edge $f$, and let $\pack{K}_1$--\;$\pack{K}_5$ be the packaged arrow presentations given in Figure~\ref{fig:kgrph}.    Then 
\begin{enumerate}
\item $\pack{G}\oplus_{\varphi} \pack{K}_1=\pack{G} \ba f$,
\item $\pack{G}\oplus_{\varphi}\pack{K}_2 =\pack{G} \con f$,
\item $\pack{G}\oplus_{\varphi}\pack{K}_3 =\pack{G} \pcon f$,
\item $\pack{G}\oplus_{\varphi}\pack{K}_4 =\pack{G} \mba f$,
\item\label{lem:2sumcon1} $\pack{G}\oplus_{\varphi}\pack{K}_5 =\pack{G} \mcon f$,
\end{enumerate}
here $\varphi$ is any coupling of $f$ and  $e$, where $e$ is the singleton edge in the appropriate  $\pack{K}_i$.  
\end{lemma}
\begin{proof}
The result can be verified by comparing the resulting packaged arrow presentations on the left and right hand sides of each equation. Figure~\ref{fig:k5} illustrates this for Item~\ref{lem:2sumcon1}.
The remaining cases are verified similarly and we omit the details.
\end{proof}

\begin{figure}[t!]
    \centering
    \begin{tabular}{ccc} \vspace{2.5mm}
     \input{Diagrams/G+K5.1} &~\qquad~& \tikzset{every picture/.style={line width=0.75pt}} %set default line width to 0.75pt        

\begin{tikzpicture}[x=0.5pt,y=0.5pt,yscale=-1,xscale=1]
%uncomment if require: \path (0,403); %set diagram left start at 0, and has height of 403

%Straight Lines [id:da8610902260849448] 
\draw    (231,89.6) -- (259,116.6) ;
%Straight Lines [id:da0941640787219189] 
\draw    (238,52.6) -- (268,24.6) ;
%Straight Lines [id:da8095715146570269] 
\draw    (285,208.6) -- (285,224.6) ;
%Shape: Arc [id:dp24783566529470313] 
\draw  [draw opacity=0][line width=2.25]  (278.42,32.07) .. controls (279.02,32.06) and (279.62,32.05) .. (280.21,32.05) .. controls (330.91,32.05) and (372,90.26) .. (372,162.06) .. controls (372,233.86) and (330.91,292.07) .. (280.21,292.07) .. controls (279.46,292.07) and (278.7,292.05) .. (277.95,292.03) -- (280.21,162.06) -- cycle ; \draw  [line width=2.25]  (278.42,32.07) .. controls (279.02,32.06) and (279.62,32.05) .. (280.21,32.05) .. controls (330.91,32.05) and (372,90.26) .. (372,162.06) .. controls (372,233.86) and (330.91,292.07) .. (280.21,292.07) .. controls (279.46,292.07) and (278.7,292.05) .. (277.95,292.03) ;  
%Shape: Arc [id:dp7873727151500824] 
\draw  [draw opacity=0][line width=2.25]  (268.73,113.55) .. controls (291.09,113.58) and (309.21,133.88) .. (309.21,158.91) .. controls (309.21,183.96) and (291.06,204.27) .. (268.68,204.27) .. controls (268.42,204.27) and (268.16,204.26) .. (267.89,204.26) -- (268.68,158.91) -- cycle ; \draw  [line width=2.25]  (268.73,113.55) .. controls (291.09,113.58) and (309.21,133.88) .. (309.21,158.91) .. controls (309.21,183.96) and (291.06,204.27) .. (268.68,204.27) .. controls (268.42,204.27) and (268.16,204.26) .. (267.89,204.26) ;  
%Shape: Arc [id:dp5559723695301908] 
\draw  [draw opacity=0][dash pattern={on 5.63pt off 4.5pt}][line width=1.5]  (268.98,196.21) .. controls (287.12,196.03) and (301.79,179.4) .. (301.79,158.91) .. controls (301.79,138.31) and (286.97,121.61) .. (268.68,121.61) .. controls (268.46,121.61) and (268.23,121.61) .. (268,121.61) -- (268.68,158.91) -- cycle ; \draw [color={rgb, 255:red, 0; green, 0; blue, 0 }  ,draw opacity=1 ][dash pattern={on 5.63pt off 4.5pt}][line width=1.5]  [dash pattern={on 5.63pt off 4.5pt}]  (268.98,196.21) .. controls (287.12,196.03) and (301.79,179.4) .. (301.79,158.91) .. controls (301.79,138.31) and (286.97,121.61) .. (268.68,121.61) .. controls (268.46,121.61) and (268.23,121.61) .. (268,121.61) ;  
%Shape: Arc [id:dp49534067169740936] 
\draw  [draw opacity=0][dash pattern={on 5.63pt off 4.5pt}][line width=1.5]  (275.84,301.36) .. controls (276.87,301.4) and (277.91,301.42) .. (278.95,301.42) .. controls (339.15,301.42) and (387.95,238.92) .. (387.95,161.82) .. controls (387.95,84.72) and (339.15,22.22) .. (278.95,22.22) .. controls (277.89,22.22) and (276.82,22.24) .. (275.76,22.28) -- (278.95,161.82) -- cycle ; \draw [color={rgb, 255:red, 0; green, 0; blue, 0 }  ,draw opacity=1 ][dash pattern={on 5.63pt off 4.5pt}][line width=1.5]  [dash pattern={on 5.63pt off 4.5pt}]  (275.84,301.36) .. controls (276.87,301.4) and (277.91,301.42) .. (278.95,301.42) .. controls (339.15,301.42) and (387.95,238.92) .. (387.95,161.82) .. controls (387.95,84.72) and (339.15,22.22) .. (278.95,22.22) .. controls (277.89,22.22) and (276.82,22.24) .. (275.76,22.28) ;  
%Straight Lines [id:da6603808572551911] 
\draw    (310.63,232.05) -- (353,219.6) ;

% Text Node
\draw (270.9,232.1) node [anchor=north west][inner sep=0.75pt]    {$[ x]_{\mathcal{V} ''}$};
% Text Node
\draw (212.65,59.4) node [anchor=north west][inner sep=0.75pt]    {$[ d]_{\mathcal{B} ''}$};

\end{tikzpicture} \\ 
        $\pack{G}$, $\pack{K}_5$ and $\varphi$ & &$\pack{G}\oplus_{\varphi}\pack{K}_5$ or $\pack{G} \mcon f$
    \end{tabular}
    \caption{Showing $\pack{G}\oplus_{\varphi}\pack{K}_5 =\pack{G} \mcon f$. }
    \label{fig:k5}
\end{figure}

\medskip

\begin{definition}\label{def:fulltens}
    Let $\pack{G}=(\ar{G},\V,\B)$ be a packaged arrow presentation and $\{\pack{H}^{(f)}\}_{f\in E(\pack{G})}$ be a family of packaged arrow presentations $\pack{H}^{(f)}=(\ar{H}^{(f)},\V^{(f)},\B^{(f)})$ indexed by the edges of $\ar{G}$. All the arrow presentations here are distinct. Further suppose that for each edge $f$ of $\ar{G}$ there is a coupling $\varphi_f$ of $f$ with an edge $e^{(f)}$ of $\ar{H}^{(f)}$. Let $\bphi = \{\varphi_f\}_{f\in E(\pack{G})}$. Then the \emph{tensor product} is the packaged arrow presentation

   \[\pack{G}\ot_{\bphi}\{\pack{H}^{(f)}\}_{f\in E(\pack{G})}
=\pack{G}\bigoplus_{\substack{\varphi_{f} \\f\in E(\pack{G})}}
\{\pack{H}^{(f)}\}_{f\in E(\pack{G})}\]
\end{definition}

\begin{definition}\label{def:littletens}
Let $\pack{G}$ be and $\pack{H}$ be packaged arrow presentations. Let $e$ be a fixed edge of $\pack{H}$ and for each edge $f$ of $\pack{G}$ let $\varphi_f$ be a coupling of $f$ and $e$, and $\bphi = \{\varphi_f\}_{f\in E(\pack{G})}$. Then 
\[  \pack{G}\ot_{\bphi} \pack{H} =  \ar{G} \ot_{\boldsymbol{\psi}} \{\pack{H}^{(f)}\}_{f\in E(\pack{G})},  \]
where each $\pack{H}^{(f)}$ is a packaged arrow presentation equivalent to $\pack{H}$ (all the copies are distinct); $e^{(f)}$ is the edge in $\pack{H}^{(f)}$ corresponding to $e$;  $\psi_f$ is the coupling of $f$ and $e^{(f)}$ induced from $\varphi_f$ under the equivalence; and  $\boldsymbol{\psi} =\{\psi_f\}_{f\in E(\pack{G})}$.
\end{definition}

\subsection{Vertex partitioned arrow presentations}

In Section~\ref{ss:brp} we consider the Bollob\'as--Riordan polynomial. For this we need to make use of vertex partitioned arrow presentations. A \emph{vertex partitioned arrow presentation}  $\vp{G}$ is a pair $(\ar{G}, \V)$ where $\ar{G}$ is an arrow presentation and $\V$ a partition of its vertex set. 
We note that as arrow presentations correspond to ribbon graphs, vertex partitioned arrow presentation correspond to Krajewski, Moffatt, and Tanasa's vertex partitioned ribbon graphs from~\cite[Section~4.7]{KMT}. 

A vertex partitioned arrow presentation $(\ar{G}, \V)$ can be regarded as a packaged arrow presentation $(\ar{G}, \V, \{B(G)\})$ in which all boundary components are in the same block and vice versa. With this we define versions of all the above packaged arrow presentation terms and operations (such as, deletion, contraction, 2-sums etc.) as those induced from packaged arrow presentations. Alternatively, just remove all mention of boundary partitions in the previous sections.

\section{Tensor product formulas for topological Tutte polynomials}\label{s.tpf}

\subsection{The Krushkal polynomial}\label{ss:kp}

\subsubsection{A review of the Krushkal polynomial}
 The {\em Krushkal polynomial} was introduced by Krushkal in \cite{krushkalpoly} for graphs in orientable surfaces,  and extended  to graphs in non-orientable surfaces by Butler in \cite{zbMATH06824436}. It is a  four-variable polynomial that we denote here by $K(G\subset \Sigma ; x,y,a,b )$ where  $G\subset \Sigma$ is a graph $G$ embedded in a closed surface $\Sigma$ (but not necessarily cellularly embedded).

Our aim is to extend Brylawski's tensor product formula  from the Tutte polynomial to the Krushkal polynomial. A difficulty with the Krushkal polynomial, as originally defined is that it has no known deletion-contraction relations that apply to an arbitrary edge. 
(A deletion-contraction relation here is a linear relation between the value of the polynomial $G\subset \Sigma$ and its two values after deleting and contraction one of its edges. It provides a recursive way to compute the polynomial.)

By extending the Krushkal polynomial to vertex partitioned graphs in surfaces, Krajewski, Moffatt and Tanasa, in~\cite{KMT}, found deletion-contraction relations that apply to any edge in the graph. This results a way to compute the Krushkal polynomial recursively with the base case consisting of its value on edgeless graphs. A more combinatorial approach to this work was taken by Huggett and Moffatt in~\cite{HM} who defined a polynomial $T_{ps}$ by framing it  in terms of both graphs embedded in pseudo-surfaces (i.e., surfaces with pinch-points), and in terms of ribbon graphs whose vertex set and set of boundary components are partitioned. 
Via Remark~\ref{r.papcrg}, in turn this can be expressed in terms of packaged arrow presentations resulting in a polynomial,  $T_{ps}(\pack{G};w,x,y,z)$. (As we do not need the details, we omit the definition of $T_{ps}$ which can be found in \cite[Definition~28]{HM}.)

\medskip 

We use an alternative form of $T_{ps}$ that has a simple deletion-contraction relation.
For a packaged arrow presentation $\pack{G}=(\ar{G},\V,\B)$  we let $Z(\pack{G};a,b,\al,\be,\ga)$ be the polynomial uniquely defined by
\begin{equation}\label{eq.arz}
Z(\pack{G};a,b,\al,\be,\ga)
=\begin{cases}
a\,Z(\pack{G}\ba e;a,b,\al,\be,\ga)+b\,Z(\pack{G}\con e;a,b,\al,\be,\ga)
\quad\text{for any edge $e$,}
\\ \al^{|V(\ar{G})|}\be^{|\V|}\ga^{|\B|} 
\quad \text{if $\ar{G}$ is edgeless. }
\end{cases}
\end{equation}
It follows from~\cite[Theorem~24]{HM} that $Z(\pack{G};a,b,\al,\be,\ga)$ is well-defined. (As taking $a_i=a$, $b_i=b$, $\al=\be$, $\be=\ga$, and $\ga = \al$ in~\cite[Theorem~24]{HM} gives~\eqref{eq.arz}.) The polynomial
$Z(\pack{G};a,b,\al,\be,\ga)$ can be written 
in terms of $T_{ps}$ by~\cite[Theorem~29]{HM}. 
An application of~\cite[Corollary~42]{HM} then shows that the Krushkal polynomial can be recovered from $Z(\pack{G};a,b,\al,\be,\ga)$.

\begin{remark}
For reference we include the specific relations between the above polynomials, although we do not use the details here.
There are various normalizations of the Krushkal polynomial in the literature. We use the following one.
 For an embedded graph $G\subset \Sigma$, 
\begin{equation}\label{d.kru}
 K(G\subset \Sigma ; x,y,a,b ) =  \sum_{A\subseteq E}  x^{r(G)-r(A)} y^{ \kappa(A)} a^{ \frac{1}{2}s(A)} b^{ \frac{1}{2} s^{\perp}(A)}.
   \end{equation}
 Using $N(X)$ to denote the surface with boundary defined by neighbourhood of the given subsets $X$ of $\Sigma$, and $\gamma( N(X))$ its Euler genus, here
  $s(A):= \gamma(N(V \cup A))$,  $s^{\perp}(A) := \gamma(\Sigma \backslash N(V \cup A))$, and 
 $\kappa(A)= \#\mathrm{components}(\Sigma  \backslash N(V \cup  A))  - \#\mathrm{components}(\Sigma ).$
Note that we use here the form of the exponent of $y$ from the proof of Lemma~4.1 of~\cite{zbMATH06127547}  rather than the homological definition given in \cite{krushkalpoly}. 

For reference, $T_{ps} \left(\pack{G} ; w,x,y,z\right)$ here is exactly that stated in~\cite[Definition 28]{HM} and we use the notation defined there but translated from ribbon graphs to arrow presentations.
By Theorem~\cite[Theorem~29]{HM}, 
\begin{multline}\label{eq.hdy}
Z(\pack{G}; a, b, \al,\be,\ga)
=
\be^{k(\ar{G}/\V)} \ga^{k(\ar{G}^*/\B)}\al^{v(\ar{G})-\rho(\ar{G})}
b^{\rho(\ar{G})} 
a^{|E|-\rho(\ar{G})}
 \\ 
 T_{ps} \left(\pack{G} ;  \frac{\al\be a}{b} , \frac{\al a}{b}, \frac{\al\ga b}{a},\frac{\al b}{a}\right).
\end{multline}
On the other hand $T_{ps}$ can be obtained from $Z$ as follows. 
\begin{multline}\label{eq.ztps2}
 T_{ps} \left(\pack{G} ; w,x,y,z\right)
=
\left(\frac{x}{w}\right)^{k(\ar{G}/\V)} \left(\frac{z}{y}\right)^{k(\ar{G}^*/\B)}\left(\frac{1}{\sqrt{xz}}\right)^{v(\ar{G})-\rho(\ar{G})}
(\sqrt{x})^{\rho(\ar{G})} 
(\sqrt{z})^{|E|-\rho(\ar{G})}\\
Z\left(\pack{G}; \frac{1}{\sqrt{z}}, \frac{1}{\sqrt{x}}, \sqrt{xz}, \frac{w}{x},\frac{y}{z}\right) .
\end{multline}
An application of~\cite[Corollary~42]{HM} then relates  $Z(\pack{G};a,b,\al,\be,\ga)$ to the Krushkal polynomial. If $\pack{G}$ is a packaged arrow presentation in which every vertex partition contains a unique element, and $G\subset \Sigma$ is its corresponding graph embedded in a surface constructed as in~\cite[Section~2.3]{HM} then
\begin{multline}\label{eq.zkr}
 K \left(G\subset \Sigma ; x,y,a,b\right) 
 =
 \left(ax\right)^{-k(\ar{G}/\V)} \left(by\right)^{-k(\ar{G}^*/\B)}\left(\sqrt{ab}\right)^{v(\ar{G})-\rho(\ar{G})}
\left(\sqrt{a}\right)^{-\rho(\ar{G})} \\
\left(\sqrt{b}\right)^{\rho(\ar{G})-|E|+\ga(\ar{G})}
a^{\rho(\ar{G})-r(\ar{G}/\V)}
 Z(\pack{G}; \sqrt{b}, \sqrt{a}, 1/\sqrt{ab}, ax,by).
\end{multline}
\end{remark}

\subsubsection{The tensor product formula}

Our aim is to extend Brylawski's tensor product formula, stated here in Equation~\eqref{btf}, to $Z(\pack{G})$. 
To do so we need to consider a generalisation $Q(\pack{G})$ of $Z(\pack{G})$. The reason for this is detailed in Remark~\ref{r.whyq}. The generalisation is defined via the following proposition.

\begin{proposition}\label{p.qdef}
Let $\pack{G}=(\ar{G},\V,\B)$ be a packaged arrow presentation. 
There is a unique map $Q(\pack{G})=Q(\pack{G}; a,b,c,x,y,\al,\be,\ga)$ from packaged arrow presentations to $\mathbb{Z}[a,b,c,x,y,\al,\be,\ga]$ such that
\begin{equation}
Q(\pack{G})
=\begin{cases}\label{eq:Q}
aQ(\pack{G}\ba e)+bQ(\pack{G}\con e)+cQ(\pack{G}\pcon e)+xQ(\pack{G}\mba e)+yQ(\pack{G}\mcon e), 
\\\hfill \qquad
\text{for any edge $e$;}
\\ \al^{|V(\ar{G})|}\be^{|\V|}\ga^{|\B|}, 
\qquad \text{if $\ar{G}$ is edgeless.} 
\end{cases}
\end{equation}
\end{proposition}
We delay the (routine) proof of this proposition and instead deduce it from a more general one below.

We shall next state our tensor product formulas for $Q(\pack{G}\otimes_{\bphi} \pack{H})$ and $Z(\pack{G}\otimes_{\bphi} \pack{H})$ before considering multivariate versions of them. These two formulas follow immediately from the multivariate extensions, but we state the simpler versions of the results first to help the reader digest the notation. 

\begin{theorem}\label{thm.main}
Let $\pack{G}$ be and $\pack{H}$ be packaged arrow presentations. Let $e$ be a fixed edge of $\pack{H}$ and for each edge $f$ of $\pack{G}$ let $\varphi_f$ be a coupling of $f$ and $e$, and $\bphi = \{\varphi_f\}_{f\in E(\pack{G})}$.  
Then
\begin{equation}\label{eq:main1}
Q(\pack{G}\otimes_{\bphi} \pack{H}; a,b,c,x,y, \al,\be,\ga) = Q(\pack{G}; \phi_{(\eq,2)},\phi_{(\pl,2)},\phi_{\x},\phi_{(\eq,1)},\phi_{(\pl,1)}, \al,\be,\ga)
\end{equation}
where the $\phi$'s are the unique solutions to the  systems of equations given as follows.
\begin{equation}\label{eq.mainmat1} \alpha\beta\gamma \left[ \begin{array}{ccccc}
   \alpha\beta & 1 & 1 & \alpha & 1 \\
    1 & \alpha\gamma & 1 & 1 & \alpha  \\
    1 & 1 & \alpha & 1 & 1 \\
    \alpha & 1 & 1 & \alpha & 1 \\
    1 & \alpha & 1 & 1 & \alpha 
\end{array}\right] 
\cdot \begin{bmatrix}
    \phi_{(\eq,2)} \\[2pt]
    \phi_{(\pl,2)} \\[2pt]
    \phi_{\x}\\[2pt]
    \phi_{(\eq,1)} \\[2pt]
    \phi_{(\pl,1)} 
\end{bmatrix}
=\begin{bmatrix}
    Q(\pack{H}\ba e; a,b,c,x,y ,\al,\be,\ga)\\[2pt]
    Q(\pack{H}\con e; a,b,c,x,y, \al,\be,\ga)\\[2pt]
    Q(\pack{H}\pcon e; a,b,c,x,y,  \al,\be,\ga)\\[2pt]
    Q(\pack{H}\mba e; a,b,c,x,y , \al,\be,\ga)\\[2pt]
    Q(\pack{H}\mcon e; a,b,c,x,y , \al,\be,\ga)
\end{bmatrix}.\end{equation}
\end{theorem}

This result will follow immediately form Theorem~\ref{thm.mainmv} below. 

Taking $c=x=y=0$ in Theorem~\ref{thm.main} gives our tensor product formula for $Z(\pack{G})$, and hence for $T_{ps}(\pack{G})$ and  $K(G\subset \Sigma)$ via Equations~\eqref{eq.ztps2} and~\eqref{eq.zkr}, respectively. Remark~\ref{r.whyq} details why in general $Q$ cannot be replaced by $Z$ in this expression.

\begin{corollary}\label{cor.main}
Let $\pack{G}$ and $\pack{H}$ be packaged arrow presentations and $e$ be a fixed edge of $\pack{H}$. For each edge $f$ of $\pack{G}$, let $\varphi_f$ be a coupling of $f$ and $e$, and  $\bphi = \{\varphi_f\}_{f\in E(\pack{G})}$.  
Then
\begin{equation}\label{eq:main2}
Z(\pack{G}\otimes_{\bphi} \pack{H}; a,b, \al,\be,\ga) = Q(\pack{G}; \phi_{(\eq,2)},\phi_{(\pl,2)},\phi_{\x},\phi_{(\eq,1)},\phi_{(\pl,1)}, \al,\be,\ga)
\end{equation}
where the $\phi$'s are the unique solutions to the  systems of equations given as follows.
\begin{equation}\label{eq.mainmat2} \alpha\beta\gamma \left[ \begin{array}{ccccc}
   \alpha\beta & 1 & 1 & \alpha & 1 \\
    1 & \alpha\gamma & 1 & 1 & \alpha  \\
    1 & 1 & \alpha & 1 & 1 \\
    \alpha & 1 & 1 & \alpha & 1 \\
    1 & \alpha & 1 & 1 & \alpha 
\end{array}\right] 
\cdot \begin{bmatrix}
    \phi_{(\eq,2)} \\[2pt]
    \phi_{(\pl,2)} \\[2pt]
    \phi_{\x}\\[2pt]
    \phi_{(\eq,1)} \\[2pt]
    \phi_{(\pl,1)} 
\end{bmatrix}
=\begin{bmatrix}
    Z(\pack{H}\ba e; a,b ,\al,\be,\ga)\\[2pt]
    Z(\pack{H}\con e; a,b, \al,\be,\ga)\\[2pt]
    Z(\pack{H}\pcon e; a,b, \al,\be,\ga)\\[2pt]
    Z(\pack{H}\mba e; a,b, \al,\be,\ga)\\[2pt]
    Z(\pack{H}\mcon e; a,b, \al,\be,\ga)
\end{bmatrix}.\end{equation}
\end{corollary}

We turn our attention to a multivariate version of Theorem~\ref{thm.main}.

\begin{proposition}\label{p.qmvdef}
Let $\pack{G}=(\ar{G},\V,\B)$ be a packaged arrow presentation.
For each edge $e$ in $E(\ar{G})$ let $(a_e,b_e,c_e,x_e,y_e)$ be an ordered tuple of variables.
Then there is a unique map $\mvQ(\pack{G})$ from packaged arrow presentations to $\mathbb{Z}[a_e,b_e,c_e,x_e,y_e,\al,\be,\ga : e\in E(\ar{G})]$ such that
\begin{equation}\label{eq:Qmv}
\mvQ(\pack{G})
=\begin{cases}
a_e\mvQ(\pack{G}\ba e)+b_e\mvQ(\pack{G}\con e)+c_e\mvQ(\pack{G}\pcon e)+x_e\mvQ(\pack{G}\mba e)+y_e\mvQ(\pack{G}\mcon e), \\\hfill \qquad
\text{for any edge $e$;}
\\ \al^{|V(\ar{G})|}\be^{|\V|}\ga^{|\B|}, 
\qquad  \text{if $\ar{G}$ is edgeless.} 
\end{cases}
\end{equation}
\end{proposition}
\begin{proof}
We use a standard argument. 
Assign an arbitrary linear order to the edges of $\pack{G}$ and compute $\mvQ(\pack{G})$ by applying Equation~\eqref{eq:Qmv} to the edges with respect to this order. The resulting summands  of $\mvQ(\pack{G})$ (before collecting terms) are in bijection with  ordered tuples $(A_1,A_2,A_3,A_4,A_5)$ where each $A_i \subseteq E(\ar{G})$, $\bigcup_{i=1}^5 A_i=E(\ar{G})$ and the  $A_i$'s are pairwise disjoint. Here $A_1$ is the set of edges that are deleted in forming the summand,  $A_2$ those that are contracted, $A_3$ those that are Penrose-contracted, $A_4$ those that are merge-deleted, and $A_5$ those that are merge-contracted. By Lemma~\ref{l.cmmt} we may carry these operations out in any order without changing the end result.  Thus the summand corresponding to $(A_1,A_2,A_3,A_4,A_5)$ contributes 
\begin{equation}\label{eq.mvqstate}
\big(\prod_{e\in A_1} a_e\big) \big(\prod_{e\in A_2} b_e\big)\big(\prod_{e\in A_3} c_e\big)\big(\prod_{e\in A_4} x_e\big) \big(\prod_{e\in A_5} y_e\big)
\alpha^{|V( \ar{G}')| } \beta^{|\V'|}\gamma^{|\B'|},
\end{equation}
where
$ (\ar{G}',\V',\B') =  \pack{G}\ba A_1\con A_2 \pcon A_3 \mba A_4 \mcon A_5 $.
The polynomial $Q(\pack{G})$ can then be obtained by summing these contributions over all possible tuples $(A_1,A_2,A_3,A_4,A_5)$. As this sum is clearly independent of the choice of linear order of the edges, the result follows.
\end{proof}

Observe that Proposition~\ref{p.qdef} follows immediately from Proposition~\ref{p.qmvdef}.

We use
$\mvQ(\pack{G}; \bw, \al,\be,\ga)$, 
where $\bw=\{(a_e,b_e,c_e,x_e,y_e) :e\in E (\ar{G})\}$
to denote the polynomial defined by~\eqref{eq:Qmv}.

The following result is immediate.
\begin{proposition}\label{p.dunion}
Let $\pack{G}=(\ar{G},\V,\B)$ be a packaged arrow presentation and $\pack{H} = (\ar{H},\{\{v\}\},\{\{b\}\})$ be a packaged arrow presentation consisting of a single vertex $v$, a single boundary component $b$ and no edges (i.e., one circle with no arrows). Denote the disjoint union of their arrow presentations by $\ar{G}\sqcup \ar{H}$. Let $\pack{G} \sqcup \pack{H}$ be a packaged arrow presentation of the form $(\ar{G}\sqcup \ar{H}, \V', \B')$ where $\V'$ and $\B'$ restrict to the partitions $\V$ and $\B$, respectively.
Then 
\[\mvQ(\pack{G} \sqcup \pack{H}; \bw, \al,\be,\ga) = \al \be^p \ga^q \,
\mvQ(\pack{G}; \bw, \al,\be,\ga), \] 
where $p=1$ if $v$ is in its own block of the partition $\V'$, and $p=0$ otherwise; and 
$q=1$ if $b$ is in its own block of the partition $\B'$, and $q=0$ otherwise.
\end{proposition}

The following theorem is the main result of this paper. It provides the tensor product formula from which formulas for the Krushkal, Bollob\'as--Riordan, ribbon graph, topological transition and Tutte polynomials will follow. In particular, it satisfactorially completes the work initiated by Huggett and Moffatt in~\cite{MR2854787}. We note that Corollary~\ref{cor.fulltens} offers a cleaner restatement of the result.

\begin{theorem}\label{thm.mainmv}
Let $\pack{G}$ and $\pack{H}^{(f_i)}$, for $i=1,\ldots,k$, be packaged arrow presentations. Suppose that each $\pack{H}^{(f_i)}$ has an edge  $e^{(i)}$, and that $f_1, \ldots, f_k$ are distinct  edges of $\pack{G}$.
Suppose also that for each $i$ there is coupling $\varphi_i$ of $f_i$ and $e^{(i)}$.
Let 
\[ \pack{G}[k]   = \pack{G}\oplus_{\varphi_1} \pack{H}^{(f_1)}\oplus_{\varphi_2} \cdots \oplus_{\varphi_k} \pack{H}^{(f_k)}.\]
Then
\begin{equation}\label{eq:mainmv}
\mvQ(\pack{G}[k]; \bw, \al,\be,\ga) = \mvQ(\pack{G}; \bw', \al,\be,\ga),
\end{equation}
where 
$\bw=\{(a_f,b_f,c_f,x_f,y_f) :f\in E(\pack{G}[k]) \}$
and 
\begin{multline*}
   \bw'=
\big\{\big(\phi^{(f_i)}_{(\eq,2)},\phi^{(f_i)}_{(\pl,2)},\phi^{(f_i)}_{\x},\phi^{(f_i)}_{(\eq,1)},\phi^{(f_i)}_{(\pl,1)}\big) : i=1,\ldots, k  \big\} \\ \cup
\big\{(a_f,b_f,c_f,x_f,y_f) :f\in E(\pack{G})-\{f_1,\ldots ,f_k\}\big\}. 
\end{multline*}
and where, for each $f_i$,  
$\phi^{(f_i)}_{(\eq,2)}$, $\phi^{(f_i)}_{(\pl,2)},  \phi^{(f_i)}_{\x}$, $\phi^{(f_i)}_{(\eq,1)}$,and $\phi^{(f_i)}_{(\pl,1)} $ arises from the unique solution to
\begin{equation}\label{eq:mainmtx}\alpha\beta\gamma \left[ \begin{array}{ccccc}
    \alpha\beta & 1 & 1 & \alpha & 1 \\
    1 & \alpha\gamma & 1 & 1 & \alpha  \\
    1 & 1 & \alpha & 1 & 1 \\
    \alpha & 1 & 1 & \alpha & 1 \\
    1 & \alpha & 1 & 1 & \alpha 
\end{array}\right] 
\cdot \begin{bmatrix}
    \phi^{(f_i)}_{(\eq,2)} \\[8pt] 
    \phi^{(f_i)}_{(\pl,2)} \\[8pt] 
    \phi^{(f_i)}_{\x}\\[8pt] 
    \phi^{(f_i)}_{(\eq,1)} \\[8pt] 
    \phi^{(f_i)}_{(\pl,1)} 
\end{bmatrix}
=\begin{bmatrix}
    \mvQ(\pack{H}^{(f_i)}\ba e^{(i)}; \bw ,\al,\be,\ga)\\[5pt] 
    \mvQ(\pack{H}^{(f_i)}\con e^{(i)}; \bw, \al,\be,\ga)\\[5pt] 
    \mvQ(\pack{H}^{(f_i)}\pcon e^{(i)}; \bw,  \al,\be,\ga)\\[5pt] 
    \mvQ(\pack{H}^{(f_i)}\mba e^{(i)}; \bw , \al,\be,\ga)\\[5pt] 
    \mvQ(\pack{H}^{(f_i)}\mcon e^{(i)}; \bw , \al,\be,\ga)
\end{bmatrix}.\end{equation}
\end{theorem}
\begin{proof}
We prove the result by induction on $k$. If $k=0$ there is nothing to prove. 
Next suppose that Equation~\eqref{eq:mainmv} is true in the  $\pack{G}[k-1]$ case.

 Apply Equation~\eqref{eq:Qmv} to each of the edges in  $\pack{H}^{(f_k)}$ other than $e^{(k)}$. Then use Proposition~\ref{p.dunion} to remove any vertices that were created in this application of Equation~\eqref{eq:Qmv} as well as any isolated vertices that were part of $\pack{H}^{(f_k)}$.
 Thus we may write 
\begin{align*}
\mvQ(\pack{H}^{(f_k)}; \bw, \al,\be,\ga)
&= 
\phi^{(f_k)}_{(\eq,2)} \mvQ\big(\pack{K}_1 ; \bw, \al,\be,\ga))
+\phi^{(f_k)}_{(\pl,2)} \mvQ\big(\pack{K}_2 ; \bw, \al,\be,\ga))
\\
&\qquad +\phi^{(f_k)}_{(\x)} \mvQ\big(\pack{K}_3 ; \bw, \al,\be,\ga))
+\phi^{(f_k)}_{(\eq,1)} \mvQ\big(\pack{K}_4 ; \bw, \al,\be,\ga))
\\
&\qquad \qquad +\phi^{(f_k)}_{(\pl,1)} \mvQ\big(\pack{K}_5 ; \bw, \al,\be,\ga)),
\end{align*}
where  $\pack{K}_1, \ldots, \pack{K}_5$ are as in Figure~\ref{fig:kgrph}.

Next, applying Equation~\eqref{eq:Qmv} and Proposition~\ref{p.dunion} to the copy of $\pack{H}^{(f_k)}$ in $\pack{G}[k]$ and making use of Lemma~\ref{l.cmmt}
we can write
\begin{align}
\mvQ(\pack{G}[k]; \bw, \al,\be,\ga)
&= 
\phi^{(f_k)}_{(\eq,2)} \mvQ\big(\pack{G}[k-1]\oplus_{\varphi_k} \pack{K}_1 ; \bw, \al,\be,\ga))
\nonumber
\\&\qquad+\phi^{(f_k)}_{(\pl,2)} \mvQ\big(\pack{G}[k-1]\oplus_{\varphi_k} \pack{K}_2 ; \bw, \al,\be,\ga))
\nonumber
\\& \qquad\qquad
+\phi^{(f_k)}_{(\x)} \mvQ\big(\pack{G}[k-1]\oplus_{\varphi_k} \pack{K}_3 ; \bw, \al,\be,\ga))
\label{eq.inproof1}
\\&\qquad\qquad\qquad
+\phi^{(f_k)}_{(\eq,1)} \mvQ\big(\pack{G}[k-1]\oplus_{\varphi_k} \pack{K}_4 ; \bw, \al,\be,\ga))
\nonumber
\\& \qquad\qquad\qquad\qquad
+\phi^{(f_k)}_{(\pl,1)} \mvQ\big(\pack{G}[k-1]\oplus_{\varphi_k} \pack{K}_5 ; \bw, \al,\be,\ga)),
\nonumber
\end{align}
By Lemma~\ref{lem:2sumcon}, we may rewrite this as
\begin{align}
\mvQ(\pack{G}[k]; \bw, \al,\be,\ga)
&= 
\phi^{(f_k)}_{(\eq,2)} \mvQ\big(\pack{G}[k-1] \ba f_k ; \bw, \al,\be,\ga)
\nonumber
\\& 
\qquad+\phi^{(f_k)}_{(\pl,2)} \mvQ\big(\pack{G}[k-1] \con f_k ; \bw, \al,\be,\ga)
\nonumber
\\&
\label{eq.inproof2}
\qquad\qquad
+\phi^{(f_k)}_{(\x)} \mvQ\big(\pack{G}[k-1] \pcon f_k ; \bw, \al,\be,\ga)
\\& \qquad\qquad\qquad
+\phi^{(f_k)}_{(\eq,1)} \mvQ\big(\pack{G}[k-1] \mba f_k ; \bw, \al,\be,\ga)
\nonumber
\\& \qquad\qquad\qquad\qquad
+\phi^{(f_k)}_{(\pl,1)} \mvQ\big(\pack{G}[k-1] \mcon f_k ; \bw, \al,\be,\ga)
\nonumber
\\&=
\mvQ(\pack{G}[k-1]; \bw'', \al,\be,\ga),
\nonumber
\end{align}
where 
\[\bw''=
\{(\phi^{(f_k)}_{(\eq,2)},\phi^{(f_k)}_{(\pl,2)},\phi^{(f_k)}_{\x},\phi^{(f_k)}_{(\eq,1)},\phi^{(f_k)}_{(\pl,1)})\} \cup
\{(a_f,b_f,c_f,x_f,y_f) :f\in E(\pack{G}[k-1])-\{f_k\}\}.
\]
The last equality follows from~\eqref{eq:Qmv} (treating the $\phi$'s as variables).
By the inductive hypothesis it follows that  
$\mvQ(\pack{G}[k]; \bw, \al,\be,\ga) = \mvQ(\pack{G}; \bw'; \al,\be,\ga)$.

It remains to show that the $\phi$'s are the solutions to the given systems of equations. We can write
\begin{align*}
    \mvQ(\pack{H}^{(f_k)}\ba e; \bw, \al,\be,\ga)
&= 
\phi^{(f_k)}_{(\eq,2)} \mvQ\big(\pack{K}_1 \ba e; \bw, \al,\be,\ga)
+\phi^{(f_k)}_{(\pl,2)} \mvQ\big(\pack{K}_2 \ba e; \bw, \al,\be,\ga)
\\
&\qquad +\phi^{(f_k)}_{(\x)} \mvQ\big(\pack{K}_3 \ba e; \bw, \al,\be,\ga)
+\phi^{(f_k)}_{(\eq,1)} \mvQ\big(\pack{K}_4 \ba e; \bw, \al,\be,\ga)
\\
&\qquad \qquad +\phi^{(f_k)}_{(\pl,1)} \mvQ\big(\pack{K}_5 \ba e; \bw, \al,\be,\ga),\\
&= \phi^{(f_k)}_{(\eq,2)} \alpha^2\beta^2\gamma
+ \phi^{(f_k)}_{(\pl,2)} \alpha\beta\gamma
+ \phi^{(f_k)}_{(\x)} \alpha\beta\gamma
+ \phi^{(f_k)}_{(\eq,1)} \alpha^2\beta\gamma
+ \phi^{(f_k)}_{(\pl,1)} \alpha\beta\gamma,
\end{align*}
and similarly for $ \mvQ(\pack{H}^{(f_k)}\con e; \bw, \al,\be,\ga)$, $ \mvQ(\pack{H}^{(f_k)}\pcon e; \bw, \al,\be,\ga)$, $ \mvQ(\pack{H}^{(f_k)}\mba e; \bw, \al,\be,\ga)$, and \linebreak $ \mvQ(\pack{H}^{(f_k)}\mcon e; \bw, \al,\be,\ga)$. The five resulting equations can be rewritten as the matrix equation~\eqref{eq:mainmtx}. As the matrix is non-singular, it uniquely determines the $\phi$'s.
\end{proof}

The following corollary is immediate from the proof of Theorem~\ref{thm.mainmv} and offers a cleaner statement of it  and a form similar to Brylawski's original tensor product formula for the Tutte polynomial.

\begin{corollary}\label{cor.fulltens}
Let $\pack{G}$, $\{\pack{H}^{(f)}\}_{f\in E(\pack{G})}$, $e^{(f)}$ and $\bphi$ be as in Definition~\ref{def:fulltens}.
   Then 
    \[\mvQ(\pack{G}\ot_{\bphi} \{\pack{H}^{(f)}\}_{f\in E(\pack{G})}; \bw, \al,\be,\ga) = \mvQ(\pack{G}; \bw', \al,\be,\ga),\]
    where 
    $\bw=\{(a_f,b_f,c_f,x_f,y_f) :f\in E(\pack{G}[k]) \}$
    and 
    \[\bw'=
    \big\{\big(\phi^{(f)}_{(\eq,2)},\phi^{(f)}_{(\pl,2)},\phi^{(f)}_{\x},\phi^{(f)}_{(\eq,1)},\phi^{(f)}_{(\pl,1)}\big) : f \in E(\ar{G}) \big\},
    \]
    and where, for each $f$, $\phi^{(f)}_{(\eq,2)}$, $ \phi^{(f)}_{(\pl,2)}$, $\phi^{(f)}_{\x} $, $\phi^{(f)}_{(\eq,1)} $ and $\phi^{(f)}_{(\pl,1)} $ arises from the unique solution to the system of equations of the form~\eqref{eq:mainmtx}, replacing the $f_i$ with $f$ and $e^{(i)}$ with $e^{(f)}$ in the displayed matrices. 
\end{corollary}

\begin{proof}[Proof of Theorem~\ref{thm.main}]
The result follows from Corollary~\ref{cor.fulltens}. 
\end{proof}

We highlight the following 2-sum formula as a special case.
\begin{corollary}
Let $\pack{G}$, $\pack{H}$, $e$ and $\varphi$ be as in Definition~\ref{d.big2sum}.
Then
\[Q(\pack{G}\oplus_{\varphi}\pack{H};a,b,c,x,y,\al,\be,\ga)=
\begin{bmatrix}
    Q(\pack{G}\ba e; a,b,c,x,y ,\al,\be,\ga)\\[2pt]
    Q(\pack{G}\con e; a,b,c,x,y, \al,\be,\ga)\\[2pt]
    Q(\pack{G}\pcon e; a,b,c,x,y,  \al,\be,\ga)\\[2pt]
    Q(\pack{G}\mba e; a,b,c,x,y , \al,\be,\ga)\\[2pt]
    Q(\pack{G}\mcon e; a,b,c,x,y , \al,\be,\ga)
\end{bmatrix}
\cdot \begin{bmatrix}
    \phi_{(\eq,2)} \\[2pt]
    \phi_{(\pl,2)} \\[2pt]
    \phi_{\x}\\[2pt]
    \phi_{(\eq,1)} \\[2pt]
    \phi_{(\pl,1)} 
\end{bmatrix}
,\]
where the $\phi$'s arise from the unique solution to
\begin{equation*}
   \alpha\beta\gamma \left[ \begin{array}{ccccc}
   \alpha\beta & 1 & 1 & \alpha & 1 \\
    1 & \alpha\gamma & 1 & 1 & \alpha  \\
    1 & 1 & \alpha & 1 & 1 \\
    \alpha & 1 & 1 & \alpha & 1 \\
    1 & \alpha & 1 & 1 & \alpha 
\end{array}\right] 
\cdot \begin{bmatrix}
    \phi_{(\eq,2)} \\[2pt]
    \phi_{(\pl,2)} \\[2pt]
    \phi_{\x}\\[2pt]
    \phi_{(\eq,1)} \\[2pt]
    \phi_{(\pl,1)} 
\end{bmatrix}
=\begin{bmatrix}
    Q(\pack{H}\ba e; a,b,c,x,y,\al,\be,\ga)\\[2pt]
    Q(\pack{H}\con e; a,b,c,x,y, \al,\be,\ga)\\[2pt]
    Q(\pack{H}\pcon e; a,b,c,x,y,  \al,\be,\ga)\\[2pt]
    Q(\pack{H}\mba e; a,b,c,x,y , \al,\be,\ga)\\[2pt]
    Q(\pack{H}\mcon e; a,b,c,x,y , \al,\be,\ga)
\end{bmatrix}.
\end{equation*}
\end{corollary}

\begin{remark}\label{r.whyq}
Our aim was to extend Brylawski's tensor product formula to $Z(\pack{G}\ot_{\bphi} \pack{H})$. However, we needed to consider the more general polynomial $Q(\pack{G}\ot_{\bphi} \pack{H})$. The reason for this can be seen in the rewriting of Equation~\eqref{eq.inproof1} as Equation~\eqref{eq.inproof2} where it is necessary to consider the operations $\pcon$, $\mba$, and $\mcon$ in addition to $\ba$ and $\con$.
 Such a need can also been seen in~\cite[Section~6]{EllisMoff} where a tensor product formula is given for the transition polynomial rather than the ribbon graph polynomial.
Avoiding the use of $Q$ altogether, and having a tensor product formula completely in terms of $Z$ can be done via Technique~\ref{r.phi0} below if we insist 
all vertex and boundary partitions have only one block and all the arrow presentations represent  orientable ribbon graphs (which ensures the $\phi^{(f_i)}_{\x}$ terms are zero, as in the proof of \cite[Corollary 6.3]{EllisMoff}). As we note in Remark~\ref{rem.hugg},  doing so results in~\cite[Theorem~2]{MR2854787}.
\end{remark}

The following observation, which we present as a technique, allows us to recover known results from Theorem~\ref{thm.mainmv}.

\begin{technique}\label{r.phi0}
If we can deduce from the properties of $\pack{H}^{(f_i)}$ that specific $\pack{K}_i$ cannot arise after applying  Equation~\eqref{eq:Qmv} to each of its edges except $e^{(i)}$, then we can simplify the matrix Equation~\eqref{eq:mainmtx}.

For example, if $e^{(i)}$ is incident to two vertices that are in the same block of the vertex partition (or it is only incident to one vertex), then $\pack{K}_1$ cannot arise and $\phi^{(f_i)}_{(\eq,2)}=0$.  Similarly, we can deduce that if the boundary components adjacent to $e^{(i)}$ are in the same boundary partition (or if  $e^{(i)}$ is only adjacent to one boundary component), then $\pack{K}_2$ cannot arise and $\phi^{(f_i)}_{(\pl,2)}=0$.
If each $\pack{H}^{(f_i)}$ is orientable (i.e., corresponds to an orientable ribbon graph) and we set $c_{f} =0$ for every $f\in E(\pack{G})$ then, since $\pack{K}_5$ cannot arise, we have $\phi^{(f_i)}_{\x}=0$. 
In all of the above cases, we can simplify the matrix Equation~\eqref{eq:mainmv} by substituting in $\phi^{(f_i)}_{(\eq,2)}=0$, $\phi^{(f_i)}_{(\pl,2)}=0$ or $\phi^{(f_i)}_{\x}=0$ accordingly and removing the redundant rows.

A similar argument can be applied to Theorem~\ref{thm.main} to simplify the matrix equation~\eqref{eq.mainmat1} under the same conditions. In general, this technique can be used whenever you can deduce that a specific $\pack{K}_i$ cannot arise after resolving the other edges in $\pack{H}^{(f_i)}$.
\end{technique}

\section{Some special cases}\label{s:cases}

\subsection{The Bollob\'as--Riordan polynomial}\label{ss:brp}

The Bolloba\'s--Riordan polynomial~\cite{bollobasriordanpoly,zbMATH01801590} is probably the best-known embedded graph polynomial. For a ribbon graph $\rib{G}$ , the \emph{Bollob\'as--Riordan polynomial}, $R(\rib{G};x,y,z) \in \mathbb{Z}[x,y,z^{1/2}]$, is 
\[ R(\rib{G};x,y,z)= \sum_{A \subseteq E(\rib{G})}   (x-1)^{r( E ) - r( A )}   y^{|A|-r(A)} z^{\gamma(A)},\] 
where $r( A )$ is the rank of the ribbon graph $\rib{G}\ba(E-A)$ and $\gamma( A )$ its Euler genus. (Again we do not need the specifics of this definition, but include it for reference.)
As shown in~\cite[Lemma~4.1]{krushkalpoly} and~\cite[Theorem~5.1]{zbMATH06824436}, the Bollob\'as--Riordan polynomial can be recovered from the Krushkal polynomial 
\begin{equation}\label{eq.kbr}
R(\rib{G}; x+1,y,z)=y^{\frac{1}{2}\gamma(\rib{G})} K(G\subset\Sigma;x,y,yz^2,1/y),
\end{equation}
where $\rib{G}$ is a ribbon graph and $G\subset\Sigma$ its corresponding cellularly embeddded graph.

As with the Krushkal polynomial, there are no known deletion-contraction relations for the Bollob\'as--Riordan polynomial that apply to an arbitrary edge. It was shown in~\cite{KMT} that extending the Bollob\'as--Riordan polynomial to vertex partitioned arrow presentations results in a polynomial that does have deletion-contraction relations that apply to any edge. That polynomial was expressed in terms of a three-variable polynomial $T_{cps}$ of graphs embedded in pseudo-surfaces, or equivalently vertex partitioned arrow presentations, in~\cite[Definition~38]{HM}. 
We shall consider a polynomial equivalent to $T_{cps}$.

Let $\vp{G}$ be a vertex partitioned arrow presentation
 and $\pack{G}$ be any packaged arrow presentation obtained by partitioning the boundary components of $\vp{G}$.
 Set 
 \[ \Z(\vp{G}; a,b, \al,\be) = Z(\pack{G} ;a,b, \al,\be,1) . \]
 Note that the value of $\Z$ is independent of the choice of boundary partition in forming $\pack{G}$.
 
It follows from Equation~\eqref{eq.arz} that
\begin{equation}\label{eq.dgtshg}
\Z(\vp{G};a,b,\al,\be)
=\begin{cases}
a\,\Z(\vp{G}\ba e;a,b,\al,\be)+b\,\Z(\vp{G}\con e;a,b,\al,\be),
\quad\text{for any edge $e$;}
\\ \al^{|V(\ar{G})|}\be^{|\V|},
\quad \text{if $\vp{G}=(\ar{G},\V)$ is edgeless.}
\end{cases}
\end{equation}
By~\cite[Theorem~39]{HM}, it  follows that $T_{cps}$ is equivalent to $\Z$. (Alternatively use the fact that $T_{cps}(\pack{G},w,x,y)= T_{ps}(\vp{G},w,x,y,y)$ together with Equations~\eqref{eq.hdy} and~\eqref{eq.ztps2}.) Thus by giving a tensor product formula for $\Z$ we give one for $R(\rib{G};x+1,y,z)$.

\begin{remark}
Although we do not use it here, for reference, and using the notation in~\cite{HM}, we will specify how $\Z$ and  $R$ are related. 
If $\rib{G}$ is a ribbon graph with corresponding arrow presentation $\ar{G}$, $\vp{G}=(\ar{G},\V)$ where each vertex is in its own block in $\V$, and $\pack{G}=(\ar{G},\V,\B)$ where $\B$ is any boundary partition, 
 Equations~\eqref{eq.zkr} and~\eqref{eq.kbr} give that 
\begin{align*}
    R(\rib{G};x+1,y,z) &= y^{\frac{1}{2}\gamma(\vp{G})} K(\pack{G};x,y,yz^2,1/y)\\
    &=(1/xyz^2)^{k(\pack{G}/\V)}\sqrt{y}^{|E|-2r(\pack{G}/\V)}z^{v(\pack{G})-2r(\pack{G}/\V)}\\
    &\qquad Z(\vp{G}; 1/\sqrt{y}, \sqrt{yz^2}, 1/z, xyz^2,1).
\end{align*}
\end{remark}

 \medskip

As it is a specialisation of the Krushkal polynomial we can recover a tensor product formula for the extension $\Z$ of the Bollob\'as--Riordan polynomial from Theorem~\ref{thm.main}. To do so, for $\vp{G}$ a vertex partitioned arrow presentation, let $\pack{G}$ be any packaged arrow presentation obtained by  partitioning the boundary components of $\vp{G}$ and set
\begin{equation}\label{eq:brqpol}
    \Q(\vp{G}; a,b,c,x,\al,\be) = Q(\pack{G}; a,b,c,x,0,\al,\be,1).
\end{equation}
Note that the value of $\Q$ is independent of the choice in partitioning the boundary components.

\begin{theorem}\label{thm.mainBR}
Let $\vp{G}$ and $\vp{H}$ be vertex partitioned arrow presentations. Let $e$ be a fixed edge of $\vp{H}$ and for each edge $f$ in $\vp{G}$ let $\varphi_f$ be a coupling of $f$ and $e$, and $\bphi = \{\varphi_f\}_{f\in E(\vp{G})}$.
Then
\begin{equation}\label{eq:main3}
\Q(\vp{G}\ot_{\bphi} \vp{H}; a,b,c,x, \al,\be) = \Q(\vp{G}; \phi_{(\eq,2)},\phi_{\pl},\phi_{\x},\phi_{(\eq,1)}, \al,\be),
\end{equation}
where the $\phi$'s arise from the unique solution to
\begin{equation}\label{eq.mainmatBR1} \alpha\beta \left[ \begin{array}{cccc}
   \alpha\beta & 1 & 1 & \alpha  \\
    1 & \alpha & 1 & 1  \\
    1 & 1 & \alpha & 1  \\
    \alpha & 1 & 1 & \alpha  
\end{array}\right] 
\cdot \begin{bmatrix}
    \phi_{(\eq,2)} \\[2pt]
    \phi_{\pl} \\[2pt] 
    \phi_{\x}\\[2pt]
    \phi_{(\eq,1)} 
\end{bmatrix}
=\begin{bmatrix}
    \Q(\vp{H}\ba e; a,b,c,x,\al,\be)\\[2pt]
    \Q(\vp{H}\con e; a,b,c,x, \al,\be)\\[2pt]
    \Q(\vp{H}\pcon e; a,b,c,x,  \al,\be)\\[2pt]
    \Q(\vp{H}\mba e; a,b,c,x, \al,\be)
\end{bmatrix}.\end{equation}
\end{theorem}
\begin{proof}
Let $\pack{G}$ be a packaged arrow presentation obtained from $\vp{G}$ by partitioning the boundary components, and
$\pack{H}$ be the packaged arrow presentation obtained from  $\vp{H}$ such that its boundary partition contains precisely one block.
Using Equation~\eqref{eq:brqpol} and Theorem~\ref{thm.main}, we can obtain an expression for $\Q(\vp{G}\ot_{\bphi} \vp{H})$ in terms of $Q(\pack{G})$ with variables arising from $Q(\pack{H}\ba e)$, $Q(\pack{H}\con e)$, $Q(\pack{H}\pcon e)$, $Q(\pack{H}\mba e)$ and $Q(\pack{H}\mcon e)$. Due to our choice of $\pack{H}$, $Q(\pack{H}\con e)=Q(\pack{H}\mcon e)$ and as observed in Technique~\ref{r.phi0}, $\phi_{(\pl,2)}=0$. So we simplify the matrix equation~\eqref{eq.mainmat1} accordingly, relabel $\phi_{(\pl,1)}$ as $\phi_{\pl}$, reorder the rows where appropriate, and use Equation~\eqref{eq:brqpol} to obtain the result. 
\end{proof}

The following is immediate from setting $c=x=0$ in Theorem~\ref{thm.mainBR}.

\begin{corollary}\label{cor.hmthm2ainBR}
Let $\vp{G}$ and $\vp{H}$ be the vertex partitioned arrow presentation obtained from $\vp{G}$ and $\vp{H}$, respectively, by placing each vertex in its own block of the partition. 
 Let $e$ be a fixed edge of $\vp{H}$ and for each edge $f$ in $\vp{G}$ let $\varphi_f$ be a coupling of $f$ and $e$, and $\bphi = \{\varphi_f\}_{f\in E(\vp{G})}$.
Then
\begin{equation}\label{eq:main4a}
\Z(\vp{G}\ot_{\bphi} \vp{H}; a,b, \al,\be) = \Q(\vp{G}; \phi_{(\eq,2)},\phi_{\pl},\phi_{\x},\phi_{(\eq,1)}, \al,\be)
\end{equation}
where the $\phi$'s arise from the unique solution to
\begin{equation}\label{eq.mainmatBR2} \alpha\beta \left[ \begin{array}{cccc}
   \alpha\beta & 1 & 1 & \alpha \\
    1 & \alpha\gamma & 1 & 1   \\
    1 & 1 & \alpha & 1  \\
    \alpha & 1 & 1 & \alpha 
\end{array}\right] 
\cdot \begin{bmatrix}
    \phi_{(\eq,2)} \\[2pt]
    \phi_{\pl} \\[2pt]
    \phi_{\x}\\[2pt]
    \phi_{(\eq,1)} 
\end{bmatrix}
=\begin{bmatrix}
    \Z(\vp{H}\ba e; a,b ,\al,\be)\\[2pt]
    \Z(\vp{H}\con e; a,b, \al,\be)\\[2pt]
    \Z(\vp{H}\pcon e; a,b, \al,\be)\\[2pt]
    \Z(\vp{H}\mba e; a,b, \al,\be)\\[2pt]
\end{bmatrix}.\end{equation}
\end{corollary}

\medskip

We will now indicate how to recover Huggett and Moffatt's tensor product formula~\cite[Theorem~4.3]{MR2854787} (which is stated as Corollary~\ref{cor.hmthm2} below) from Corollary~\ref{thm.mainBR}. 
Let $\ar{G}$ be an arrow presentation, and $\mathbf{b}=\{b_e: e\in E(\ar{G})\}$ be a set of indeterminates indexed by the edges of $\ar{G}$. 
The \emph{multivariate Bollob\'as--Riordan polynomial}~\cite{zbMATH05216495} is defined by
\[ \mvZ(\ar{G};a,\mathbf{b},c)= \sum_{A\subseteq E(\ar{G})} a^{k(A)} \Big( \prod_{e\in A} b_e \Big)c^{b(A)},\]
where $k(A)$ is the number of connected components of (the ribbon graph corresponding to) $\ar{G}\ba (E(\ar{G})-A)$ and $b(A)$ its number of boundary components. 

For convenience write $\mvW(\pack{G})$ for $Q(\pack{G}; \{(1,0,0,0,b_e)\}_{e\in E(\ar{G})}, c,a,1)$, where $\pack{G}$ is a packaged arrow presentation. Proposition~\ref{p.qmvdef} gives that 
\[
\mvW(\pack{G})
=\begin{cases}
\mvW(\pack{G}\ba e)+b_e\mvW(\pack{G}\mcon e) &
\text{ for any edge $e$,}
\\ a^{|\V|} c^{|V(\ar{G})|}
& \text{if $\ar{G}$ is edgeless.} 
\end{cases}
\]
By making use of Equation~\eqref{eq.mvqstate}, 
\begin{equation}\label{eq.hdyv}
\mvW(\pack{G}) = \sum_{A\subseteq E(\pack{G})} \big(\prod_{e\in A} b_e\big) c^{|V( \ar{G}\mcon A \ba (E(\pack{G})-A))| } a^{|\V(\pack{G}\mcon A \ba (E(\pack{G})-A))|}.\end{equation}
Thus if $\ar{G}$ is an arrow presentation and $\pack{G} = (\ar{G},\V,\B)$ where $\V$ places each vertex in its own block, and $\B$ is any boundary partition, we see that 
 $|V( \ar{G}\mcon A \ba (E(\pack{G})-A))| = b(\ar{G}\ba (E(\pack{G})-A)))$ since contraction does not change the number of boundary components, and it is readily seen that   
$|\V(\pack{G}\mcon A \ba (E(\pack{G})-A))| = k(\ar{G}\ba (E(\pack{G})-A))$. 
Thus Equation~\eqref{eq.hdyv} and the definition of $\mvW(\pack{G})$ gives that 
\begin{equation}\label{eq.zplane}
    \mvZ(\ar{G};a,\mathbf{b},c) =
\mvQ(\pack{G}; \{(1,0,0,0,b_e)\}_{e\in E(\ar{G})}, c,a,1)
.
\end{equation}
With this observation we can now recover~\cite[Theorem~4.3]{MR2854787} as a corollary of Theorem~\ref{thm.mainmv} as follows.
\begin{corollary}\label{cor.hmthm2}
Let $\ar{G}$ be an arrow presentation that describes an orientable ribbon graph, and let $\ar{H}$ be an arrow presentation that describes a plane ribbon graph.
Let $e$ be a fixed edge of $\ar{H}$ and for each edge $f$ in $\ar{G}$ let $\varphi_f$ be a coupling of $f$ and $e$, and $\bphi = \{\varphi_f\}_{f\in E(\ar{G})}$.
Then
\begin{equation}\label{eq:main4b}
\mvZ(\ar{G}\ot_{\bphi} \ar{H}; a,b,c) = (ac)^{-e(\ar{G)}} \Big(\prod_{e\in E(\ar{G})} g_e \Big) \mvZ(\ar{G}; a, \{f_e/g_e\}_{e\in E(\ar{G})}, c),
\end{equation}
where the $f_i$'s and $g_i$'s arise as the unique solutions to
\begin{align*}
acg_e + f_e & = \mvZ( \ar{H} \ba e; a,b,c)\\
g_e + cf_e & = \mvZ( \ar{H} \ba e; a,b,c).
\end{align*}
\end{corollary}
\begin{proof}
Let $\pack{G}$ be a packaged arrow presentation obtained from $\ar{G}$, and $\pack{H}$  from $\ar{H}$, in which  each vertex and each boundary component is in a block of size one. Using Theorem~\ref{thm.mainmv} and Equation~\eqref{eq.zplane}, we can obtain an expression for $\mvZ(\ar{G}\ot_{\bphi}\ar{H})$ in terms of $\mvQ(\pack{G})$ with variables arising from $\mvQ(\pack{H}\ba e)$, $\mvQ(\pack{H}\con e)$, $\mvQ(\pack{H}\pcon e)$, $\mvQ(\pack{H}\mba e)$ and $\mvQ(\pack{H}\mcon e)$. 
As $\ar{H}$ is plane and the partitions of $\pack{H}$ only contain blocks of size one, applying the methods in Technique~\ref{r.phi0} we can deduce that $\phi_{\x}=0$, $\phi_{(\eq,1)}=0$ and $\phi_{(\pl,1)}=0$.
We simplify matrix equation~\eqref{eq.mainmat1} accordingly, rename $\phi_{(=,2)}$ as $g_e$ and $\phi_{(\pl,2)}$ as $f_e$ and use Equation~\eqref{eq.zplane} to obtain the result. 
\end{proof}

We highlight, in particular that~\cite[Corollary~4.3]{MR2854787} applies the above result to find a tensor product formula for the Bollob\'as--Riordan polynomial (with the same orientability and planarity conditions). Thus Theorem~\ref{thm.mainBR} is strictly stronger than Huggett and Moffatt's result for the Bollob\'as--Riordan polynomial.

\subsection{The topological transition  and ribbon graph polynomials}

 The \emph{topological transition polynomial},  introduced in \cite{zbMATH06024136,zbMATH05937085}, is a multivariate polynomial of ribbon graphs,or equivalently arrow presentations. It contains both the 2-variable version of Bollob\'as and Riordan's ribbon graph polynomial (described below) and the Penrose polynomial~\cite{MR1428870,MR2994409} as specializations, and is intimately related to  Jaeger's transition polynomial~\cite{MR1096990} and the 
 generalized transition polynomial of 
 \cite{MR1980048}.

Let $\ar{G}$ be an arrow presentation, let $(\textbf{a},\textbf{b},\textbf{c})=\{(a_e,b_e,c_e)\}_{e\in E(\ar{G})}$ be an indexed family  of triples of indeterminates, and $t$  be another indeterminate. Then, as shown in~\cite{zbMATH06024136} the
 \emph{topological transition polynomial}, $\tQ(\ar{G};(\textbf{a},\textbf{b},\textbf{c}),t)$, is defined by the recursion relation
\begin{equation}\label{newdctopt}
\tQ(\ar{G})=
\begin{cases}
a_e \, \tQ(\ar{G}\con e)+b_e \,\tQ( \ar{G}\ba e)+c_e\, \tQ(\ar{G}\pcon e), &\text{for any edge $e$;}
\\
t^{p} &\text{if $\ar{G}$ is edgeless with $p$ vertices};
\end{cases}
\end{equation}
where we have written $\tQ(\ar{G})$ for $\tQ(\ar{G};(\textbf{a},\textbf{b},\textbf{c}),t)$ in the recursion relation. (Note the different order of the deletion and contraction terms in Equations~\eqref{p.qmvdef} and~\eqref{newdctopt}.)

Observe that if  $\pack{G}$ is any packaged arrow presentation obtained by partitioning the vertices and boundary components of $\ar{G}$ then
\begin{equation}\label{eq.dob}
\tQ(\ar{G};(\textbf{a},\textbf{b},\textbf{c}),t)  = 
\mvQ(\pack{G}; \bw, t,1,1), 
\end{equation}
where $\bw=\{(b_e,a_e,c_e,0,0) :e\in E(\pack{G}) \}$.

The following result extends \cite[Theorem 28]{moffatt2023} and \cite[Theorem~6.1]{EllisMoff}.
\begin{theorem}\label{th.dajkl}
    Let $\ar{G}$, $\{\ar{H}^{(f)}\}_{f\in E(\ar{G})}$,  $e^{(f)}$ and $\bphi$ be as in Definition~\ref{d.argenten}.
Then 
\[\tQ(\ar{G}\ot_{\bphi} \{\ar{H}^{(f)}\}_{f\in E(\ar{G})}; (\textbf{a},\textbf{b},\textbf{c}),t)
= \tQ(\ar{G};( \boldsymbol{\phi}_{\pl} ,\boldsymbol{\phi}_{\eq}, \boldsymbol{\phi}_{\x}), t),\]
where 
$( \boldsymbol{\phi}_{\pl} ,\boldsymbol{\phi}_{\eq}, \boldsymbol{\phi}_{\x}) = \{  ( \phi^{(f)}_{\pl} ,\phi^{(f)}_{\eq}, \phi^{(f)}_{\x}) : f\in E(\ar{G}) ) \}$,  
and for each  $f$ we have that $\phi^{(f)}_{\eq}$, $\phi^{(f)}_{\pl}$ and $\phi^{(f)}_{\x}$  are given by 
  \begin{equation}
    \left[ \begin{array}{ccc}
   t^2& t  & t  \\
    t & t^2 & t   \\
    t & t & t^2  \\
\end{array}\right] 
\cdot \begin{bmatrix}
    \phi^{(f)}_{\pl} \\[2pt]
    \phi^{(f)}_{\eq} \\[2pt]
    \phi^{(f)}_{\x}
\end{bmatrix}
=\begin{bmatrix}
    \tQ(\ar{H}^{(f)}\con e; (\textbf{a},\textbf{b},\textbf{c}),t)\\[2pt]
   \tQ(\ar{H}^{(f)}\ba e; (\textbf{a},\textbf{b},\textbf{c}) ,t)\\[2pt]
    \tQ(\ar{H}^{(f)}\pcon e; (\textbf{a},\textbf{b},\textbf{c}) , t)
\end{bmatrix}.
\end{equation}
\end{theorem}
\begin{proof}
Let $\pack{G}$ and each $\pack{H}^{(f)} $ be the packaged arrow presentations consisting of $\ar{G}$ and the $\ar{H}^{(f)} $ with all the partitions containing precisely one block. Then the result follows from Corollary~\ref{cor.fulltens} upon noting that the $\phi^{(f)}_{(\eq,2)}$ and $\phi^{(f)}_{(\pl,2)}$ are zero, and that $\ar{H}^{(f)}\ba e = \ar{H}^{(f)}\mba e$ and $\ar{H}^{(f)}\con e = \ar{H}^{(f)}\mcon e$.
\end{proof}

\begin{remark}\label{rem.hugg}
We note that as was detailed for the proof of~\cite[Corollary~6.3]{EllisMoff}, if $\ar{G}$ and $\ar{H}$ represent orientable ribbon graphs and if we set each $c_e=0$ throughout Theorem~\ref{th.dajkl}, then $\phi_{\x}=0$ and so the  $\Q(\ar{H}\pcon e)$ need not be considered. This results in~\cite[Theorem~4.3]{MR2854787}, restated here as Corollary~\ref{cor.hmthm2}.
\end{remark}

\subsection{The Tutte polynomial}
We conclude by indicating how to recover Brylawski's original formula~\eqref{btf} from Theorem~\ref{thm.main}. 
Let $G=(V,E)$ be a graph, Then set
\[ \dZ(G;a,b,c) = \sum_{A\subseteq E} a^{k(A)}b^{|A|}c^{|E-A|},  \]
where $k(A)$ is the number of connected components of the spanning subgraphs $(V,A)$ of $G$.
This polynomial satisfies the deletion-contraction formula
\[
\dZ(G) =
\begin{cases} c \dZ(G\ba e) +b \dZ(G\con e) &\text{for any edge $e$ of $G$,}
\\
a^{|V|} & \text{when $G$ is edgeless with $|V|$ vertices.}
\end{cases}
\]
Here $G\ba e$ and $G\con e$ are the usual graph deletion and contraction operations, and $\dZ(G)$  denotes $\dZ(G;a,b,c)$.

Next let $\ar{G}$ be any arrow presentation that represents any embedding of $G$. (So the vertices and edges of $G$ correspond to the vertices and edges of $\ar{G}$, and for each edge in $G$ place $e$-labelled arrows on circles of the arrow presentation corresponding to the ends of that edge.) Let $\V$ be the vertex partition of $\ar{G}$ that places each vertex in its own block, and let $\B$ be any boundary partition of $\ar{G}$. Finally let $\pack{G}=(\ar{G},\V,\B)$. For a graph $H$, construct some $\pack{H}$ similarly.
It is easily seen that $\pack{G}\otimes_{\bphi} \pack{H}$ represents a tensor product of graphs $G\otimes H$ analogously (with the tensor product acting on the same edge).

With $G$ and $\ar{G}$ as above, we have
\[
\dZ(G;a,b,c) = Q(\pack{G}; c,0,0,0,b,1,a,1).
\]
This follows from the observation that the vertex partition of $\pack{G} \mcon e$ (and also of $\pack{G} \mba e$) records when the two vertices of $G$ are merged under contraction.

An application of Theorem~\ref{thm.main}, after some rewriting of the system of equations,  gives that
\[ \dZ(G\otimes H;a,b,1) = \dZ(G;a, f,g),  \]
where $f$ and $g$ arise from the unique solution to
\begin{align*}
    af + a^2g &= \dZ(H \ba e; a,b,1),\\
    af + ag &= \dZ(H\con e;a,b,1).
\end{align*}

This is exactly the result displayed in \cite[Corollary~3.7]{MR2854787}. The proof of Corollary~3.9 of that reference then shows how to rewrite the above formula to obtain Brylawski's tensor product formula~\eqref{btf}.

\section*{Declarations}

For the purpose of open access, the authors have applied a Creative Commons
Attribution (CC BY) licence to any Author Accepted Manuscript version arising. No underlying data is associated with this article.

\section*{Open access and research data}
For the purpose of open access, the authors have applied a Creative Commons Attribution
(CC BY) licence to any Author Accepted Manuscript version arising. No underlying data is
associated with this article. There are no conflicts of interest.

\bibliographystyle{abbrv}
\bibliography{tensor.bib}

\end{document}